\documentclass[12pt]{amsart}
\usepackage{pgf,tikz,pgfplots} 
\pgfplotsset{compat=1.12}

\usepackage{mathrsfs} 

\usetikzlibrary{arrows} 

\usetikzlibrary{calc,fadings,decorations.pathreplacing} 

\usepackage{tkz-euclide} 
\usepackage{float} 

\usepackage[colorlinks=true,urlcolor=blue, citecolor=red,linkcolor=blue,linktocpage,pdfpagelabels, bookmarksnumbered,bookmarksopen]{hyperref}
\usepackage[capitalize]{cleveref} 
\usepackage[hyperpageref]{backref}
\usepackage{xcolor}
\usepackage{amsthm} 
\usepackage{latexsym,amsmath,amssymb}
\usepackage{accents}
\usepackage{a4wide}
\usepackage{soul}
\usepackage{enumitem}
\usepackage{mathtools} 
\usepackage{xparse} 
\usepackage{enumitem}

\title[Singular Set of Minimizing Harmonic Maps]{On the size of the singular set of minimizing harmonic maps}

\date{\today}

\author{Katarzyna Mazowiecka}
\address[Katarzyna Mazowiecka]{
Institute of Mathematics,%
University of Warsaw,
Banacha 2,
02-097 Warszawa, Poland
\newline
\&
Universit\'e catholique de Louvain, Institut de Recherche en Math\'ematique et Physique, Chemin du Cyclotron 2 bte L7.01.02, 1348 Louvain-la-Neuve, Belgium}
\email{katarzyna.mazowiecka@uclouvain.be}

\author{Micha\l{} Mi\'{s}kiewicz}
\address[Micha\l{} Mi\'{s}kiewicz]{
Institute of Mathematics,
Polish Academy of Sciences,
{\'S}niadeckich 8, 
00-656 Warszawa, Poland
\newline
\&
Institute of Mathematics, 
University of Warsaw,
Banacha 2,
02-097 Warszawa, Poland}
\email{m.miskiewicz@mimuw.edu.pl}

\author{Armin Schikorra}
\address[Armin Schikorra]{Department of Mathematics,
University of Pittsburgh,
301 Thackeray Hall,
Pittsburgh, PA 15260, USA}
\email{armin@pitt.edu}

\definecolor{chameleongreen}{RGB}{98,189,25} 
\definecolor{commentgreen}{RGB}{40,130,10} 
\definecolor{bluish}{rgb}{0,0,0.8}
\definecolor{redish}{rgb}{1,0,0}
\definecolor{indigo}{rgb}{0.29, 0.0, 0.51}

plus 6pt minus 12pt
\newcommand{\sing}{\operatorname{sing}}
\newcommand{\id}{\operatorname{id}}

\def\dif{\varsigma}		
\def\eps{\varepsilon}
\def\vp{\varphi}

\def\pl{\partial}

\def\B{{B}}

\def\d{{d}}

\def\N{{\mathbb N}}
\def\n{{\mathcal N}}
\def\H{{\mathcal H}}

\def\S{{\mathbb S}}

\newcommand{\cH}{\mathcal H}
\newcommand{\cN}{\mathcal N}

\newcommand{\cB}{\mathcal B}
\newcommand{\cA}{\mathcal A}
\newcommand{\cU}{\mathcal U}
\newcommand{\cW}{\mathcal W}
\newcommand{\cQ}{\mathcal Q}
\newcommand{\bad}{{\mathsf{Bad}}}
\newcommand{\good}{{\mathsf{Good}}}

\newtheorem{theorem}{Theorem}
\newtheorem{lemma}[theorem]{Lemma}
\newtheorem{corollary}[theorem]{Corollary}
\newtheorem{proposition}[theorem]{Proposition}

\newtheorem{remark}[theorem]{Remark}
\newtheorem{definition}[theorem]{Definition}
\newtheorem{example}[theorem]{Example}

\newtheorem{conjecture}[theorem]{Conjecture}

\newcommand{\loc}{\mathrm{loc}}

\def\dist{{\rm dist\,}}

\def\Lip{{\rm Lip}}

\def\supp{{\rm supp\,}}
\def\loc{{\rm loc}}

\newcommand{\dd}{\,\mathrm{d}}
\newcommand{\dx}{\dd x}
\newcommand{\dy}{\dd y}
\renewcommand{\dh}{\dd \mathcal{H}^{2}}
\newcommand{\dhn}{\dd \mathcal{H}^{n-1}}
\newcommand{\R}{\mathbb{R}}

\newcommand{\brac}[1]{\left (#1 \right )}
\newcommand{\abs}[1]{\left |#1 \right |}

\newcommand{\norm}[1]{\left\|{#1}\right\|}

\newcommand{\barint}{
\rule[.036in]{.12in}{.009in}\kern-.16in \displaystyle\int }

\newcommand{\barcal}{\mbox{$ \rule[.036in]{.11in}{.007in}\kern-.128in\int $}}

\def\mvint_#1{\mathchoice
          {\mathop{\vrule width 6pt height 3 pt depth -2.5pt
                  \kern -8pt \intop}\nolimits_{\kern -3pt #1}}%
          {\mathop{\vrule width 5pt height 3 pt depth -2.6pt
                  \kern -6pt \intop}\nolimits_{#1}}%
          {\mathop{\vrule width 5pt height 3 pt depth -2.6pt
                  \kern -6pt \intop}\nolimits_{#1}}%
          {\mathop{\vrule width 5pt height 3 pt depth -2.6pt
                  \kern -6pt \intop}\nolimits_{#1}}}

\numberwithin{theorem}{section} \numberwithin{equation}{section}

\newcommand{\lap}{\Delta }
\newcommand{\aleq}{\precsim}
\newcommand{\ageq}{\succsim}
\newcommand{\aeq}{\approx}

\let\latexchi\chi
\makeatletter
\renewcommand\chi{\@ifnextchar_\sub@chi\latexchi}
\newcommand{\sub@chi}[2]{
	\@ifnextchar^{\subsup@chi{#2}}{\latexchi^{}_{#2}}%
}
\newcommand{\subsup@chi}[3]{
	\latexchi_{#1}^{#3}%
}
\makeatother

\tikzset{>=stealth}

\begin{document}

\begin{abstract}
We consider minimizing harmonic maps $u$ from $\Omega \subset \R^n$ into a closed Riemannian manifold $\n$ and prove:
\begin{enumerate}
 \item an extension to $n \geq 4$ of Almgren and Lieb's linear law. That is, if the fundamental group of the target manifold $\n$ is finite, we have
 \[
 \H^{n-3}(\sing u) \le C \int_{\partial \Omega} |\nabla_T u|^{n-1} \dhn;
 \] 
 \item an extension of Hardt and Lin's stability theorem. Namely, assuming that the target manifold is $\n=\S^2$ we obtain that the singular set of $u$ is stable under small $W^{1,n-1}$-perturbations of the boundary data. 
\end{enumerate}
In dimension $n=3$ both results are shown to hold with weaker hypotheses, i.e., only assuming that the trace of our map lies in the fractional space $W^{s,p}$ with $s \in (\frac{1}{2},1]$ and $p \in [2,\infty)$ satisfying $sp \geq 2$. We also discuss sharpness.
\end{abstract}

\sloppy

\maketitle
{
 \tiny 
 \setcounter{tocdepth}{1}
\tableofcontents}
\sloppy

\section{Introduction}

A \emph{minimizing harmonic map} from an $n$-dimensional domain $\Omega \subseteq \R^n$ into $\n$ is a map $u \in W^{1,2}(\Omega, \n)$ that minimizes the Dirichlet energy
\[
E(u)\coloneqq\int_{\Omega} |\nabla u|^2 \dx
\]
among all maps in $W^{1,2}(\Omega, \n)$ with the same boundary data $\vp\colon  \partial \Omega \to \n$. Here,
the target manifold $\n$ is a smooth, closed (i.e., compact and without boundary) Riemannian manifold isometrically embedded in $\R^d$. The Sobolev space $W^{1,2}(\Omega, \n)$ is defined as
\[
 W^{1,2}(\Omega,\mathcal{N}) \coloneqq \left \{u \in W^{1,2}(\Omega,\R^d)\colon \quad u(x) \in \mathcal{N} \mbox{ almost everywhere} \right \}.
\]
In such a geometrical setup, one might suspect that minimizing harmonic maps are always smooth. However, this holds only in the case of geodesics ($n=1$) and in the conformal case ($n=2$), see Morrey's classical result \cite{Morrey}. 

In contrast, in dimensions $n \ge 3$ even continuity cannot be guaranteed.
Minimizers of the Dirichlet energy satisfy the Euler--Lagrange system of equations
\[
-\Delta u  =  A(u)(\nabla u,\nabla u) \quad \mbox{in $\Omega$},\\
\]
where $A$ is the second fundamental form of the isometric embedding $\n \subset \R^d$. In the special case when $\n = \S^{d-1}$, this system takes the form
\[
-\Delta u  =  |\nabla u|^2 u \quad \mbox{in $\Omega$.}\\
\]
For $n \geq 3$, \emph{critical points}, i.e., solutions to the Euler--Lagrange equations might be everywhere discontinuous, see Rivi\`{e}re's seminal \cite{R95}.
\emph{Minimizers} enjoy better regularity, but discontinuities may still appear.
The simplest example are obstructions of topological nature: by Hopf--Brouwer theorem we know that if the topological degree of a boundary map $\varphi\colon\partial B^3 \to\S^2$ is not zero, there is no continuous extension $u\colon B^3\to \S^2$.  
The continuity of minimizers may also fail without such a topological obstruction. Hardt and Lin in \cite{HL1986} constructed a boundary datum $\varphi\in C^\infty(\partial B^3,\S^2)$ with $\deg \varphi =0$ for which all minimizers\footnote{Minimizing harmonic maps into $\S^2$ may be non-unique, consider $u\colon \B^n \to \S^2$ and $-u\colon \B^n \to \S^2$ which have the same energy and may have the same boundary datum.} must have singularities. 

Consequently, in dimensions $n \geq 3$, the analysis of singularities of minimizing harmonic maps is an intriguing theory. The singular set $\sing u$ of a mapping $u\in W^{1,2}(\Omega,\n)$ is defined as the complement of its regular set 
\[
 \sing u = \Omega \setminus \{x \in \Omega \colon u \text{ is smooth on some neighborhood of } x \}.
\]
For harmonic maps the analysis of the singular set started with the fundamental work of Schoen and Uhlenbeck \cite{SU1}. They showed that one can estimate the Hausdorff dimension of the singular set of minimizers. Namely,
\[
\dim_{\mathcal{H}}(\sing u) \le n-3
\]
for any minimizing harmonic map from an $n$-dimensional domain $\Omega \subset \R^n$ into an arbitrary closed Riemannian manifold $\mathcal{N}$ \cite[Theorem II]{SU1}. They also prove in the special case $n=3$ that the $\mathcal H^0$ measure (the counting measure) of the singular set is locally finite. The latter result was recently generalized for any $n\ge 3$ by Naber and Vatorta, see \Cref{th:NabVal-general-meas-bound} below.

A basic example of a singular minimizing harmonic map is given by the ``hedgehog" map
\[
 u_0 \colon B^n \to \S^{n-1}, \quad u_0(x) = \frac{x}{|x|}.
\]
The minimality of $u_0$ was proved for $n=3$ by Brezis, Coron, and Lieb \cite[Theorem 7.1]{BCL1986} and for all $n\ge 3$ by Lin \cite{Lin1987}, see also Coron and Gulliver \cite[Theorem 1.1]{CoronGulliver}. This example shows the optimality of the result of Schoen and Uhlenbeck. Indeed, let $\Psi(x',x'')=\frac{x'}{|x'|}$ for $x'\in\R^3$ and $x''\in \R^{n-3}$. Then $\Psi\in W^{1,2}(B^n,\S^2)$ is a minimizing harmonic map with $\sing \Psi=\brac{\{0\}\times\R^{n-3}} \cap B^n$.

Additionally, the map $u_0\colon B^3 \to \S^2$ is the unique minimizer for its boundary mapping $\id\colon \partial B^3 \to \S^2$ \cite[Theorem 7.1]{BCL1986}. In general there is no uniqueness of minimizing harmonic maps for a given boundary datum. For example, in \cite{HKL-thevariety}, Hardt, Kinderlehrer,  and Lin construct a boundary map $\varphi\colon \partial B^3 \to \S^2$ for which there exist countably many minimizing harmonic mappings  or, in \cite{HL1989} Hardt and Lin construct a boundary datum which admits at least two minimizers: one of which is smooth and the other one is singular.

More quantitative results were obtained in the late 80's of the last century --- for $n=3$ and $\n=\S^2$, Almgren and Lieb \cite{AlmgrenLieb1988} showed that one can estimate the number of singularities of minimizing harmonic maps in terms of their trace maps, which became to be known as \emph{Almgren and Lieb's linear law}:
\begin{equation}\label{eq:ALestimate}
\#\{\text{singularities of }u\} \leq C(\Omega) \int_{\partial \Omega} |\nabla_T u|^2\dh.
\end{equation}

Moreover, in \cite{HL1989} Hardt and Lin showed that for a \emph{unique}\footnote{By considering restrictions to a proper subset one may obtain uniqueness. In fact, the set of boundary data $\varphi\colon \S^2 \to \S^2$ admitting unique minimizers is dense in $W^{1,2}(\partial B^3,\S^2)$, see \cite[Theorem 4.1]{AlmgrenLieb1988}.} harmonic minimizer $v\in W^{1,2}(\Omega,\S^2)$ the number of singularities remains the same for all minimizers whose trace is close in the Lipschitz-norm to the trace of $v$. This is known as \emph{Hardt and Lin's stability theorem}.

The natural questions arise: what are the minimal regularity assumptions of the boundary map in order to obtain a linear estimate similar to \eqref{eq:ALestimate}? Is it possible to obtain a similar bound for other target manifolds? Is it possible to extend this result to higher dimensions of the domain?

The first named author and Strzelecki proved in \cite{MazowieckaStrzelecki} that singularities of minimizing harmonic maps from a 3-dimensional domain to $\S^2$ are not stable under $W^{1,p}$-perturbations of the boundary data, for $p<2$. Thus, combining this with the stability result of Hardt and Lin one can wonder whether the stability theorem of Hardt and Lin cannot be strengthened to $W^{1,2}$-perturbations of the boundary.

To introduce our main results, let us begin with definitions of the spaces involved. For $s\in (0,1)$ the fractional Sobolev space on the boundary of a smooth enough set $\Omega$ is defined as 
\[
 W^{s,p}(\partial \Omega) \coloneqq \{f\in L^p(\partial \Omega) \colon [f]_{W^{s,p}(\partial \Omega)}<\infty\},
\]
where the Gagliardo seminorm is given by 
\[
 [f]_{W^{s,p}(\partial \Omega)} \coloneqq \brac{\int_{\partial \Omega} \int_{\partial \Omega} \frac{|f(x) - f(y)|^p}{|x-y|^{n-1+sp}}\dd x \dd y}^\frac{1}{p}.
\]
We note here that for $s=1$ we will sometimes write 
\[
 [\varphi]_{W^{1,p}(\partial \Omega)} \coloneqq \brac{\int_{\partial \Omega}|\nabla \varphi|^p \dd \cH^{n-1}}^\frac1p. 
\]

We recall that by Gagliardo's trace theorem \cite{Gagliardo} the fractional Sobolev spaces are naturally related to boundary value problems: For any function $f\in W^{1,2}(\Omega)$ we have $\phi \in W^{\frac12,2}(\partial \Omega)$, where $\phi \coloneqq f\big\rvert_{\partial \Omega}$ (in the trace sense) with $\|\phi\|_{W^{\frac12,2}(\partial \Omega)}\le C\|f\|_{W^{1,2}(\Omega)}$. Conversely, given any $\gamma\in W^{\frac12,2}(\partial \Omega)$ there exists an extension $g\in W^{1,2}(\Omega)$,  $g\big\rvert_{\partial \Omega} = \gamma$ (in the trace sense) for which $\|g\|_{W^{1,2}(\Omega)}\le C \|\gamma\|_{W^{\frac12,2}(\partial \Omega)}$. 

However, we will be working with maps which have values in a manifold $\n$ (with finite fundamental group). Gagliardo's theorem \cite{Gagliardo} extends to the vectorial case, assuring us that for any map $u\in W^{1,2}(\Omega,\n)$ its trace belongs to the fractional space $W^{\frac12,2}(\partial \Omega,\n)$, but it gives us no information about the range of an extension of a $W^{\frac12,2}(\partial \Omega, \n)$ map. In fact, in our case there are boundary maps $\vp \in W^{\frac12,2}(\partial \Omega, \n)$, which admit no manifold-valued extension $u\in W^{1,2}(\Omega,\n)$ with $u\big\rvert_{\partial \Omega}=\vp $. See discussion in \cref{s:extension}. A reader not familiar with algebraic topology is encouraged to think of $\n$ as of $\S^2$. Let us remark that our  condition on the target manifold applies also to $\n=\mathbb{RP}^2$ as a target manifold, which is the most relevant case for the theory of liquid cristals\footnote{Harmonic maps are \emph{highly simplified} models of liquid crystals, see \cite{HardtKinderlehrerLin1986, ERICKSEN1976233}}.

Our main results are the following.
\begin{theorem}\label{th:stabilitynD}
Let $\Omega \subset \R^n$ be a bounded smooth domain. Assume that for some boundary map $\varphi \in W^{1,n-1}(\partial \Omega,\S^2)$ there is a unique minimizer $u \in W^{1,2}(\Omega,\S^2)$. 

Then, for each $\varepsilon > 0$, there is a $\delta > 0$ such that
\begin{equation}\label{eq:stabilityn}
\| \psi - \varphi \|_{W^{1,n-1}(\partial \Omega)} < \delta
\quad \Longrightarrow \quad 
d_W \left( \H^{n-3} \llcorner \sing v, \H^{n-3} \llcorner \sing u \right) < \varepsilon
\end{equation}
for any minimizer $v$ with boundary datum $\psi$. 

Here, $d_W$ is the $1$-Wasserstein distance, see \eqref{eq:wasserstein}, which in particular satisfies
\[
 \left | \H^{n-3} (\sing u) - \H^{n-3} (\sing v) \right | \le d_W \left( \H^{n-3} \llcorner \sing u, \H^{n-3} \llcorner \sing v \right) .
\]
\end{theorem}

\begin{theorem}\label{th:intro:al}
Let $\pi_1(\n)$ be finite and let $\Omega\subset \R^n$ be a bounded smooth domain. Assume that $u\in W^{1,2}(\Omega,\n)$ is a minimizing harmonic map with $u\rvert_{\partial \Omega} = \varphi$ (in the trace sense). Let us also assume that $sp=n-1$ and $s\in(\frac12,1]$. Then 
 \[
  \H^{n-3}(\sing u) \le C(n,\Omega,s) [\varphi]^p_{W^{s,p}(\partial \Omega)}.
 \]
\end{theorem}

In the case $n=3$ \Cref{th:intro:al} is optimal, in the sense it cannot be improved to the case when $sp<2$, see \Cref{la:sharpll1s}.

Also note that the theorem includes the limiting case $s = 1$, in which the inequality closely resembles Almgren and Lieb's linear law for $n=3$: 
\[
  \H^{n-3}(\sing u) \le C(n,\Omega) \int_{\pl \Omega} |\nabla \vp|^{n-1} \dd \cH^{n-1}.
\]

A very challenging question is whether the $sp=n-1$ condition in Theorem~\ref{th:intro:al} can be improved for $n>3$. Technically, this condition controls each singularity close to the boundary of $\Omega$ --- but only $(n-3)$-dimensional singularities appear in the estimate. So one might be tempted to believe that a bound with $sp=2$ is sufficient in \eqref{eq:mainestimate} for any dimension~$n$. On the other hand, one might also be able to analyze each stratum of the singular set via a different Sobolev norm along the boundary, see a discussion in \Cref{ss:optimalboundarynorm}. 

As for Theorem~\ref{th:stabilitynD} our argument relies crucially on the classification of tangent maps. For the case when the target manifold is $\S^2$ such a classification was obtained by Brezis, Coron, and Lieb in \cite[Theorem 1.2]{BCL1986}. 
It is also known for $\S^3$ in the target, see Nakajima's \cite{Nakajima} and also \cite{LinWang2006}. For $\n=\S^3$ it was proved by Schoen and Uhlenbeck that the estimate on the Hausdorff dimension of the singular set of any minimizing harmonic map $u\in W^{1,2}(\Omega,\S^3)$ may be improved to $\dim_{\mathcal{H}} \sing u \leq n-4$ \cite{SU84}. It is possible to extend Theorems \ref{th:stabilitynD} and \ref{th:intro:al} to consider $\mathcal{H}^{n-4} \llcorner \sing u$, following the same argument. Apart for the cases $\n=\S^2$ and $\n=\S^3$ we do not know whether \Cref{th:stabilitynD} can be extended to any other target manifold.

For results in the higher dimension $n>3$ the main new ingredient is the interior analysis of the singular set of minimizing harmonic maps by Naber and Valtorta \cite{NabVal17}. The following measure bound allows us to generalize Almgren and Lieb's linear law to higher dimensions. 

\begin{theorem}[{{\cite[Theorem~1.6]{NabVal17}}}]
\label{th:NabVal-general-meas-bound}
For $n \geq 3$, let $u \colon B_{2r}(y) \to \n$ be energy minimizing and 
\[
r^{2-n} \int_{B_{2r}(y)} |\nabla u|^2 \dx \le \Lambda.
\] 
Then there exists a constant $C=C(n,\n,\Lambda) > 0$ such that \[\H^{n-3}(\sing u \cap B_r(y)) \le C r^{n-3}.\]
\end{theorem}
Let us also mention here that results on estimates of the singular sets in the case of quasiconvex functionals were obtained earlier in \cite{M03,KM07}.

In order to prove the stability theorem (Theorem \ref{th:stabilitynD}), one needs to refine the above measure estimates. In addition to the methods developed in \cite{NabVal17} (discussed in Section \ref{sec:NabVal}), these refinements also involve the results of the second named author, concerned with the structure of the singular set:  
\begin{theorem}[{\cite[Cor.~1.5]{Mis18}}]
If $n\ge 4$ and $u \colon \B^n \to \S^2$ is a minimizing harmonic map, then the top-dimensional part $\sing_* u$ forms an open subset of $\sing u$ and it is a topological $(n-3)$-dimensional manifold of H{\"o}lder class $C^{0,\gamma}$ with every $0 < \gamma < 1$. 
\end{theorem}
The necessary ingredients are further discussed in Section \ref{sec:stab-flatness}. 

Along the way we survey several results for minimizing harmonic maps, in particular we put an emphasis to present in details proofs from \cite{AlmgrenLieb1988}. For an excellent survey on the topic of singularities of harmonic maps we refer the interested reader to \cite{Hardt1997}. 

\subsection{Outline of the article}

In \Cref{s:partialregularity} we recall the basic results on partial regularity of minimizing harmonic maps, in particular we recall the monotonicity formula, define tangent maps, and the stratification of the singular set.

In \Cref{s:boundary} we recall basic tools for boundary regularity, in particular we set the notation for ``straightening'' the boundary.

In \Cref{sec:NabVal} we outline those results of Naber and Valtorta's \cite{NabVal17} that are the most relevant in our case.

In \Cref{s:extension} we discuss the properties of minimizers with values in a manifold with finite fundamental group. In particular in \Cref{ss:extension} we recall the ``extension'' property, which combined with a lifting theorem is our basic tool for comparing the energies. As a consequence we obtain uniform local boundedness of minimizing harmonic maps in \Cref{ss:unifromboundedness} and Caccioppoli inequalities in \Cref{ss:caccioppoli}.

In \Cref{sec:strong-convergence} as a consequence of the previous section we present the compactness results for minimizing harmonic maps. 

In \Cref{sec:boundaryregularity} we present boundary regularity results. In particular we present the improved uniform boundary regularity results for singular boundary data. 

In \Cref{sec:stability} we give the proof of our first main result \Cref{th:stabilitynD}.

In \Cref{sec:AL}, we give the proof of our second main result \Cref{th:intro:al}, which is stated there as \Cref{th:almgren-lieb-high}.

In \Cref{s:finalremarks} we first give examples to prove the optimality of our results in $n=3$ (\Cref{ss:examples}). Next we discuss the possible improvement of our results for $n>3$ and state a conjecture in \Cref{ss:optimalboundarynorm}. Then, in \Cref{ss:othertargetmanifolds}, we briefly discuss the case of target manifolds with infinite fundamental group. 

We close the paper with \Cref{a:traces}, where we review the trace theorems used throughout the paper.

\subsection{Notation}
Throughout the paper we let $\n\subset \R^d$ be a smooth, closed Riemannian manifold, embedded into $\R^d$ and $\Omega\subset\R^n$ be a bounded domain with at least $C^1$-boundary. We write $B_\delta(\n) = \{x\in\R^d\colon \dist(x,\n)<\delta\}$ for the tubular neighborhood of $\n$ on which the nearest point point projection map $\pi_\n$ is well defined and as smooth as the manifold $\n$. Existence of such a neighborhood is a well-known fact (see, e.g., \cite[Appendix 2.12.3]{Simon1996}).

Similarly, for an affine plane $L$ we will write $B_\delta(L)$ to denote the tubular neighborhood of $L$. We denote by $B_r(y)$ the ball centered in $y$ with radius $r$. By $\R^n_+ = \R^n\cap\{x\in\R^n\colon x_n > 0\}$ we denote the upper half-space. $B^+_r(y)$ is given by $B^+_r(y) = B_r(y) \cap \R^n_+$. For any $r>0$ we write $T_r(y) = B_r(y) \cap \{x\in\R^n\colon x_n=0\}$ for the flat part and $S_r^+(y) = \partial B_r \cap \R^n_+$ for the curved part of the boundary of $B_r^+(y)$. Sometimes, we will omit the center and write only $B_r$, $T_r$, or $S_r^+$, when it will not cause any confusion. We will also write sometimes $B^k_r(y)$, to emphasize that the ball is $k$-dimensional. 

For simplicity we use lowercase Greek letters $\psi,\, \varphi,\, \phi$ for boundary maps and $u$, $v$, $w$ etc. for interior maps. The letters $r, R, \rho, t$ will be usually reserved for the radii. We denote the tangential gradient of $u$ (i.e., the gradient of its restriction $u |_{\partial \Omega}$) by $\nabla_T u = \nabla u - (\nabla u)\nu \otimes \nu$, where $\nu$ is the outward normal vector. The indices $s,\,p$ will be reserved for the order and integrability in the fractional Sobolev space. Capital Greek letters $\Phi,\,\Psi$ will be reserved for tangent maps, $\theta_u(y,r)$ for the rescaled energy of the map $u$ at point $y$, and $\Theta$ will be the energy density of $x'/|x'|$. For $x\in\R^n$ we will write $x=(x',x'')\subset\R^{n-k}\times\R^{k}$ unless $k=1$, then $x=(x',x_n)\in \R^{n-1}\times\R$. We write $\omega_{k}$ for the measure of the $k$-dimensional unit ball. In estimates we will often write $A \aleq B$, which means that there exists a constant $C$, not dependent on any crucial quantity, such that $A\le CB$. 

We note here that we will use the letter $\pi$ to denote several things, which might be misleading. Here we list our usage of this letter (of course we will also use the letter to denote the Archimedes' constant):
\begin{enumerate}
 \item $\pi$ without any index will be a Riemannian cover $\pi \in C^\infty(\widetilde{\n},\n)$;
 \item $\pi_\n$ will be the nearest point projection onto the manifold defined on a tubular neighborhood of a manifold $\n$:
 \[
  \pi_\n \colon B_\delta (\n) \to \n;
 \]
\item $\pi_i(\n)$, for $i=1,2,\ldots$ will denote the usual $i$-th homotopy group of the manifold $\n$. We call $\pi_1(\n)$ the fundamental group of $\n$. As customary, we write $\pi_0(\n) = 0$ for connected manifolds. 
\end{enumerate}
Throughout the paper the term \emph{minimizer} or \emph{energy minimizer} will refer to an $\n$-valued map minimizing the Dirichlet energy among $W^{1,2}(\Omega,\n)$ maps with same boundary data, unless otherwise stated.

\subsection{Acknowledgments.}
The authors would like to thank Pawe\l{} Strzelecki for suggesting extending the results of \cite{AlmgrenLieb1988} to higher dimensions. The authors would also like to thank Jean Van Schaftingen for suggesting to extend the results to manifolds with finite fundamental group.

During the long preparation of this manuscript we obtained the following financial support:
\begin{itemize}
\item National Science Centre Poland via grant no. 2015/17/N/ST1/02360 (KM);
\item National Science Centre Poland via grant no. 2016/21/B/ST1/03138 (MM);
\item National Science Centre Poland via grant no. 2020/36/C/ST1/00050 (MM); 
\item German Research Foundation (DFG) through grant no.~SCHI-1257-3-1 (KM, AS);
\item Daimler and Benz foundation, grant no 32-11/16 (KM, AS); 
\item Simons foundation, grant no 579261 (AS);
\item Etiuda scholarship no. 2018/28/T/ST1/00117 (MM); 
\item Mandat d'Impulsion scientifique (MIS) F.452317 - FNRS (KM);
\item FSR Incoming postdoc (KM).
\end{itemize}
Part of the work was carried out while KM was visiting the University of Pittsburgh; she thanks the Department of Mathematics for their hospitality. 

\section{Partial regularity in the interior}\label{s:partialregularity}
In this section we recall the basic regularity results for minimizing harmonic maps, used throughout the paper. 
\subsection{The $\eps$-regularity theorem}

The regularity theory of harmonic maps is based on the $\eps$-regularity theorem from the seminal work of Schoen--Uhlenbeck, see \cite[Theorem I and Theorem 3.1]{SU1}. 

\begin{theorem}[$\eps$-regularity of minimizers]
\label{th:epsreg2}
There exists a constant $\eps>0$ depending on $n$ and $\cN$ such that the following holds. If $u\in W^{1,2}(\Omega,\n)$ is a minimizing harmonic map in an open domain $\Omega \subset \R^n$ and 
\[
R^{2-n} \int_{B_R(y)} |\nabla u|^2 \dx < \eps
\]
for some ball $B_R(y) \subset \Omega$, then $u$ is smooth in the smaller ball $B_{R/2}(y)$.
\end{theorem}

\begin{remark}
 In fact in \cite{SU1} the authors prove that the solutions are H\"{o}lder continuous on smaller balls. For smoothness of the solutions we refer the reader to, e.g., \cite{Moser-book,LinWangbook}. Moreover, if the target manifold is analytic, for example $\n=\S^{d-1}$, then it can be shown that any continuous harmonic map is analytic; we refer to \cite{BG80,T72}.
\end{remark}

Note that the rescaled energy $R^{2-n} \int_{B_R(y)} |\nabla u|^2$ appears naturally in this context. In fact, one can reduce this theorem to the case $\B_R(y) = \B_1$ as follows. With the assumptions as above, one easily checks that the rescaled map $\overline{u}(x) = u((x-y)/R)$ is a minimizing harmonic map in $\B_1$ satisfying $\int_{B_1} |\nabla u|^2 \le \eps$. Then smoothness of $\overline{u}$ in $\B_{1/2}$ implies that $u$ is regular in $\B_{R/2}(y)$. This rescaling argument will be used multiple times in our considerations. 

\subsection{Monotonicity formula}

As already mentioned, the rescaled energy 
\begin{equation}\label{eq:thetau}
 \theta_u(y,r)\coloneqq r^{2-n}\int_{\B_r(y)} |\nabla u|^2 \dx
\end{equation}
is a central object in the study of singularities. We now show that it is monotone in $r$. The first published version of a monotonicity formula for minimizing harmonic maps was in Schoen and Uhlenbeck's \cite[Proposition 2.4]{SU1}. 

\begin{theorem}[Interior Monotonicity formula]\label{th:monotonicityformula}
Let $\Omega \subset \R^n$ and let $u\in W^{1,2}(\Omega, \n)$ be a~minimizing harmonic map. Then for any $0<r<R<\dist(y,\partial \Omega)$ 
 \begin{equation}\label{eq:monotonicityformula}
  R^{2-n}\int_{B_R(y)}|\nabla u|^2\dx - r^{2-n}\int_{B_r(y)}|\nabla u|^2 \dx 
  \ge \int_{B_R(y)\setminus B_r(y)}|x-y|^{2-n}\left|\frac{\partial u}{\partial \nu}\right|^2 \dx,
  \end{equation}
  where $\frac{\partial u}{\partial \nu}$ is the directional derivative in the radial direction $\frac{x-y}{|x-y|}$.
\end{theorem}

It is now evident that $r^{2-n}\int_{B_r}|\nabla u|^2 \dx$ is constant in $r$ if and only if $u$ is a $0$-homogeneous (i.e., radially constant) map. Minimizing harmonic maps with this property are called \emph{tangent maps} and play a special role (see Section \ref{s:tangentmaps}). 

\begin{proof}
We follow the original proof of Schoen and Uhlenbeck. 

Without loss of generality we may assume that $y=0$. For almost every $\varrho \in [r,R]$, the trace $u|_{\B_\varrho}$ belongs to $W^{1,2}(\partial \B_\varrho,\n)$. For such $\varrho$, the map $v(x) = u(\varrho \cdot \frac{x}{|x|})$ lies in $W^{1,2}(\B_\varrho,\n)$, moreover $\int_{\B_\varrho} |\nabla v|^2 = \frac{\varrho}{n-2} \int_{\partial \B_\varrho} |\nabla_T u|^2$, where $\nabla_T u$ denotes the differential restricted to directions tangent to $\partial \B_\varrho$. 

Since $v = u$ on $\partial \B_\varrho$, $v$ is a valid competitor for $u$ on $\B_\varrho$. It follows from minimality of $u$ that 
\begin{align*}
\int_{\B_\varrho} |\nabla u|^2 & \le \frac{\varrho}{n-2} \int_{\partial \B_\varrho} |\nabla_T u|^2.
\end{align*}
Thus,
\begin{align*}  
\tfrac{\partial}{\partial \varrho} \theta_u(0,\varrho) 
& = \varrho^{2-n} \int_{\partial \B_\varrho} |\nabla u|^2 - (n-2) \varrho^{1-n} \int_{\B_\varrho} |\nabla u|^2 \\
& \ge \varrho^{2-n} \int_{\partial \B_\varrho} |\nabla u|^2 - \varrho^{2-n} \int_{\partial \B_\varrho} |\nabla_T u|^2 \\
& = \varrho^{2-n} \int_{\partial \B_\varrho} \left|\frac{\partial u}{\partial \nu}\right|^2, 
\end{align*}
using the decomposition $|\nabla u|^2 = |\nabla_T u|^2 + \left|\frac{\partial u}{\partial \nu}\right|^2$ in the last line. Since $\theta(0,\varrho)$ is an absolutely continuous function, the claim can be obtained by integrating the above inequality from $\varrho = r$ to $\varrho = R$. 
\end{proof}

As observed by Hardt and Lin in \cite[Lemma 4.1]{HL1987}, Theorem \ref{th:monotonicityformula} can be strenghtened by considering a squeeze deformation as in \cite[2.4, 2.5]{Almgren-big} --- one can prove an equality instead of a lower bound. 
\begin{theorem}\label{th:monotonicityformula-general}
Let $u,\ r,\ R$ be as in Theorem \ref{th:monotonicityformula}, then
 \begin{equation}\label{eq:monotonicityformulaequality}
  R^{2-n}\int_{B_R(y)}|\nabla u|^2\dx - r^{2-n}\int_{B_r(y)}|\nabla u|^2 \dx = 2\int_{B_R(y)\setminus B_r(y)}|x-y|^{2-n}\left|\frac{\partial u}{\partial \nu}\right|^2 \dx.
  \end{equation}
\end{theorem}

We remark here that this result was generalized for stationary harmonic maps (and Yang--Mills fields) by Price in \cite{Price}, see also \cite[Section 2.4]{Simon1996} for a nice presentation.

\subsection{Energy density}

It follows from the monotonicity formula (Theorem \ref{th:monotonicityformula}) that the limit 
\begin{equation}
\label{eq:energy-density}
 \theta_u(x,0) \coloneqq \lim_{r\searrow 0} \theta_u(x,r) = \lim_{r\searrow 0} r^{2-n}\int_{B_r(x)} |\nabla u|^2 \dx 
\end{equation}
exists. We shall call it the \emph{energy density} of $u$ at $y$. Evidently, $\theta_u(x,0) = 0$ at regular points. 

With this in hand, the $\eps$-regularity theorem (Theorem \ref{th:epsreg2}) can be restated as follows: 
\begin{align*}
\text{there is } \varepsilon(n,\cN)>0 \text{ s.t. } \theta_u(x,2r) < \varepsilon 
& \Rightarrow u \text{ is smooth in } \B_r(x), \\
\text{in particular, } \theta_u(x,0) < \varepsilon & \Rightarrow x \notin \sing u.
\end{align*}

A weak version of partial regularity can now be deduced directly: $\cH^{n-2}(\sing u) = 0$. Indeed, for any finite positive measure $\mu$ in $\R^n$ it is true that the set of points $x \in \R^n$ satisfying $\liminf_{r \to 0} \frac{\mu(\B_r(x))}{r^k} \ge \eps$ has zero $k$-dimensional Hausdorff measure. In our case, it is enough to set $\dd \mu = |\nabla u|^2 \dx$ and $k = n-2$. 

As already mentioned in the introduction, it is possible to upgrade this dimension bound to $\dim_{\cH} (\sing) u \le n-3$. For this one needs to study the so-called tangent maps, see \Cref{s:tangentmaps}, and apply Federer's dimension reduction as in \cite[Section 5]{SU1}, see also \cite[Theorem A.4]{S83}. 

\subsection{Convergence of singular points}
For each fixed $r > 0$, the function $\theta_u(y,r)$ is continuous in both $u \in W^{1,2}$ and $y \in \Omega$. It is useful to note that $\theta_u(y,0)$ is upper semicontinuous as it is the pointwise infimum of these functions, thus 
\[
u_k \to u \text{ strongly in } W^{1,2}(\Omega,\cN), \ 
y_k \to y \text{ in } \Omega 
\quad \Rightarrow \quad 
\theta_{u}(y,0) \ge \limsup_{k \to \infty} \theta_{u_k}(y_k,0),
\]
given that all maps $u$, $u_k$ are minimizing. We give a simple consequence of this fact. 
\begin{theorem}[{Singular points converge to singular points, {\cite[Thm 1.8 (i)]{AlmgrenLieb1988}}}]\label{th:ALs2s}
Assume that a sequence of energy minimizing maps $u_k \in W^{1,2}(\Omega, {\n})$ converges strongly in $W^{1,2}_{loc}$ to a~minimizer $u$, and a sequence of their singularities $y_k \in \sing u_k$ converges to $y \in \Omega$. Then $y$ is a singular point of $u$. 
\end{theorem}

\begin{proof}
By $\eps$-regularity, we have $\theta_{u_k}(y_k,0) \ge \eps$ for each $k$. Upper semicontinuity now implies $\theta_u(y,0) \ge \eps$, hence $y$ has to be a singular point.
\end{proof}

\begin{remark}
 We would like to note here that if we work with $\n=\S^2$ then a reverse statement of \Cref{th:ALs2s} is true, see Theorem \ref{th:ALs2sii}.
\end{remark}

\subsection{Tangent maps}\label{s:tangentmaps}

In this subsection we recall various facts concerning tangent maps which will be useful for future purposes. For more details we refer the interested reader to \cite[Chapter 3]{Simon1996}.

Let $u\in W^{1,2}(\Omega,\n)$ be a minimizing harmonic map, $y\in\Omega$ and $\lambda>0$. We define the rescaled maps $u_{y,\lambda}\in W^{1,2}\brac{\frac 1\lambda (\Omega-y),\n}$ by
\[
 u_{y,\lambda}(x) \coloneqq u(y+ \lambda x).
\]

We say that $\Phi\colon\R^n\rightarrow\n$ is a \emph{tangent map to $u$ at point $y$} if it is a $W^{1,2}_{loc}$ strong limit of $u_{y,\lambda}$ for some sequence $\lambda\searrow 0$. Another consequence of the monotonicity formula is the following lemma on existence of tangent maps.

\begin{lemma}[{tangent maps, {\cite[Lemma 2.5]{SU1}}}]\label{la:tangentmaps}
For any $y\in\Omega$ and any sequence $\lambda_i \searrow 0$ there is a subsequence (still denoted $\lambda_i$), for which $u_{y,\lambda_i}$ is strongly convergent in $W^{1,2}_{loc}$ to a~minimizing harmonic map $\Phi \in W^{1,2}_{loc}(\R^n,\n)$. Moreover, $\Phi$ is homogeneous of degree $0$ and its energy is consistent with the energy density of $u$ at $y$: 
\[
\int_{\B_1(0)} |\nabla \Phi|^2 \dd x = \lim_{r \to 0} r^{2-n} \int_{\B_r(y)} |\nabla u|^2 \dd x.
\]  
\end{lemma}

\begin{proof}
Fix a ball $B_r(y) \subset \Omega$. For any $R > 0$, the monotonicity formula \eqref{eq:monotonicityformula} yields the bound
\[
 \begin{split}
  \int_{B_R(0)}|\nabla u_{y,\lambda_i}|^2 \dx
  = \lambda_i^{2-n} \int_{B_{\lambda_i R}(y)} |\nabla u|^2 \dx 
  \le (r/R)^{2-n} \int_{B_r(y)}|\nabla u|^2 \dx 
 \end{split}
\]
for all large enough $i$ (the condition $\lambda_i \le r/R$ is used above). Thus, the sequence $u_{\lambda_i,y}$ is bounded in $W^{1,2}(\B_R(0),\cN)$. By a diagonal argument, we can choose a weakly convergent subsequence, i.e.,  $u_{y,\lambda_{i}} \rightharpoonup \Phi$ weakly in $W^{1,2}_{loc}$ for some $\Phi \in W^{1,2}_{loc}(\R^n,\n)$.  

By the compactness result, Theorem \ref{th:bsc}\footnote{The Compactness Theorem is stated only for manifolds with finite fundamental group but as noted in the introduction to Section \ref{sec:strong-convergence}, due to Luckhaus Lemma~\cite{L88}, the result holds true for any closed $\n$.} (see also \cite[Section 2.9]{Simon1996}), we infer that the convergence is in fact strong and that the limiting map $\Phi$ is minimizing. 

\medskip

To show that $\Phi$ is $0$-homogeneous, we take the limit $i \to 0$ in the estimate above. For each $R > 0$ we have 
\begin{align*}
\int_{\B_R(0)} |\nabla \Phi|^2 \dx 
& = \lim_{i \to \infty} \int_{\B_R(0)} |\nabla u_{y,\lambda_i}|^2 \dx \\
& = R^{n-2} \cdot \lim_{i \to \infty} (\lambda_i R)^{2-n} \int_{B_{\lambda_i R}(y)} |\nabla u|^2 \dx. 
\end{align*}
Whatever $R$ is, the last limit is just the energy density $\theta_u(y,0)$, so the rescaled energy $\theta_\Phi(0,R)$ does not depend on $R$. By monotonicity formula, this implies $\frac{\partial \Phi}{\partial \nu} \equiv 0$, and hence $\Phi$ is $0$-homogeneous. 
\end{proof} 

In the case when $\n$ is an analytical manifold and $n=3$ (or more generally, when $\n$ is analytic and there exists a tangent map $\Phi$ for which $\sing \Phi = \{0\}$) Simon proved uniqueness of the tangent map \cite[Section 8]{Simon1983}. In general, the limiting map depends on the choice of the sequence $\lambda_i$ and its subsequence (see \cite{White1992}). 

\subsection{Stratification of the singular set}

Now we explain the relationship of the singular set with tangent maps. First, we observe that $y\in\Omega$ is a~regular value if there exists a~constant tangent map to $u$ at $y$.

Next, we observe that for any tangent map $\Phi$ the maximum of the energy density is attained at $0\in\R^n$:
\[
 \theta_\Phi(y,0)\le \theta_\Phi(0,0) \quad \text{for any } y \in \R^n.
\]
If we assume additionally that $\theta_\Phi(y,0) = \theta_\Phi(0,0)$ then we obtain
\[
 \Phi(x+\lambda y) = \Phi(x) \quad \text{for any } \lambda \in \R \text{ and } x\in \R^n,
\]
which leads to the definition
\[
 S(\Phi)\coloneqq \{y\in\R^n \colon \theta_\Phi(y,0) = \theta_\Phi(0,0)\}.
\]
Observe that for non-constant tangent map $\Phi$ we have $S(\Phi)\subset \sing \Phi$.

We introduce the notion of $k$-symmetric maps. A map $f \colon \R^n \to \cN$ is called $k$-symmetric if $f(\lambda x) = f(x)$ for any $x \in \R^n$, $\lambda > 0$, and there exists a linear $k$-dimensional plane $L \subset \R^n$ such that $f(x+y) = f(x)$ for any $x \in \R^n$, $y \in L$. The space of such functions will be denoted by $\mathrm{sym}_{n,k}$. 

Next we observe
\[
 y\in \sing u \Longleftrightarrow\dim S(\Phi) \le n-1 \quad \text{ for every tangent map $\Phi$ of $u$ at $y$}.
\]
We define for all $j\in \{0,\ldots,n-1\}$ the stratification of the singular set
\[
\begin{split}
 S_j &\coloneqq \{y\in \sing u \colon \dim S(\Phi) \le j \text{ for all tangent maps $\Phi$ of $u$ at $y$} \}\\
 &= \{y\in \sing u \colon \text{ no tangent map of $u$ at $y$ belongs to $\mathrm{sym}_{n,j+1}$}\}.
 \end{split}
\]
Note that 
\[
 S_0 \subset S_1 \subset \ldots \subset S_{n-4}\subset S_{n-3} = S_{n-2} = S_{n-1} = \sing u,
\]
since the existence of nonconstant $(n-2)$-symmetric tangent maps would contradict the weak version of partial regularity $\cH^{n-2}(\sing u) = 0$ that we mentioned earlier. It can be shown that 
\begin{equation}
\label{eq:Sjdim}
 \dim_{\mathcal{H}} (S_j) \leq j,
\end{equation}
which implies in particular the partial regularity result $\dim_\H (\sing u) \le n-3$, see \cite[Theorem II]{SU1}.

We will be mainly interested in the top-dimensional part of the singular set, so for this purpose we define
\[
 \sing_* u = S_{n-3} \setminus S_{n-4}.
\]

\subsection{Classification of tangent maps into $\S^2$}

We recall the classification of tangent maps into $\S^2$ by Brezis--Coron--Lieb \cite{BCL1986}.
\begin{theorem}[{{\cite[Theorem~1.2]{BCL1986}}}]
\label{th:BCLclassification}
Every nonconstant locally minimizing harmonic map $\R^3 \to \S^2$ has the form $\frac{\tau x}{|\tau x|}$ for some linear isometry $\tau$ of $\R^3$. 
\end{theorem}

In higher dimensions, we will use the symbol $\Psi \colon \R^n \to \S^2$ to denote the map
\begin{equation}
\label{eq:def-psi}
\R^3 \times \R^{n-3} \ni (x',x'')
\xmapsto{\quad \Psi \quad}
\frac{x'}{|x'|} \in \S^2.
\end{equation}
Its energy density will be denoted by 
\begin{equation}
\label{eq:def-Theta}
\Theta \coloneqq \int_{B_1(0)} |\nabla \Psi|^2 \dx.
\end{equation}

We remark that $\Theta = 8\pi$ for $n=3$ and 
\[
\Theta 
= 8 \pi \int_{\B_1^{n-3}} \sqrt{1-|x''|^2} \dd x''
= 4 \pi |\omega_{n-2}| 
\qquad \text{for } n \ge 4,
\]
but the precise value of $\Theta$ has no importance for our considerations. 

We note that the map $\Psi$ belongs to $\mathrm{sym}_{n,k}$ for all $k=0,1,\ldots,n-3$ but not to $\mathrm{sym}_{n,n-2}$. 

The classification of tangent maps in dimension $3$ (Theorem \ref{th:BCLclassification}) can be used to classify the $(n-3)$-symmetric tangent maps in general (see also \cite[Corollary 2.2]{HLB4}). 

\begin{corollary}\label{co:uniquetangent}
Suppose $u\in W^{1,2}(\Omega, \S^2)$ is a minimizing harmonic map and $y\in \sing_*u$, then up to isometries of $\R^n$ the only tangent map of $u$ at $y$ is $\Psi$. In particular the energy density of $u$ at a point from $\sing_* u$ equals $\Theta$.
\end{corollary}

\begin{proof}
By definition, $u$ has a nonconstant $(n-3)$-symmetric tangent map $\Phi$ at $y$. Denoting the $(n-3)$-dimensional plane $L = S(\Phi)$, we can represent $\Phi$ as 
\[
\Phi(x',x'') = \Phi_0(x') 
\quad \text{for } (x',x'') \in L^\perp \times L, 
\]
where $\Phi_0$ is another nonconstant $0$-homogeneous map. By \cite[Lemma 5.2]{SU1} we know that $\Phi_0$ is also a locally minimizing harmonic map. Theorem~\ref{th:BCLclassification} implies now that (up to an isometry) $\Phi_0 = \frac{x'}{|x'|}$. This shows that indeed $\Phi$ differs from $\Psi$ by a composition with an isometry. 
\end{proof}

\subsection{Additional properties of singularities of minimizers into $\S^2$}
\hfill

In this subsection we mention a few results that hold in the special case when the target manifold is a two dimensional sphere. Those results were used by Almgren and Lieb \cite{AlmgrenLieb1988} in their proof of the linear law. We will not use them but we present them here to familiarize the reader with the special case $n=3$ and $\n = \S^2$.

In the case when $\n = \S^2$ a reverse statement of Theorem \ref{th:ALs2s} is true.

\begin{theorem}[{{\cite[Thm 1.8 (2)]{AlmgrenLieb1988}}}]\label{th:ALs2sii}
Assume $u_i\in W^{1,2}(\Omega, \S^2)$ is a sequence of minimizing maps in $\Omega\subset\R^3$, which converges strongly in $W^{1,2}_{loc}$ to $u$. Then, if $y\in\Omega$ is a singular point of $u$, then for all sufficiently large $i$, $u_i$ has a singular point at $y_i$ and $y_i\rightarrow y$.
\end{theorem}

\begin{proof}
By classification of tangent maps, Theorem~\ref{th:BCLclassification}, if $y$ is a singular point of $u$ we know that on balls of small radius $B_r(y)$
\[
 u\sim \tau \brac{\frac{x-y}{|x-y|}},  
\]
for a linear isometry $\tau$ of $\R^3$. Since, the map $\frac{x}{|x|}\in W^{1,2}(B_1^3,\S^2)$ cannot be approximated by $C^\infty(\overline{B}_1^3,\S^2)$ maps, see \cite[Section 4]{SU2} or \cite{B91}, we infer that for $i \ge i_0(r)$, $u_i$ must have a~singular point $y_i\in B_r(y)$. 

Applying this reasoning for a sequence $r_i \searrow 0$, we obtain a sequence of singularities $y_i \in \sing u_i$ converging to $y$. 
\end{proof}

The following two results exploit the classification of singularities into $\S^2$ (Theorem \ref{th:BCLclassification}) even further --- here it will be important that at each singular point $y \in \sing u$ the energy density $\theta_u(y,0)$ \eqref{eq:energy-density} has the same value 
\[
\int_{B_1^3(0)} \left| \nabla \brac{\frac{x}{|x|}} \right|^2 \dx = 8 \pi.
\]
as noted in Corollary \ref{co:uniquetangent}. 

\begin{lemma}[{Liouville theorem, {\cite[Thm 2.2]{AlmgrenLieb1988}}}]
\label{lem:liouville}
Let $u \colon \R^3 \to \S^2$ be locally energy minimizing in all of $\R^3$. Then, up to a translation, $u$ is a tangent map, i.e., $u$ is either constant or has the form $u(x) =  \tau \left( \frac{x-y}{|x-y|} \right)$ for some $y \in \R^3$ and some linear isometry $\tau$ of $\R^3$. 
\end{lemma}

\begin{proof}
Let us first consider the singular case, without loss of generality we may assume $0 \in \sing u$. By Theorem \ref{th:uniformboundedness} (uniform boundedness), the monotone quantity $\theta_u(0,r)$ is bounded, so it has a finite limit
\[
\Theta' \coloneqq \lim_{r \to \infty} \theta_u(0,r). 
\]
Moreover, the sequence of rescaled maps $u_r(x) = u(rx)$ is bounded in $W_\loc^{1,2}(\R^3,\S^2)$ as $r \to \infty$, and after choosing a subsequence, we can obtain in the limit an energy minimizing limit map $\Phi \colon \R^3 \to \S^2$ (\emph{a tangent map at infinity}). Note that for each $s > 0$ we have
\[
s^{-1}\int_{B_s} |\nabla \Phi|^2 \dx
= \lim_{r \to \infty} s^{-1} \int_{B_s} |\nabla u_r|^2 \dx 
= \lim_{r \to \infty} (rs)^{-1} \int_{B_{rs}} |\nabla u|^2 \dx 
= \Theta',
\]
in particular $\Phi$ has energy density $\theta_\Phi(0,0) = \Theta'$. As there is only one possible energy density, we infer that $\Theta'$ is equal to $8\pi = \theta_u(0,0)$. Monotonicity formula (Theorem \ref{th:monotonicityformula}) now implies that $u$ is $0$-homogeneous (i.e., a tangent map), since the monotone quantity has the same limit for $r \to 0$ and $r \to \infty$. 

\medskip

If $u$ is smooth, most of the above discussion still applies. If $\Theta' = 0$, then in particular $\int_{B_r} |\nabla u|^2\dx$ tends to zero as $r \to \infty$, so $u$ is constant. If $\Theta' > 0$, then the obtained map $\Phi$ has a singularity at the origin, and by Theorem \ref{th:ALs2sii} the rescaled maps $u_r$ are also singular for large $r$. Since $u$ is smooth, this yields a contradiction. 
\end{proof}

\begin{theorem}[{Uniform distance between singular points, {\cite[Thm 2.1]{AlmgrenLieb1988}}}]\label{th:udsg}
There is a~constant $c > 0$ such that the following holds. Let $\Omega\subset\R^3$ be a bounded domain, and $u \in W^{1,2}(\Omega, \S^2)$ a minimizing harmonic map with a singularity at $y \in \Omega$. Then there is no other singularity within distance $cD$ of $y$, where $D \coloneqq \dist(y,\partial\Omega)$ is its distance to the boundary. 
\end{theorem}

\begin{proof}
Assume the claim is false. Then we can find a sequence $u_i \in W^{1,2}(\Omega_i,\S^2)$ with two distinct singularities $x_i,y_i \in B_{D_i/i}(y_i)$, where $D_i = \dist(y_i,\partial\Omega_i)$. For each $i$ consider the rescaled map
\[
\tilde{u}_i(z) \coloneqq u_i \left ( y_i + |x_i-y_i| z \right ),
\]
which is a minimizing harmonic map in a large ball $B_i(0)$. This map has two singularities at $0$ and $\frac{x_i-y_i}{|x_i-y_i|}$. Using Theorem \ref{th:uniformboundedness} (uniform boundedness), Theorem \ref{th:bsc} (compactness of minimizers), compactness of $\S^2$, and a diagonal argument, we obtain an energy minimizing limit map $u \colon \R^3 \to \S^2$, which is singular at least in two points $0$ and $x$ with $|x|=1$. However, the possibility of two singularities is excluded by the Liouville theorem (Lemma \ref{lem:liouville}). 
\end{proof}

\begin{corollary}[Uniform bound for singularities in the interior]\label{co:sgint}
Let $\Omega \subset \R^3$ be a bounded domain, and $u \in W^{1,2}(\Omega,\S^2)$ a minimizing harmonic map. Then for any $\sigma > 0$, the number of singularities with distance to the boundary at least $\sigma$ is bounded by a constant depending only on $\Omega$ and $\sigma$: 
\[
\# \{ x \in \sing u\colon  \dist(x,\partial \Omega) \ge \sigma \} \le C(\Omega,\sigma). 
\]

\end{corollary}

\section{Monotonicity formula and tangent maps at the boundary}\label{s:boundary}


\subsection{Notation for "straightening" the boundary}\label{sub:str8}

As we will be working at the boundary we need to distinguish one variable. For a point $x\in \R^n$ we write $x=(x',x_n)\in \R^{n-1}\times\R$. Let us assume that $\Omega\subset\R^n$ is a bounded domain with a $C^1$-boundary. This means that $\pl \Omega$ can be described by a $C^1$-graph on each ball $\B_{R_0}(a)$ around $a \in \pl \Omega$ (with uniform $R_0 > 0$). 

For each $a \in \pl \Omega$, we may choose a rigid motion (i.e., a rotation followed by a translation) $h \colon \R^n \to \R^n$ that sends $a$ to $0$, $T_a \pl \Omega$ to $\R^{n-1} \times 0$ and the outer normal vector to $(0,\ldots,0,-1)$. Thus, there is a $C^1$ function $\alpha_a \colon \R^{n-1} \to \R$ such that 
\begin{gather*}
h(\Omega \cap \B_{R_0}(a)) = \{ x \in \B_{R_0}(0) \colon x_n > \alpha(x') \}, \\
\alpha_a(0) = 0, \ \nabla \alpha_a(0) = 0.
\end{gather*}
Since $\alpha_a \in C^1$ and $\pl \Omega$ is compact, we may assume $|\nabla \alpha_a(x')| \le \omega(|x'|)$ with a uniform modulus of continuity $\omega$ ($\omega(t) \to 0$ as $t \to 0$). 

The boundary of $\Omega$ can be "straightened out" by the diffeomorphism 
\[
\varsigma_a(x',x_n) = (x',x_n-\alpha_a(x')), 
\quad 
\varsigma_a^{-1}(x',x_n) = (x',x_n+\alpha_a(x')),
\]
i.e., $\varsigma_a$ is a $C^1$-diffeomorphism for which $\varsigma_a(\pl \Omega \cap \B_{R_0}(a)) \subseteq \R^{n-1} \times 0$. Moreover, the estimates for $\alpha_a$ give us 
\[
|\nabla \varsigma_a(x) - \id| \le |\nabla \alpha_a(x')| \le \omega(|x'|) 
\quad \text{for } x \in \Omega \cap \B_{R_0}.
\]

In what follows, we will consider rescaled maps at the boundary. For this we need to define the functions that describe the rescaled boundary. For each $0 < r \le 1$ let 
\[
\alpha_{r,a}(x') = r^{-1} \alpha_{a}(r x'). 
\]
One can easily see that
\[
 \norm{\alpha_{r,a}}_{\Lip(B_R)} = \norm{\alpha_{a}}_{\Lip(B_{r R})} \xrightarrow{r \to 0} 0 
 \quad \text{for any } R > 0.
\]
This map describes the boundary of the set $\Omega_{r^{-1}}(a)$ defined as
\[
 \Omega_{r^{-1}}(a) \coloneqq \{x\in B_{R_0/r}(0)\colon x_n > \alpha_{r,a}(x')\} 
 = \{ x \in \R^n : rx \in h(\Omega) \}.
\]
We observe that as $r \rightarrow 0$ we have $\alpha_{r,a}\rightarrow 0\eqqcolon\alpha_{0,a}$ (locally uniformly), which motivates the definition 
\[
 \Omega_\infty(a)\coloneqq\{x\in \R^n\colon x_n>\alpha_{0,a}(x')\}= \R^n_+.
\]


Assume that $u\in W^{1,2}(\Omega, \n)$ is a minimizing harmonic and let us denote its trace by $\varphi$. We define the map $u_{r,a}$ on the domain $\Omega_{R_0/r}(a)$ by
\[
u_{r,a}(\cdot)\coloneqq u(rh(\cdot)+a), \quad \text{ in } \Omega_{R/r}(a),
\]
where on the portion of the boundary, $\partial\Omega_{R_0/r}(a)\cap B_1(0) = \brac{\{x_n = \alpha_{r,a}(x')\}\cap B_1(0)}$,  its trace is given by 
\[
 \varphi_{r,a}(\cdot)\coloneqq\varphi(rh(\cdot)+a), \quad \text{ on } \partial\Omega_{R_0/r}(a)\cap B_1(0), 
\]
with Lipschitz constant
\[
 \norm{\varphi_{r,a}}_{\Lip} = {r}\norm{\varphi}_\Lip.
\]
We note that $u_{r,a}$ is also a minimizing harmonic map and
\[
 \int_{B_1(0) \cap \Omega_{r^{-1}}(a)} |\nabla u_{r,a}|^2 \dx = r^{2-n}\int_{B_r(a)\cap\Omega}|\nabla u|^2\dd y.
\]

\subsection{Boundary monotonicity formula}

We will employ a boundary monotonicity formula of Schoen and Uhlenbeck \cite[Lemma 1.3]{SU2}. 

In the sequel, we will only use it for $\Omega = \B_1^+$ and constant boundary data (cf., Theorem \ref{th:boundaryregularity-constant}). In this particular case, it is enough to repeat the simple argument given for the interior monotonicity formula (Theorem \ref{th:monotonicityformula}) to obtain 
\[
R^{2-n}\int_{B_R^+}|\nabla u|^2\dx - r^{2-n}\int_{B_r^+}|\nabla u|^2 \dx 
\ge \int_{B_R^+ \setminus B_r^+} |x-y|^{2-n}\left|\frac{\partial u}{\partial \nu}\right|^2 \dx 
\quad \text{for } 0 < r < R < 1.
\]
For the reader's convenience, we include the proof in full generality, with Lipschitz boundary data. In this case, the almost-monotone quantity is slightly different.

For another proof of boundary regularity for \emph{minimizing} harmonic maps we refer the interested reader to \cite[Theorem 2]{GBmono}.

\begin{theorem}[Boundary Monotonicity formula]\label{th:bdmonotonicityformula}
Let $\Omega\subset \R^n$ be a bounded domain with $C^1$-boundary, $u\in W^{1,2}(\Omega, \n)$ be a minimizing harmonic map with $u = \varphi$ on $\partial \Omega$ and  $\varphi \in \Lip(\partial\Omega,\n)$. Then, there exists a radius $R_0 = R_0(\|\varphi\|_\Lip,\Omega)$ such that for any $a\in \partial \Omega$ and $0< r < R < R_0$ 
\begin{equation}\label{eq:boundarmonotonicityformula}
\begin{split}
\bigg[ (1+C\varrho\|\vp\|_{\Lip})^{n-2} \varrho^{2-n}& \int_{B_\varrho(a)\cap \Omega} |\nabla u|^2 \dx \bigg] \Bigg|_{\varrho=r}^{\varrho=R}\\
&\quad \ge 
\int_{\Omega \cap B_R(a)\setminus B_r(a)}|x|^{2-n}\left|\frac{\partial u}{\partial \nu}\right|^2 \dx 
- C\|\vp\|_{\Lip}(R-r).
\end{split}
\end{equation}
where $C=C(n,\Omega,\n)$.
\end{theorem}

\begin{proof}
\textsc{Step 1 (flat boundary).} First we assume that the boundary is flat, i.e., $B_R(a)\cap \Omega = B_R^+(0) = B_R^+$.

We will use a similar comparison map as in proof of Theorem \ref{th:monotonicityformula}, but we need to change that construction so that the maps agree on the flat part of the boundary. In order to do so we extend the map $\varphi$ defined on $T_\varrho = \partial B_\varrho^+ \cap \{x_n =0\}$ to the whole half-ball $B_\varrho^+$ --- we let $\varphi(x) = \varphi(x_1, x_2, \ldots, x_{n-1},0)$ and abuse the notation slightly by using $\vp$ to denote the extension as well.

We define
\[
 \tilde{v}_\varrho (x)\coloneqq (u-\varphi)\brac{\varrho\frac{x}{|x|}} + \varphi(x), \quad \text{for } x\in B_\varrho^+(0)
\]
for any $\varrho< R_0$. One easily sees that
\[
 \begin{split}
  \tilde{v}_\varrho = \varphi &\quad \text{on } T_\varrho\\
  \tilde{v}_\varrho = u &\quad  \text{on } S_\varrho^+,
 \end{split}
\]
where $S_\varrho^+ =\partial B_\varrho(0)\cap \R^n_+$ is the curved part of the boundary $\partial B_\varrho^+(0)$.

Although $\tilde{v}_\varrho$ does not lie in the manifold $\n$ we can  estimate for $x \in B_\varrho^+(0)$
\begin{equation}\label{eq:flatdistanceestimate}
 \dist(\tilde{v}_\varrho,\n) \le \left| \tilde{v}_\varrho(x) - u\brac{\varrho\frac{x}{|x|}}\right| = \left|\varphi(x) - \varphi\brac{\varrho\frac{x}{|x|}}\right| \le \left|x - \varrho \frac{x}{|x|}\right| \|\varphi\|_{\Lip} \le \varrho \|\varphi\|_{\Lip}.
\end{equation}
Thus, for sufficiently small $R_0$ the mapping $\tilde{v}_\varrho$ has values in $B_\delta(\n)$ --- a neighborhood of $\n$ on which the nearest point projection $\pi_\n\colon B_\delta(\n) \rightarrow \cN $ is well defined.

Now we can define the comparison map as $v_\varrho\coloneqq \pi_\n \circ \tilde{v}_\varrho$. By minimality

\begin{equation}\label{eq:bdrymonocomparisonstep}
 \int_{B_\varrho^+(0)} |\nabla u|^2 \dx \le \int_{B_\varrho^+(0)} |\nabla v_\varrho|^2 \dx,
\end{equation}
so we carry on with estimating the right-hand side. First, we estimate the energy of $\tilde{v}_\varrho$ Using Cauchy--Schwartz inequality with $\varepsilon=\varrho\|\vp\|_{Lip}$ we estimate
\begin{equation}\label{eq:bdrymonocs-tilde}
\begin{split}
 \int_{B_\varrho^+(0)} |\nabla \tilde{v}_\varrho|^2\dx &\le 
 (1+\eps)\int_{B_\varrho^+(0)} \left|\nabla\brac{u\brac{\varrho\frac{x}{|x|}}}\right|^2 \dx\\
 &\quad + \brac{1+\frac{1}{\eps}} \int_{B_\varrho^+(0)} \left|\nabla \brac{\varphi(x) - \vp\brac{\varrho\frac{x}{|x|}}}\right|^2 \dx \\
 &= (1+\varrho \|\vp\|_{Lip})\int_{B_\varrho^+(0)} \left|\nabla\brac{u\brac{\varrho\frac{x}{|x|}}}\right|^2 \dx + C\varrho^{n-1}\|\vp\|_{Lip}.
 \end{split}
\end{equation}
To estimate the energy of $v_\varrho$, we note that $\nabla v_\varrho(x) = \nabla \pi_\cN(v(x)) \cdot \nabla \tilde{v}_\varrho(x)$. For $p \in \cN$, $\nabla \pi_\cN(p)$ is an orthogonal projection, hence $|\nabla \pi_\cN|^2 = 1$ (we consider the operator norm here). Since this is a smooth function of $p$, we have $|\nabla \pi_\cN(p)|^2 \le 1 + C \dist(p,\cN)$ for $p$ close to $\cN$. In particular, 
\[
|\nabla \pi_\cN(\tilde{v}_\varrho(x))|^2 
\le 1 + C \dist(\tilde{v}_\varrho(x),\cN) 
\le 1 + C \varrho \|\varphi\|_{\Lip}.
\]
As $|\nabla v_\varrho(x)| \le |\nabla \pi_\cN(v(x))| \cdot |\nabla \tilde{v}_\varrho(x)|$, the energy of $v_\varrho$ is at most $1 + C \varrho \|\varphi\|_{\Lip}$ times the energy of $\tilde{v}_\varrho$. However, we only consider small values of $\varrho$, so by enlarging the constants in \eqref{eq:bdrymonocs-tilde} we conclude 
\begin{equation}\label{eq:bdrymonocs}
 \int_{B_\varrho^+(0)} |\nabla v_\varrho|^2\dx 
 \le (1+\varrho \|\vp\|_{Lip})\int_{B_\varrho^+(0)} \left|\nabla\brac{u\brac{\varrho\frac{x}{|x|}}}\right|^2 \dx + C\varrho^{n-1}\|\vp\|_{Lip}.
\end{equation}


Observing that
\[
\frac{d}{d\varrho} \int_{B_\varrho^+(0)} f \dx = \int_{S_\varrho^+} f \dhn,
\]
we can now compute similarly as in the proof of Theorem \ref{th:monotonicityformula}
\begin{align*}
  \int_{B_\varrho^+(0)} \left| \nabla \brac{u\brac{\varrho \frac{x}{|x|}}}\right|^2 \dx 
  & = \frac{\varrho}{n-2} \int_{S^+_\varrho} |\nabla_T u|^2 \dhn \\
  & = \frac{\varrho}{n-2}\brac{ \frac{d}{d\varrho} E_\varrho^+(u) - \int_{S^+_\varrho}\left|\frac{\partial u}{\partial \nu}\right|^2 \dhn},
\end{align*}
where 
\begin{equation}\label{eq:bdryder}
 E_\varrho^+(u) \coloneqq \int_{B_\varrho^+(0)} |\nabla u|^2 \dx.
\end{equation}
Combining \eqref{eq:bdrymonocomparisonstep}, \eqref{eq:bdrymonocs}, and \eqref{eq:bdryder} we obtain
\[
 E_\varrho^+(u) \le \frac{(1+C\varrho\|\vp\|_{\Lip})\varrho}{n-2}\brac{ \frac{d}{d\varrho} E_\varrho^+(u) - \int_{S^+_\varrho}\left|\frac{\partial u}{\partial \nu}\right|^2 \dhn} + C\varrho^{n-1}\|\vp\|_{\Lip}.
\]
Now we can estimate the derivative 

\[\begin{split}
&\frac{d}{d \varrho} \left( (1 + C \varrho \|\vp\|_\Lip)^{n-2} \varrho^{2-n} E_\varrho^+(u) \right) \\
&= (n-2) (1 + C \varrho \|\vp\|_\Lip)^{n-3} \varrho^{1-n} 
\left( \frac{(1 + C \varrho \|\vp\|_\Lip)\varrho}{n-2} \frac{d}{d\varrho} E_\varrho^+(u) - E_\varrho^+(u) \right) \\
&\ge (n-2) (1 + C \varrho \|\vp\|_\Lip)^{n-3} \varrho^{1-n} 
\left( \frac{(1 + C \varrho \|\vp\|_\Lip)\varrho}{n-2} \int_{S^+_\varrho}\left|\frac{\partial u}{\partial \nu}\right|^2 \dhn - C\varrho^{n-1}\|\vp\|_{\Lip} \right) \\
&\ge \varrho^{2-n} \int_{S^+_\varrho}\left|\frac{\partial u}{\partial \nu}\right|^2 \dhn - C (1 + C \varrho \|\vp\|_\Lip)^{n-3} \| \vp \|_\Lip
\end{split}
\]
and integrate it from $r$ to $R$: 
\[
(1+C\varrho\|\vp\|_{\Lip})^{n-2} \varrho^{2-n} E_\varrho^+(u) \Big|_r^R
\ge 
\int_{B_R^+(0)\setminus B_r^+(0)}|x|^{2-n}\left|\frac{\partial u}{\partial \nu}\right|^2 \dx 
- C\|\vp\|_{\Lip}(R-r).
\]
This finishes the proof in the case of the flat boundary.

\textsc{Step 2 (general boundary).} The general case, when the boundary of $\Omega$ is not necessary flat, differs from the flat case in a few details (see for example \cite[Proof of Lemma 5.6]{HL1987}), which we sketch here. 

Choose a point $a \in \pl \Omega$ at the boundary. As described already, up to a rigid motion, we may assume that $a = 0$ and $\pl \Omega \cap \B_{R_0}$ is described by a $C^1$ function $\alpha \colon \R^{n-1} \to \R$: 
\[
\Omega \cap \B_{R_0} = \{ x \in \B_R \colon x_n > \alpha(x') \},
\]
with $\alpha(0) = 0$, $\nabla \alpha(0) = 0$. Then $\varsigma(x',x_n) = (x',x_n-\alpha(x'))$ is a $C^1$-diffeomorphism for which $\varsigma(\pl \Omega \cap \B_{R_0}) \subseteq \R^{n-1} \times 0$. 

Again, we extend the map $\vp$ to the whole domain by letting it be constant on vertical lines. However, since the rays from $0$ to $\Omega \cap \pl B_{\varrho}$ possibly cross the boundary, the definition of $\tilde{v}_\varrho$ needs to be altered. We solve this problem by considering the curves $t \mapsto \varsigma^{-1}(t \cdot \varsigma(x))$ instead. Hence, we take: 
\[
\tilde{v}_\varrho \coloneqq (u-\varphi)\brac{\varsigma^{-1}\brac{\varrho(x)\frac{\varsigma(x)}{|\varsigma(x)|}}}+\varphi(x).
\]
Note that we also replaced $\varrho$ by $\varrho(x)$, as the image $\varsigma(\Omega \cap \pl B_{\varrho})$ is no longer a~part of a~sphere. One can define $\varrho(x)$ as the length of the ray from $0$ to $\varsigma(\Omega \cap \pl B_{\varrho})$ passing through $x$. 

By previous considerations, $\varsigma$ is $C^1$-close to identity, in consequence also $\varrho(x)$ is $C^1$-close to~$\varrho$. With some care, one can check that this altered version of $\tilde{v}_\varrho$ also satisfies the required estimates.
\end{proof}

\subsection{Tangent maps at the boundary}
Similarly as in the interior case, for any point at the boundary $y\in \partial\Omega$ we can consider boundary tangent maps at $y$, which will arise as limits of rescaled mappings. Here we will consider only the case when $\Omega = B_1^+(0)$  has a~flat boundary. 

\begin{lemma}[Boundary tangent maps, {\cite[lemma 1.4]{SU2}}]\label{la:boundarytangentmaps}
 Let $u\in W^{1,2}(B_1^+(0), \n)$ be a~minimizing harmonic map with $u = \varphi$ on $T_1$ and let $\varphi$ be continuous at $a\in T_1$. Then there exists a sequence $\{\lambda_i\}$, $\lambda_i\searrow 0$ and a map $\Phi\in W^{1,2}(B_1^+(0),\n)$ such that 
 \[
  u_{\lambda_i,a}(\cdot)\coloneqq u(\lambda_i\cdot + a) \rightarrow \Phi(\cdot) \quad \text{in } W^{1,2}(B_1^+(0). \n)
 \]
Moreover, $\Phi$ is a minimizing harmonic map, homogeneous of degree 0, and $\Phi\big\rvert_{T_1} = u(a) = \varphi(a)$.

We call $\Phi$ a boundary tangent map to $u$ at $a$.
\end{lemma}

\begin{proof}
First we note that
\[
\int_{B_1^+(0)}|\nabla u_{\lambda_i,a}|^2\dx= \int_{B_1^+(0)}|\nabla (u(\lambda_ix+a))|^2 \dx = \lambda_i^{2-n}\int_{B_{\lambda_i}^+(a)} |\nabla u|^2 \dd y. 
\]
By the boundary monotonicity formula, Theorem~\ref{th:bdmonotonicityformula}, we know that $\sup_{\lambda_i>0}\lambda_i^{2-n}\int_{B_{\lambda_i}^+(a)} |\nabla u|^2 \dd y <\infty$, thus $\sup_{\lambda_i>0}\|u_{\lambda_i,a}\|_{W^{1,2}(B_1^+(0))}<\infty$. 
The proof of strong convergence follows from Theorem \ref{th:bsc}~\eqref{it:stongconvergenceboundary}, the proof of homogeneity follows from the proof of Lemma \ref{la:tangentmaps} with the monotonicity formula replaced by the boundary monotonicity inequality \eqref{eq:boundarmonotonicityformula}. Finally, $\Phi\big\rvert_{T_1} = \varphi(a)$ follows from the continuity of $\varphi$.
\end{proof}

The following result, due to \cite{SU2}, states that there exist no nonconstant boundary tangent maps. This is the main reason why at the boundary we have full regularity for certain boundary data, see Section \ref{sec:boundaryregularity}. We refer the reader for the proof of this fact to \cite[Theorem 5.7]{HL1987}.

\begin{lemma}[Nonexistence of nonconstant boundary tangent maps]\label{la:bdtangentmaps}
Assume that $u\in W^{1,2}(B^+_1(0),\n)$ is a minimizing harmonic map, homogeneous of degree 0 and constant at the flat part of the boundary, i.e., $u\big\rvert_{T_1} = const$. Then $u$ is constant.
\end{lemma}

\section{Refined estimates by Naber and Valtorta}
\label{sec:NabVal}

Here we discuss the results of Naber and Valtorta \cite{NabVal17} needed in the sequel. A simplified presentation of these is available in their later article \cite{NabVal18}.

The main ingredient is \Cref{th:NabVal-general-meas-bound}.  
%
%
In the special case of manifolds $\n$ with finite fundamental group, uniform boundedness of minimizers, Theorem~\ref{th:uniformboundedness}, implies that the energy assumption in \Cref{th:NabVal-general-meas-bound} is redundant. 

\begin{corollary}
\label{co:NabVal-meas-bound}
If $u \colon B_{2r}(y) \to \n$ is energy minimizing and $\pi_1(\n)$ is finite, then $\H^{n-3}(\sing u \cap B_r(y)) \le C r^{n-3}$ with some constant $C(n,\cN) > 0$.

In particular, whenever $\Omega' \subset \subset \Omega$ and $u$ is a minimizing harmonic map on $\Omega$, then
\[
 \mathcal{H}^{n-3}(\sing u \cap \Omega') < \infty.
\]
\end{corollary}

In order to prove the stability theorem, Theorem~\ref{th:stabilitynD}, one needs more subtle measure estimates. Note that for the tangent map $\Psi$, the singular set is an $(n-3)$-plane and so $\H^{n-3}(\sing \Psi \cap B_r) = \omega_{n-3} r^{n-3}$. If $u$ is close to $\Psi$, one could expect its singular set to have similar measure, see Lemma~\ref{lem:local-stability}. To this end, we will need two more results, which are essential ingredients of \cite{NabVal17}. 

To state them, we first recall the definition of Jones' height excess $\beta$-numbers. Choosing a~Borel measure $\mu$ in $\R^n$, a dimension $0 < k < n$ and an exponent $p \ge 1$, we can define for each ball $B_r(x)$ 
\[
\beta_{\mu,k,p} \coloneqq \inf_{L} \left( r^{-k-p} \int_{B_r(x)} \dist(y,L)^p \dd \mu (y) \right)^{1/p},
\]
where the infimum is taken over all $k$-dimensional affine planes $L \subset \R^n$. This measures how far the support of $\mu$ is from a $k$-dimensional plane (on the ball $B_r(x)$). However, we shall not work directly with this definition, but rather rely on the two theorems below, since they encompass all the geometric information we need.

The first theorem is a general geometric result that gives sharp measure estimates. 

\begin{theorem}[{Rectifiable Reifenberg {\cite[Theorem~3.3]{NabVal17}}}]
\label{th:NabVal-reifenberg}
For every $\eps > 0$ there is a~$\delta = \delta(n,\eps) > 0$ such that the following holds. Let $S \subset \R^n$ be an $\H^{k}$-measurable subset and assume that for each ball $B_r(x) \subset B_2$ 
\[
\int_{B_r(x)} \int_0^r \beta_{\mu,k,2}(y,s)^2 \frac{\dd s}{s} \dd \mu(y) 
\le \delta r^{k}, 
\]
where $\mu$ denotes the measure $\H^{k}\llcorner S$. Then $\mu(B_1) \le (1+\eps) \omega_k$. 
\end{theorem}

As a side remark, let us note that in our application the set $S$ will satisfy the so-called Reifenberg condition and so one could work with the $W^{1,p}$-Reifenberg theorem \cite[Theorem~3.2]{NabVal17} instead. 

\begin{theorem}[{$L^2$-best approximation {\cite[Theorem~7.1]{NabVal17}}}]
\label{th:NabVal-best-approx}
For every $\eps > 0$ there are $\delta(n,\eps) > 0$ and $C(n,\eps) > 0$ such that the following holds. If $u \colon B_{10}(0) \to \S^2$ is energy minimizing, 
\begin{align*}
\dist_{L^2(B_{10}(0))} (u, \ \mathrm{sym}_{n,0}  ) & \le \delta, \\
\dist_{L^2(B_{10}(0))} (u, \ \mathrm{sym}_{n,k+1}) & \ge \eps, 
\end{align*}
then for any finite measure $\mu$ on $B_1(0)$ we have 
\[
\beta_{\mu,k,2} (0,1)^2
\le C \int_{B_1(0)} \left( \theta_u(y,8) - \theta_u(y,1) \right) \dd \mu(y).
\]
\end{theorem}
Again, the formulation in \cite{NabVal17} involves an energy bound. However, Theorem~\ref{th:uniformboundedness} shows a~uniform bound on $\int_{B_9(0)} |\nabla u|^2 \dx$ and thus we obtain the stronger formulation above. 

Since we shall only consider $k = n-3$, $p = 2$ and $\mu = \H^{n-3} \llcorner \sing u$ from now on, we abbreviate $\beta_{\mu,n-3,2}$ by $\beta$; this should not cause any confusion. 

\section{Extension property and its consequences}\label{s:extension}
In this section, we collect the results concerning local uniform boundedness of minimizers into manifolds with finite fundamental group. One of its many consequences is that every sequence of minimizing harmonic maps has a subsequence that converges locally weakly (and in fact strongly, see Theorem~\ref{th:bsc}). 

We recall that $\n$ is a smooth closed connected Riemannian manifold. We let $\widetilde \n$ be its universal covering, and $\pi\in C^\infty(\widetilde \n,\n)$ be a Riemannian covering. We recall that $\widetilde \n$ is compact if and only if the fundamental group of $\n$ is finite, we also recall that $\pi_1(\widetilde\n) =0$, whereas the higher order homotopy groups of $\n$ and $\widetilde \n$ are the same. Of course if $\n$ is simply connected then $\widetilde \n = \n$ and $\pi = \id$. For further properties of the universal covering we refer the interested reader, e.g., to \cite{Hatcher}. From now on, we assume that $\widetilde \cN$ is isometrically embedded into some Euclidean space, which is possible by Nash's embedding theorem. 

We will use the following lifting theorem of Bethuel and Chiron.

\begin{theorem}[{\cite[Theorem 1]{BethuelChiron}}]\label{th:BethuelChiron}
If $\Omega$ is simply connected (e.g., a ball) and $u\in W^{1,2}(\Omega,\n)$, then there exists $\tilde u \in W^{1,2}(\Omega, \widetilde \n)$ such that $u= \pi \circ \tilde u$. Moreover, this $\tilde u$ is unique up to the action of an element of $\pi_1(\n)$ and satisfies a.e. $|\nabla \tilde u| = |\nabla u|$ .
\end{theorem}

We note also that in \cite[Remark 6.2]{AlmgrenLieb1988} it was noted that the results of the paper \cite{AlmgrenLieb1988} can be translated, using a lifting argument, into the case when the target manifold is $\n = \mathbb{RP}^2$ (i.e., to the case of a manifold which is not simply connected\footnote{Simply connected manifolds are manifolds that satisfy $\pi_0(\n) = \pi_1(\n)=0$.}, but has a finite fundamental group).

\subsection{Extension lemma}\label{ss:extension}

An underlying tool used in \cite{AlmgrenLieb1988} is the following $W^{1,2}$-extension property of $W^{1/2,2}$ maps into simply connected manifolds $\n$. This is a result of  Hardt and Lin \cite[Theorem 6.2]{HL1987}, which holds for even more general class of $W^{1,p}$-maps into $\lfloor p -1\rfloor$-connected manifolds\footnote{$k$-connected manifolds are manifolds that satisfy $\pi_0(\n) = \pi_1(\n) = \ldots = \pi_{k}( \n) = 0$.}. However, it was first published (and acknowledged to \cite{HL1987}) in the paper by Hardt--Kinderlehrer--Lin \cite[p.556]{HardtKinderlehrerLin1986} (see also \cite[Lemma A.1]{HardtKinderlehrerLinstable}), where it was stated for $\n = \S^2$. This result was also extended by Gastel \cite[Proposition 4.3]{Gastel} for $W^{k,p}$ maps and manifolds with sufficiently simple topology (see also the earlier extension to the case $\n = \S^{d-1}$ and $k,\,p=2$ by Hong and Wang \cite[Lemma 2.1]{Hong-Wang}). 

\begin{theorem}[Extension Property: unconstrained energy dominates constrained energy]\label{th:extensionthm}
Let $\Omega \subset \R^n$ be a bounded domain and $\n \subset \R^d$ be a simply connected submanifold. Assume that we have a map $v \in W^{1,2}(\Omega,\R^d)$ with $v(x)\in\n$ for a.e. $x\in\partial \Omega$. Then there exists a~map $u \in W^{1,2}(\Omega,\n)$, 
\[
 u \Big\rvert_{\partial \Omega} =  v\Big |_{\partial \Omega}
\]
with the estimate 
\[
 \|\nabla u\|_{L^2(\Omega)}\leq C\,\|\nabla v\|_{L^2(\Omega)}
\]
for a uniform constant $C$. 
\end{theorem}

\begin{remark}\label{rem:noextension}
We note that, unlike in the unconstrained case, it is not true for a general target manifold $\n$ that every boundary map $\varphi\in W^{1/2,2}(\partial \Omega,\n)$ has an extension in $W^{1,2}(\Omega,\n)$. A counterexample was provided by Hardt and Lin \cite[Section 6.3]{HL1989}. They prove that the map $\psi\in W^{1/2,2}(\S^2,\S^1)$ given by $\psi(x_1,x_2,x_3) = \frac{(x_1,x_2)}{|(x_1,x_2)|}$ cannot be extended to a map in $W^{1,2}(B^3_1(0),\S^1)$. 

More generally, every $\vp\in W^{\frac12,2}(\partial \Omega, \n)$ can be extended to a map $u\in W^{1,2}(\Omega,\n)$ with $u\big\rvert_{\partial \Omega} = \vp$ (in the trace sense) if and only if $\pi_1(\n) = 0$ (\cite[Theorem 4]{Bethuel-Demengel}).

For general obstructions for the existence of an extension and discussion of traces of manifold valued Sobolev maps we refer the reader to \cite{Bethuel-traces, mironescu2020trace}.
\end{remark}

We recall the key ingredient in the proof of Theorem \ref{th:extensionthm}. 

\begin{lemma}[{\cite[Lemma 6.1]{HL1987}}]\label{la:retraction}
Assume that $\n \subset \R^d$ is a closed simply connected submanifold, contained in a large cube $[-R,R]^{\d}$. Then, there exists a finite $(\d-3)$-dimensional Lipschitz complex $Y \subset [-2R,2R]^\d$ and a locally Lipschitz retraction $P \colon [-2R,2R]^\d \setminus Y \to \n$ such that
\begin{enumerate}
 \item for some small $\varrho > 0$, the restriction $P \big \rvert_{B_\varrho(\n)}$ is the nearest point projection to $\n$;
 \item $|\nabla P(x)| \le C \, \dist(x,Y)^{-1}$ for $x \notin Y$;
 \item $\int_{[-2R,2R]^\d} |\nabla P(x)|^2 \dd x < \infty$.
\end{enumerate}
\end{lemma}

\begin{remark}
Such a retraction is easy to construct if the target manifold is the standard sphere $\S^{d-1} \subset \R^d$ with $\d\ge 3$. Indeed, one can check that the radial projection $P \colon \R^d \setminus \{0\} \to \S^{\d-1}$, $P(x)=\frac{x}{|x|}$ gives the retraction from Lemma \ref{la:retraction}.
\end{remark}

%

\begin{proof}[Proof of Lemma \ref{la:retraction}]
Denote the $\varrho$-neighborhood of $\cN$ by 
\[
B_{\varrho}(\n) = \{ x \in \R^d \colon \dist(x,\cN) < \varrho \}; 
\]
let us fix some small $\varrho > 0$ for which the nearest point projection $\pi_\n \colon B_{2\varrho}(\n) \to \n$ is well-defined and smooth. We also fix the decomposition $\cQ$ of $[-2R,2R]^d$ into $(4N)^\d$ cubes of side length $R/N$, where $N$ is chosen large enough to ensure that the family of cubes 
\[
\cU \coloneqq \left \{ Q \in \cQ \colon Q \subset B_{2\varrho}(\n) \right \} 
\]
covers $B_{\varrho}(\n)$. For $j = 0,1,\ldots,\d$, we let $\cQ^j$ be the family of $j$-dimensional faces of cubes in $\cQ$ (in particular, $\cQ^\d = \cQ$) and distinguish the $j$-cubes not covered by $\cU$: 
\[
 \cW^j \coloneqq \left\{ Q \in \cQ^j \colon Q \nsubseteq \bigcup \cU \right\}.
\]
Note that all cubes in $\cQ$ are partitioned into $\cU$ and $\cW^\d$, hence $\bigcup \cU \cup \bigcup \cW^\d = [-2R,2R]^\d$. 

\medskip

We now proceed with an inductive contruction of maps 
\[
 \Upsilon^j \colon \mathcal W^j \cup \bigcup \mathcal U \to \n 
 \quad \text{for } j=0,\ldots,\d 
\]
that extend the nearest point projection $\pi_\cN \colon \bigcup \cU \to \cN$. 

A well-known construction (see \cite[Lemma 4.7]{Hatcher}) shows that continuous maps into simply connected manifolds can be continuously extended to $2$-dimensional CW complexes. We use it here to extend $\pi_\cN \colon \bigcup \cU \to \cN$ to a Lipschitz map
\[
 \Upsilon^2 \colon \cW^2 \cup \bigcup \mathcal U \to \n.
\]
Fixing an arbitrary point $\xi \in \cN$, we define $\Upsilon^0 (x) \coloneqq \xi$ for all $\{ x \} \in \cW^0$. Since $\cN$ is path-connected, we can define $\Upsilon^1$ on each segment $Q \in \cW^1$ by taking the geodesic in $\cN$ (or any other Lipschitz curve) joining the two points $\Upsilon^0(\partial Q)$. Finally, since $\pi_1(\cN)=0$, for each $Q \in \cW^2$ the Lipschitz boundary map $\Upsilon^1 \colon \partial Q \to \cN$ is homotopically trivial and can be extended to a Lipschitz map $\Upsilon^2 \colon Q \to \cN$. The Lipschitz constant of $\Upsilon^2$ can be bounded depending on the submanifold $\cN \subset \R^d$ only.

\medskip

Without additional topological assumptions on $\cN$ (i.e., vanishing of higher-dimensional homotopy groups), we cannot extend $\Upsilon^2$ continuously to a higher-dimensional complex, yet we are able to define singular extensions $\Upsilon^j$ for $j=3,\ldots,\d$ such that 
\begin{equation}
\label{eq:upsilon-estimate}
|\nabla \Upsilon^j(x)| \le C_j ( 1 + \dist(x,Y_j \cap Q)^{-1})
\quad \text{for each } Q \in \cQ^j \text{ and } x \in Q \setminus Y_j, 
\end{equation}
where $Y_j$ is a finite $(j-3)$-dimensional Lipschitz complex. 

Assuming that $\Upsilon^{j-1}$ satisfying \eqref{eq:upsilon-estimate} is already constructed, we only need to extend it to the interior of each $j$-cube $Q \in \cW^{j}$. To this end, we take the standard radial projection $P_{Q} \colon Q \setminus \{ y_Q \} \to \partial Q$, which is Lipschitz except for the center $y_Q$ of $Q$, and define $\Upsilon^j = \Upsilon^{j-1} \circ P_Q$ on $Q$. 

Let us check the estimate \eqref{eq:upsilon-estimate} for $\Upsilon^j$. The set of discontinuities of $\Upsilon^j$ in $Q$ is exactly 
\[
Y_j \cap Q = \{ y_Q \} \cup P_Q^{-1}(Y_{j-1} \cap \partial Q),
\]
which only adds one dimension to $Y_{j-1}$. Now choose a point $x \in Q \setminus Y_j$ and denote by $Q'$ the face of $Q$ onto which $x$ is projected; there are two possibilities. If $\Upsilon^{j-1}$ happens to be regular on $Q'$, we simply have 
\[
|\nabla \Upsilon^j(x)| \lesssim |\nabla P_Q(x)| \lesssim |x-y_Q|^{-1} \le \dist(x,Y_j \cap Q)^{-1}. 
\]
If $\Upsilon^{j-1}$ is not regular, the inductive assumption \eqref{eq:upsilon-estimate} gives us a weaker estimate 
\[
|\nabla \Upsilon^j(x)| 
\le |\nabla P_Q(x)| \cdot |\nabla \Upsilon^{j-1}(P_Q(x))| 
\lesssim |x-y_Q|^{-1} \cdot \dist(P_Q(x),Y_{j-1} \cap Q')^{-1}.
\]
Noting a similarity of triangles, we obtain 
\[
\frac{\dist(x,Y_j \cap Q)}{|x-y_Q|}
\le \frac{\dist(x,P_Q^{-1}(Y_{j-1} \cap Q'))}{|x-y_Q|} 
= \frac{\dist(P_Q(x),Y_{j-1} \cap Q')}{|P_Q(x)-y_Q|} 
\]
and since $|P_Q(x)-y_Q| \ge \frac{R}{2N}$, the inductive claim \eqref{eq:upsilon-estimate} follows.

\medskip

As the outcome of this inductive construction, we finally obtain the map $P = \Upsilon^{\d}$ which is locally Lipschitz outside the finite $(\d-3)$-dimensional Lipschitz complex $Y \coloneqq Y_\d$. Recall that $\bigcup \mathcal U \cup \bigcup \cW^\d = [-2R,2R]^\d$, and so $P$ is defined on the whole large cube. Since $Y$ is non-empty, one can drop the constant term in \eqref{eq:upsilon-estimate} and infer the Lipschitz estimate 
\[
|\nabla P(x)| \le C \, \dist(x,Y)^{-1}
\quad \text{for } x \notin Y.
\]
To derive the $W^{1,2}$-estimate, first observe that $Y$ is contained in a union of finitely many affine subspaces $L_1,\ldots,L_s$, each of dimension at most $\d-3$. Since 
\[
\int_{[-2R,2R]^\d} \frac{\dd x}{(\dist(x,L_i))^2} < \infty 
\quad \text{for } i=1,\ldots,s 
\]
by Fubini's theorem, it follows that 
\[
 \int_{[-2R,2R]^\d} |\nabla P (x)|^2 \dx \lesssim \int_{[-2R,2R]^\d} \frac{\dd x}{(\dist(x,Y))^2} < \infty. 
\]
Actually, it is this $W^{1,2}$-estimate that we shall use later. 
\end{proof}

With the retraction $P$ in hand, one could hope to simply take the map $P \circ v$ as the required extension. It may happen however that $v$ takes values very near the set of singularities $Y$, and in result $P \circ v$ does not lie in $W^{1,2}$. For this reason, we consider the family of maps $P_a(x) = P(x-a)$ (with singularities shifted to $Y+a$) and compose with $P_a$ for some generic value of $a$. The details of this construction are given below. 

\begin{proof}[Proof of Theorem \ref{th:extensionthm}]
Fitting $\cN \subset \R^d$ into some cube $[-R,R]^d$, we can choose $P$ and $\varrho$ as in Lemma \ref{la:retraction}. We also assume that $v$ takes values in $[-R,R]^d$; the general case follows by a simple reduction argument explained at the end of the proof. 

Since the restriction $\pi_\cN \colon \cN \to \cN$ is the identity and $\cN$ is compact, we can choose $0 < r < \varrho$ such that the maps 
\[
\cN \ni x \mapsto \pi_\cN(x-a) \in \cN 
\]
are uniformly bi-Lipschitz for all $a \in \B_r$. For $a \in \B_r$ we let $P_a(x) = P(x-a)$ and consider the composition $P_a \circ v$. Since 
\[
 \left| \nabla\brac{P_a\circ v}\right| \le |\nabla P_a(v)| \cdot |\nabla v|, 
\]
we can check that $P_a \circ v \in W^{1,2}(\Omega, \n)$ for almost every $a \in B_r(0)$. Indeed, 
\begin{equation*}\label{eq:mainestiamteextension}
\begin{split}
 \int_{B_r(0)}\int_\Omega |\nabla\brac{P_a \circ v}|^2\dx \dd a 
  & \le \int_{\Omega} |\nabla v|^2 \int_{B_r(0)} \left|\nabla P_a(v)\right|^2 \dd a \dx \\
  & \le \int_{\Omega} |\nabla v|^2 \int_{B_r(0)} \left|\nabla P(v-a)\right|^2 \dd a \dx \\
  & \le \int_{\Omega} |\nabla v|^2 \int_{[-2R,2R]^\d} \left|\nabla P (y)\right|^2 \dd y \dx\\
  & \le C \int_{\Omega} |\nabla v|^2  \dx < \infty
  \end{split}
 \end{equation*}
with the constant $C$ depending only on $\n$. We infer that there is $a_0 \in B_r(0)$ such that 
\begin{equation}\label{eq:extensionalmostfinal}
 \int_\Omega |\nabla\brac{P_{a_0} \circ v}|^2\dx \le C |B_r(0)|^{-1} \int_{\Omega} |\nabla v|^2 \dx.
\end{equation}
To fully justify that $\nabla P_a(v) \nabla v$ is indeed the distributional gradient, one should first carry out these estimates for a sequence of smooth maps $P_\eps \to P$ and then pass to the limit, but we leave this standard step to the reader.

Composing $v$ with $P_{a_0}$, we have unnecessarily altered $v$ on the set where it already maps into $\cN$, so we need a small correction. It follows from our previous discussion that $P_{a_0}\big\rvert_\cN$ is invertible and 
\begin{equation}\label{eq:inverseprojection}
 |\nabla \brac{P_{a_0}\big\rvert_{\n}}^{-1}| \le C
\end{equation}
for a uniform constant $C$. We set 
\[
 u = \brac{P_{a_0}\big\rvert_{\n}}^{-1} \circ P_{a_0} \circ v \in W^{1,2}(\Omega, \n), 
\]
hence combining \eqref{eq:extensionalmostfinal} with \eqref{eq:inverseprojection} gives us 
\[
 \int_{\Omega} |\nabla u|^2 \dx 
 \le \| \nabla \brac{P_{a_0}\big\rvert_{\n}}^{-1} \|_{L^\infty}^2 
 \int_{\Omega} |\nabla \brac{P_{a_0}\circ v}|^2 \dx 
 \le C \int_{\Omega} |\nabla v|^2 \dx. 
\]
It is evident that $u \equiv v$ on the whole set $\{ v(x) \in \cN \}$, and this will be used in the sequel. One can also check that $u\big\rvert_{\partial \Omega} =  v\big\rvert_{\partial \Omega}$ follows in the sense of traces. 

\medskip

For the general case of an unbounded $v$, consider the retraction $P_R \colon \R^d \to [-R,R]^\d$ given by the identity on $[-R,R]^\d$ and by the radial projection outside of this cube; note that $P_R$ is Lipschitz with constant $1$. One can then apply the previous construction to $\overline{v} \coloneqq P_R \circ v$ and replace $\overline{v}$ by $v$ in the final claim, as these two maps agree on $\partial \Omega$ and the $W^{1,2}$-energy of $\overline{v}$ does not exceed that of $v$. 
\end{proof}

\subsection{Uniform boundedness}\label{ss:unifromboundedness}
We will show how Theorem~\ref{th:extensionthm} combined with trace inequalities \Cref{th:trace} implies a uniform bound for the minimizers. Due to the lifting theorem \Cref{th:BethuelChiron} of Bethuel and Chiron this result holds additionally for target manifolds whose fundamental group is finite. 

This is a counterpart of {\cite[Theorem 1.1]{AlmgrenLieb1988}}.

\begin{corollary}\label{co:est}
Assume that $\pi_1(\n)$ is finite and that $\Omega$ is a bounded Lipschitz domain. There exists a constant $C(\Omega,\n)$ such that for any minimizing harmonic map $u\in W^{1,2}(\Omega,\n)$ we have 
\begin{equation}\label{eq:traceestimateLipdom}
 \|\nabla u\|_{L^2(\Omega)} \leq C(\Omega,\n) \sqrt{\|\nabla_T u\|_{L^2(\partial \Omega)}}.
\end{equation}
In particular, if $u \colon \B_r \to \n$ is a minimizing harmonic map, then the following estimate holds
\begin{equation}\label{eq:traceextension}
 \|\nabla u\|_{L^2(B_r)} \aleq \sqrt{r^{\frac{n-1}{2}}\|\nabla_T u\|_{L^2(\partial B_r)}}.
\end{equation}
Let $\Omega \subset \R^n$ be a bounded domain with a $C^1$-boundary and let $y_0\in\partial \Omega$ be any point on the boundary. If $u \colon B_r(y_0)\cap\Omega \to \n$ is a minimizing harmonic map with $u=\varphi$ on $B_r(y_0)\cap \partial \Omega$, then the following estimates hold
\begin{equation}\label{eq:traceestimatedonhalfball}
 \|\nabla u\|_{L^2(B_r(y_0)\cap\Omega)} \aleq \sqrt{r^{\frac{n-1}{2}}\|\nabla_T u\|_{L^2(\partial B_r(y_0)\cap\Omega)} +  r^{\frac{n-1}{2}}\|\nabla \varphi\|_{L^{2}(B_r(y_0)\cap\partial \Omega)}}
\end{equation}
and for any $s> \frac 12$, $p>1$, $sp>1$, and $1<\theta<2$
\begin{equation}\label{eq:traceestimatedonhalfballWsp}
\|\nabla u\|_{L^2(B_r(y_0)\cap\Omega)} \aleq \sqrt{r^{(3-n)\frac{\theta}{2}+(n-2)}\|\nabla_T u\|^{\theta}_{L^2(\partial B_r(y_0)\cap\Omega)}+r^{(sp-(n-1))\frac{\theta}{sp} + n-2}[\varphi]^\frac{\theta}{s}_{W^{s,p}(L^{2}(B_r(y_0)\cap\partial \Omega))}}.
\end{equation}

\end{corollary}
\begin{proof}
In order to prove \eqref{eq:traceestimateLipdom} we apply \Cref{th:BethuelChiron} to $u\in W^{1,2}(\Omega,\n)$, we obtain the existence of a map $\tilde u \in W^{1,2}(\Omega,\widetilde \n)$ such that $u= \pi \circ \tilde u$. Now, since $\pi_1(\n)$ is finite we know that the universal cover $\widetilde \n$ is compact and since it is the universal cover of $\n$ we have $\pi_1(\widetilde \n) = 0$. Thus, we may apply \Cref{th:extensionthm} to $\tilde u$ and obtain the existence of a manifold valued map $\tilde v \in W^{1,2}(\Omega, \widetilde \n)$ satisfying $\tilde v \big\rvert_{\partial \Omega} = \tilde u \big\rvert_{\partial \Omega}$ with the estimate
 \begin{equation}\label{eq:uniformbddextensionthm}
  \|\nabla \tilde v\|_{L^2(\Omega)} \aleq \|\nabla \tilde u^h\|_{L^2(\Omega)},
 \end{equation}
where $\tilde u^h$ is the harmonic extension of $\tilde u\big\rvert_{\partial \Omega}$. 
By \Cref{th:trace}, (equation \eqref{eq:tracespherercombined}), we know that
\begin{equation}\label{eq:tracethminunibdd}
 \|\nabla \tilde u^h\|_{L^{2}(\Omega)} \leq C(\Omega) \|\tilde u\|_{L^\infty(\partial \Omega)}^{\frac{1}{2}} \|\nabla_T \tilde u\|_{L^2(\partial \Omega)}^{\frac{1}{2}} 
 \le C(\Omega, \widetilde \n) \|\nabla_T u\|_{L^2(\partial \Omega)}^{\frac{1}{2}},
\end{equation}
in the last estimate we used that $\widetilde \n$ is compact, $\tilde u\big\rvert_{\partial \Omega} \in \widetilde \n$, and that a.e. $|\nabla u| = |\nabla \tilde u|$.  

On the other hand we set $v = \pi \circ \tilde v \in W^{1,2}(\Omega, \n)$ and we have $|\nabla v| = |\nabla \tilde v|$ a.e., thus,
\[
 \|\nabla v\|_{L^2(\Omega)} = \|\nabla \tilde v\|_{L^2(\Omega)}.
\]
Since $\pi$ is a local isometry, we have $v\big\rvert_{\partial \Omega} = \pi \circ \tilde v \big \rvert_{\partial \Omega} = \pi \circ \tilde u \big \rvert_{\partial \Omega} = u \big\rvert_{\partial \Omega}$ and thus $v$ is a good comparison map and we obtain from the minimality of $u$
\[
 \|\nabla u\|_{L^2(\Omega)} \le \| \nabla v\|_{L^2(\Omega)}.
\]
Combining this with the above inequalities we obtain
\[
 \|\nabla u\|_{L^2(\Omega)} \le \|\nabla v\|_{L^2(\Omega)} = \|\nabla \tilde v\|_{L^2(\Omega)} \aleq \|\nabla \tilde u^h\|_{L^2(\Omega)} \aleq \|\nabla_T u\|^\frac12_{L^2(\Omega)}.
\]
This finishes the proof of \eqref{eq:traceestimateLipdom}. The estimate \eqref{eq:traceextension} follows from scaling. 

We proceed similarly to obtain the other inequalities: we apply \Cref{th:extensionthm} to the lifted map $\tilde u = \pi^{-1} \circ u\in W^{1,2}(B_r^+,\widetilde \n)$, then \eqref{eq:traceestimatedonhalfball} follows from \eqref{eq:tracespherercombined} and \eqref{eq:traceestimatedonhalfballWsp} follows from \eqref{eq:tracehalfspherecombined}. 
\end{proof}

\begin{theorem}[Uniform Boundedness of Minimizers, {\cite[Theorem 2.3 (2)]{AlmgrenLieb1988}}]\label{th:uniformboundedness}
Let $\pi_1(\n)$ be finite. Then the following assertions hold:
\begin{enumerate}
 \item \label{it:uniformboundinterior} Let $u\in W^{1,2}(B_R,\n)$ be a minimizing harmonic map. Then for any $r < R$,
\[
r^{2-n} \int_{B_r} |\nabla u|^2 \dx\leq C\frac{R}{R-r},
\]
where $C(n,\n)$ is an absolute constant.
 
 \hfill
 \item \label{it:uniformboundboundary}Let $\Omega$ be a bounded domain with a $C^1$ boundary. Then, there exists an $R \coloneqq R(\Omega)$ such that the following holds: assume $u \in W^{1,2}(B_{2r}(y_0) \cap \Omega,\n)$ is a minimizing harmonic map, where $y_0 \in \partial \Omega$, then
 \begin{equation}\label{eq:uniformboundbdry}
 r^{2-n}\int_{B_{r}(y_0) \cap \Omega} |\nabla u|^2 \dx \leq C \, \left( 1 + r^{\frac{3-n}{2}}\| \nabla \vp \|_{L^{2}(B_{2r}(y_0) \cap \partial \Omega)} \right),
\end{equation}
for $0 < r \le R$. 

\hfill
 \item \label{it:uniformboundfractionalboundry}The last statement can be strengthened, assume additionally that $s>\frac12$, $p>1$, and $p>1$. Then for any $1<\theta<2$,
 \begin{equation}\label{eq:uniformboundbdryfractional}
 r^{2-n}\int_{B_{r}(y_0) \cap \Omega} |\nabla u|^2 \dx \leq C \, \left( 1 + r^{(sp-(n-1))\frac{\theta}{sp}}[\varphi]^\frac{\theta}{s}_{W^{s,p}(B_{2r}(y_0) \cap \partial \Omega)} \right),
\end{equation}
for $0 < r \le R$. 
\end{enumerate}
\end{theorem}

\begin{proof}
	Let $D(\rho)\coloneqq \int_{B_\rho(y_0)\cap \Omega} |\nabla u|^2 \dd x$, where $y_0 \in \overline{\Omega}$. This is an absolutely continuous function of $r\in[0,R]$ and by the fundamental theorem of calculus we have for a.e. $r\in [0,R]$
	\begin{equation}\label{eq:derivationofderivativeovervariableboundary}
	\begin{split}
	\frac{d}{d\rho} \brac{\int_{B_\rho(y_0)\cap \Omega} |\nabla u|^2 \dd x}
	&= \frac{d}{d\rho} \brac{\int_{B_\rho(y_0)} |\nabla u|^2 \chi_{\Omega}\dd x}\\
	&= \frac{d}{d\rho} \brac{\int_{0}^\rho \int_{\pl B_r(y_0)} |\nabla u(\xi)|^2 \chi_{\Omega}(\xi) \dhn(\xi) \dd r}\\
	&= \int_{\partial B_\rho(y_0)\cap \Omega}|\nabla u|^2 \dhn.
	\end{split}
	\end{equation}

	\underline{\textsc{Proof of \eqref{it:uniformboundinterior}}}:
	
By minimality of $u$ we may apply \Cref{{co:est}} (equation \eqref{eq:traceextension}) and we get
\[
D(\rho) \coloneqq \int_{B_\rho} |\nabla u|^2 \dx \leq  C\ \rho^{\frac{n-1}{2}}\brac{\int_{\partial B_\rho} |\nabla_T u|^2 \dhn}^{\frac{1}{2}},
\]
for a uniform constant $C>0$. This gives
\[
 D(\rho) \leq \rho^{\frac{n-1}{2}}\, C\, \sqrt{D'(\rho)}.
\]
Taking the square, we obtain
\[
 \frac{1}{C^2} \rho^{1-n}\leq \frac{D'(\rho)}{D(\rho)^2}.
\]
Integrating the last inequality from $r$ to $R$ we obtain 
\[
 (n-2)\left (r^{2-n} - R^{2-n} \right )\frac{1}{C^2} \leq \frac{1}{D(r)} - \frac{1}{D(R)}
\]
In particular,
\[
 r^{2-n}D(r) \leq \frac{C^2}{n-2} \frac{R^{n-2}}{R^{n-2}-r^{n-2}} \le \frac{C^2}{n-2} \frac{R}{R-r}.
\]
\hfill

\underline{\textsc{Proof of \eqref{it:uniformboundboundary}}}:

 Denote $D(\rho) \coloneqq \|\nabla u\|_{L^2(B_r(y_0)\cap \Omega)}^2$ and $A \coloneqq r^{\frac{n-1}{2}} \| \nabla \vp \|_{L^{2}(B_{2r}(y_0)\cap \partial \Omega)}$. Using \eqref{eq:derivationofderivativeovervariableboundary} we can restate \eqref{eq:traceestimatedonhalfball} as 
\[
 D(\rho) \le C \left( \rho^{\frac{n-1}{2}} \sqrt{D'(\rho)} + A \right) 
 \qquad \text{for } 0 < \rho \le 2r. 
\]
Since our aim is an estimate $D(r) \lesssim r^{n-2} + A$, we may assume that $D(r) \ge 2CA$ with $C$ as above. Then 
\[
 D(\rho) \le 2C \rho^{\frac{n-1}{2}} \sqrt{D'(\rho)}
 \qquad \text{for } r \le \rho \le 2r. 
\]
Rewriting this as the differential inequality $(-D(\rho)^{-1})' \ge 4 C^{-2} \rho^{1-n}$ and integrating, we obtain 
\[
D(r)^{-1} - D(2r)^{-1} \ge 4 C^{-2} \int_r^{2r} \rho^{1-n} \dd \rho.
\] 
The final claim now follows from $D(2r)^{-1} \ge 0$. 

\underline{\textsc{Proof of \eqref{it:uniformboundfractionalboundry}}}:

 The proof follows almost exactly as the proof of \eqref{it:uniformboundboundary}. The only difference is that in place of \eqref{eq:traceestimatedonhalfball} we use \eqref{eq:traceestimatedonhalfballWsp} which leads us to the inequality:
 \[
 D(\rho) \le C \left( \rho^{(3-n)\frac{\theta}{2}+(n-2)} \brac{D'(\rho)}^{\frac{\theta}{2}} + B \right) 
 \qquad \text{for } 0 < \rho \le 2r,   
 \]
where $B \coloneqq r^{(sp-(n-1))\frac{\theta}{s}+n-2}[\varphi]^p_{W^{s,p}(B_{2r}(y_0)\cap \partial \Omega)}$. Reasoning as before and having in mind that by assumption $1<\theta<2$ we may rewrite the latter inequality as
\[
 \brac{D(\rho)^{-\frac{2}{\theta}+1}}'\ge \tilde C \rho^{3-n - \frac{2}{\theta}(2-n)}.
\]
Integrating over $(r,2r)$, we finish the proof.
\end{proof}

\begin{remark}\label{rem:gentargetnouniformbound}
Theorem~\ref{th:uniformboundedness} does not hold for general target manifolds, it is not true for example for $\S^1$ and $\mathbb T^2$. A simple counterexample is due to Hardt--Kinderlehrer--Lin \cite[p.22]{HardtKinderlehrerLinstable}:
the energy minimizers $u_j\in W^{1,2}(B^n,\S^1)$, $u_j(x) = (\cos jx_1, \sin jx_1)$ have unbounded energies on each subdomain. 
\end{remark}

\subsection{Caccioppoli inequality and higher local integrability}\label{ss:caccioppoli}

In this section we derive the Caccioppoli inequality for minimizing harmonic maps. A Caccioppoli type inequality was obtained by Schoen and Uhlenbeck in their pioneering work \cite[Lemma 4.3]{SU1} in order to obtain strong convergence of minimizers. Later on, this result was generalized by Hardt and Lin to the case of minimizing $p$-harmonic maps \cite[Corollary 2.3]{HL1987}. As observed there, in case the target manifold is simply connected, the result might be strengthened --- the small oscillation condition from \cite[Lemma 4.3]{SU1} can be omitted (see also \cite[Lemma 2.3]{HardtKinderlehrerLin1986}). Finally, thanks to Luckhaus Lemma \cite[Lemma 1]{L88}, it was proved that the smallness condition can be also omitted for a general target manifold $\n$. We refer the interested reader to \cite[Section 2.8, Lemma 1]{Simon1996}.

The $W^{1,2}$-extension property (Theorem \ref{th:extensionthm}) will play here a~crucial role as it provides a~tool to compare energies of maps that agree on the boundary but do not have to take values in the manifold. We remark that this could also be done using Luckhaus Lemma. We will use a variant of an iteration lemma, see \cite[Chapter V, Lemma 3.1]{Giaquinta}.

\begin{lemma}[Iteration lemma]\label{la:iteration}
 Let $0 \le a < b < \infty$ and $f \colon [a,b] \to [0,\infty)$ be a bounded function. Suppose that there are constants $\theta \in (0,1)$, $A,\,\Lambda,\,\alpha > 0$ such that 
\begin{equation}\label{eq:iterationinequality}
 f(s) \le \theta f(t) + \frac{A}{(t-s)^\alpha}+ \Lambda 
 \quad \text{for all } a \le s < t \le b.
\end{equation}
Then we obtain the bound 
\[
 f(r)\le C \frac{A}{(b-r)^\alpha} + \frac{\Lambda}{1-\theta} 
 \quad \text{for all } a \le r < b 
\]
with some constant $C(\theta,\alpha) > 0$. 
\end{lemma}

\begin{proof}
We fix $a \le r < b$ and define the sequence $r_i$ by $r_0 = r$ and $r_{i+i} - r_i = (1-\tau)\tau^i(b-r)$, with $\tau \in (0,1)$ to be chosen later. By iterating the inequality \eqref{eq:iterationinequality} $\ell$-times, we obtain
\[
\begin{split}
   f(r_0) &\le \theta f(r_1) + \frac{A}{(r_1 - r_0)^\alpha} + \Lambda= \theta f(r_1) + \frac{A}{(1-\tau)^\alpha(b - r)^\alpha} + \Lambda\\
   &\le \theta \brac{\theta f(r_2) + \frac{A}{(1-\tau)^\alpha(b - r)^\alpha} + \Lambda} + \frac{A}{(1-\tau)^\alpha(b - r)^\alpha} + \Lambda\le\ldots\\
   &\le \theta^\ell f(r_\ell) + \frac{A}{(1-\tau)^\alpha(b-r)^\alpha}\sum_{i=0}^{\ell-1}\theta^i \tau^{-i\alpha} + \Lambda\sum_{i=0}^{\ell-1} \theta^i.
   \end{split}
 \]
Now we choose $\tau$ in such a way that $\tau^{-\alpha}\theta < 1$ and let $\ell\rightarrow \infty$ in the above inequality.
\end{proof}

\begin{proposition}[Caccioppoli inequality]\label{pr:cac}
Let $\n$ be such that $\pi_1(\n)$ is finite and let $\pi\colon \widetilde \n \to \n$ be its universal covering. Then there is a constant $C(n,\cN) > 0$ such that the following holds:
\hfill
\begin{enumerate}
 \item \label{it:Caccioppoliinterior} Let $u\in W^{1,2}(\Omega,\n)$ be a minimizing harmonic map and let $u=\pi\circ \tilde u$, $\tilde u \in W^{1,2}(\Omega, \widetilde \n)$, $B_{2r}(y)\subset\subset\Omega$, then 
\begin{equation*}\label{eq:comparisoncaccioppoli}
\int_{B_{r}(y)}|\nabla u|^2 \dx \leq C r^{-2} \int_{B_{2r}(y)} |\tilde u-(\tilde u)_{B_{2r}(y)}|^2 \dx,
\end{equation*} 
where $(\tilde u)_{B_{2r}(y)}$ denotes the mean value of $\tilde u$ on $\B_{2r}(y)$.

\hfill
\item \label{it:Caccioppoliboundary}Assume $\Omega \subset \R^n$ is a bounded domain with $C^1$-boundary. Let $u\in W^{1,2}(\Omega,\n)$ be a minimizing harmonic map, and let $\vp\in W^{\frac{1}{2},2}(\partial \Omega,\n)$ be the trace of $u$ on $\partial \Omega$. Then for all $r>0$ and $y_0\in\partial \Omega$ we have
 \begin{equation*}\label{eq:comparisoncaccioppoliboundary}
  \int_{B_r(y_0)\cap \Omega}|\nabla u|^2 \dx \le Cr^{-2} \int_{B_{2r}(y_0)\cap \Omega}|\tilde u(x) - \tilde \vp^{ext}(x)|^2\dx + C\int_{B_{2r}(y_0)\cap \Omega}|\nabla \tilde \vp^{ext} (x)|^2 \dx,
 \end{equation*}
where $u=\pi\circ \tilde u$, $\tilde u\in W^{1,2}(\Omega,\widetilde \n)$, $\vp = \pi\circ \tilde \vp \in W^{\frac12,2}(\partial \Omega, \widetilde{\mathcal N})$ and $\tilde \vp^{ext} \in W^{1,2}(B_{2r}(y_0)\cap \Omega)$ is any map with $\tilde \varphi^{ext} = \tilde \vp = \tilde u$ on $B_{2r}(y_0)\cap \partial \Omega$.

\end{enumerate}
\end{proposition}

\begin{proof}
The proofs of these two statements are similar; we will treat the boundary case \eqref{it:Caccioppoliboundary} here.

We use \Cref{th:BethuelChiron} and lift $u = \pi \circ \tilde u$, where $\tilde u \in W^{1,2}(\Omega,\widetilde \n)$ and $\widetilde \n$ is simply connected and compact (by the assumption that $\pi_1(\n)$ is finite). Since $\pi$ is a local isometry we have $\varphi = \pi \circ \tilde \varphi$ on $\partial \Omega$ and we take any extension $\tilde \varphi^{ext}\in W^{1,2}(B_{2r}(y_0)\cap \Omega)$ with $\tilde \varphi^{ext} = \tilde \varphi$ on $ B_{2r}(y_0)\cap \partial \Omega$. 

Fix $r \le \rho < R \le 2r$ and let 
\[
\tilde v(x) = \eta(x) \tilde \vp^{ext}(x) + (1-\eta(x)) \tilde u(x),
\]
where $\eta\in C_c^\infty(B_{R}(y_0),[0,1])$ is a cutoff function such that $\eta \equiv 1$ in $B_\rho(y_0)$, $\eta\equiv 0$ outside $B_{R}(y_0)$ and $|\nabla \eta|\le \frac{C}{R-\rho}$. It follows that $\tilde v \in W^{1,2}(\B_{2r}(y_0)\cap \Omega)$ coincides with $\tilde u$ on $\partial (B_{2r}(y_0)\cap \Omega)$. Therefore, by applying \Cref{th:extensionthm} we obtain a~map $\tilde w \in W^{1,2}(\B_{2r}(y_0)\cap \Omega, \widetilde \n)$ with $\tilde w = \tilde v$ on $\partial (B_{2r}(y_0)\cap \Omega)$ with the estimate
\begin{equation*}
  \int_{B_{2r}(y_0)\cap \Omega} |\nabla \tilde w|^2 \dx \le C \int_{B_{2r}(y_0)\cap \Omega} |\nabla \tilde v|^2 \dx.
\end{equation*}
We define now $w = \pi \circ \widetilde w \in W^{1,2}(B_{2r}(y_0)\cap \Omega,\n)$ and as before observe that $|\nabla w| = |\nabla \tilde w|$ a.e. and that $w\big\rvert_{\partial (B_{2r}(y_0)\cap \Omega)} = u \big\rvert_{\partial(B_{2r}(y_0)\cap \Omega)}$ thus, $w$ is a good comparison map and from the minimality of $u$ we get
\begin{equation}\label{eq:extineq}
 \int_{B_{R}(y_0)\cap \Omega} |\nabla u|^2 \le \int_{B_{R}(y_0)\cap \Omega} |\nabla w|^2 = \int_{B_{R}(y_0)\cap \Omega} |\nabla \tilde w|^2 \le C \int_{B_{R}(y_0)\cap \Omega} |\nabla \tilde v|^2 \dx.
\end{equation}
We compute
\[
 \nabla \tilde v(x) = (1-\eta(x)) \nabla \tilde u(x) - \nabla \eta(x) (\tilde u(x) - \tilde \vp^{ext}(x)) + \eta(x) \nabla \tilde \vp^{ext}(x),
\]
thus, from \eqref{eq:extineq} and $|\nabla u| = |\nabla \tilde u|$ a.e., we have
\[
\begin{split}
\int_{B_\rho(y_0)\cap \Omega } |\nabla u|^2 \dx &\le C\int_{(B_R(y_0)\setminus B_\rho(y_0))\cap \Omega}|\nabla \tilde u|^2 \dx +\frac{C}{(R-\rho)^2} \int_{B_R(y_0)\cap \Omega}|\tilde u(x) - \tilde \vp^{ext}(x)|^2\dx\\
&\quad + C\int_{B_R(y_0)\cap \Omega} |\nabla \tilde \vp^{ext}(x)|^2 \dx\\
&= C\int_{(B_R(y_0)\setminus B_\rho(y_0))\cap \Omega}|\nabla u|^2 \dx +\frac{C}{(R-\rho)^2} \int_{B_R(y_0)\cap \Omega}|\tilde u(x) - \tilde \vp^{ext}(x)|^2\dx\\
&\quad + C\int_{B_R(y_0)\cap \Omega} |\nabla \tilde \vp^{ext}(x)|^2 \dx.
\end{split}
\]
By a hole-filling argument, there is a $0 < \theta < 1$ such that 
\[
\begin{split}
 \int_{B_\rho(y_0)\cap \Omega} |\nabla u|^2 \dx &\le \theta\int_{B_R(y_0)\cap \Omega}|\nabla u|^2 \dx +\frac{C}{(R-\rho)^2} \int_{B_R(y_0)\cap \Omega}|\tilde u(x) - \tilde \vp^{ext}(x)|^2\dx\\
& \quad+ C\int_{B_R(y_0)\cap \Omega} |\nabla \tilde \vp^{ext}(x)|^2 \dx
\end{split}
 \]
for all $r\le \rho<R\le 1$. 

Thus, by Lemma \ref{la:iteration} we obtain
\[
 \int_{B_\rho(y_0)\cap \Omega} |\nabla u|^2 \dx \le \frac{C}{(R-\rho)^2} \int_{B_R(y_0)\cap \Omega}|\tilde u(x) - \tilde\varphi^{ext}(x)|^2\dx + C\int_{B_R(y_0)\cap \Omega} |\nabla \tilde \varphi^{ext}(x)|^2 \dx.
\]
We conclude our claim by taking $\rho=r$ and $R=2r$.
\end{proof}

As consequences of Poincar\'{e} inequality, Sobolev embedding, and Gehring Lemma we readily obtain the following, see also \cite[Theorem 4.1]{HardtKinderlehrerLinstable}.

\begin{corollary}[Higher integrability]\label{co:higherintint}
Let $\Omega\subset\R^n$ be a bounded domain with a $C^1$ boundary, $\pi_1(\n)$ be finite, and let $u\in W^{1,2}(\Omega,\n)$ be a minimizing harmonic map. There exist constants $q > 2$ and $C>0$ such that: 
\hfill

\begin{enumerate}
 \item \label{it:higherintegrabilityinterior} If $B_{2r}(y)\subset\subset\Omega$, then
\[
 \brac{r^{q-n}\int_{B_{r}(y)}|\nabla u|^q \dx}^{1/q} \le C \brac{r^{2-n}\int_{B_{2r}(y)} |\nabla u|^2 \dx}^{1/2}
\]
and the constants $q,\ C$ do not depend on $\Omega;$
%
\item \label{it:higherintegrabilityboundary}If $u\big\rvert_{\partial \Omega} = \vp \in W^{s,p}(\partial \Omega, \n)$, $s>\frac12$, and $p>1$, then for all $r>0$, $y_0\in\partial\Omega$ we have
\[
\begin{split}
 \brac{r^{q-n} \int_{B_{r}(y_0)\cap \Omega}|\nabla u|^q \dx}^{1/q} 
 &\le  C\brac{r^{2-n} \int_{B_{2r}(y_0)\cap \Omega} |\nabla u|^2 \dx}^{1/2}\\
 &\quad + Cr^{(sp-(n-1))\frac{1}{2sp}} [\vp]_{W^{s,p}(B_{2r}(y_0)\cap \partial \Omega)}^{\frac{1}{2s}}.
\end{split}
 \]
\end{enumerate}
\end{corollary}

\begin{proof}
Applying \Cref{th:BethuelChiron} we have $u = \pi \circ \tilde u $, where $\tilde u \in W^{1,2}(\Omega, \widetilde \n)$ and $\vp = \pi\circ \tilde \vp$ with $\tilde \vp \in W^{s,p}(\partial \Omega, \widetilde \n)$. 

\underline{We begin with the proof of \eqref{it:higherintegrabilityinterior}}. We have by Poincar{\'e}-Sobolev inequality
\begin{equation}\label{eq:reversehoeldersobolev}
 \begin{split}
 \int_{B_{2r}(y)} |\tilde u - (\tilde u)_{B_{2r}(y)}|^2 \dx 
 &\le C r^2 \brac{\int_{B_{2r}(y)}|\nabla \tilde u |^{\frac{2n}{n+2}} \dx}^{\frac{n+2}{n}}\\
 &= C r^2 \brac{\int_{B_{2r}(y)}|\nabla u |^{\frac{2n}{n+2}} \dx}^{\frac{n+2}{n}},
\end{split}
 \end{equation}
where we used that $|\nabla u| = |\nabla \tilde u|$ a.e.. 

Combining this with Proposition \ref{pr:cac} \eqref{eq:comparisoncaccioppoli}, we obtain
\[
 \int_{B_r(y)} |\nabla u|^2 \dx \le \brac{\int_{B_{2r}(y)}|\nabla u |^{\frac{2n}{n+2}} \dx}^{\frac{n+2}{n}}.
\]
Thus applying Gehring lemma \cite[p.122]{Giaquinta} (see also \cite{Iwaniec-Gehring}) we conclude.

\underline{In order to prove part \eqref{it:higherintegrabilityboundary}} we estimate 
\[
\int_{B_{2r}(y_0)\cap \Omega} |\tilde u (x) - \tilde \vp^h (x)|^2 \dx \aleq r^2 \brac{\int_{B_{2r}(y_0)\cap \Omega} |\nabla u|^\frac{2n}{n+2} \dx}^{\frac{n+2}{n}}
\]
in the same way as \eqref{eq:reversehoeldersobolev}. 

Applying \Cref{pr:cac} \eqref{eq:comparisoncaccioppoliboundary} with $\tilde \vp^{ext} = \tilde \vp^h$, where $\tilde \vp^h\colon B_{2r}(y_0)\cap \Omega \to \R^d$ is the harmonic extension of $\tilde \vp\colon B_{2r}(y_0)\cap \partial \Omega \to \widetilde \n$ 
\[
\int_{B_r(y_0)\cap \Omega} |\nabla u|^2 \dx \aleq \brac{\int_{B_{2r}(y_0)\cap \Omega} |\nabla u|^\frac{2n}{n+2} \dx}^{\frac{n+2}{n}} +  \int_{B_{2r}(y_0)\cap \Omega} |\nabla \tilde \vp^h|^2 \dx.
\]
Thus we may apply Gehring lemma and obtain existence of a number $q>2$ such that
\[
\brac{r^{q-n}\int_{B_r(y_0)\cap \Omega} |\nabla u|^q \dx}^\frac{1}{q} \aleq \brac{r^{2-n}\int_{B_{2r}(y_0)\cap \Omega} |\nabla u|^2 \dx}^{\frac{1}{2}} +  \brac{r^{2-n}\int_{B_{2r}(y_0)\cap \Omega} |\nabla \tilde \vp^h|^2 \dx}^{\frac{1}{2}}.
\]
Additionally, we have by the trace inequality, the Gagliardo--Nirenberg inequality for any $s > \frac12$, $p>1$, and by the compactness of $\widetilde{ \mathcal N}$
\[
\begin{split}
\brac{r^{2-n}\int_{B_{2r}(y_0)\cap \Omega} |\nabla \tilde \vp^h|^2 \dx }^\frac12 
&\aleq r^{\frac{2-n}{2}}[\tilde \vp]_{W^{\frac12,2}(B_{2r}(y_0)\cap \partial \Omega)}\\
& \aleq r^{(sp-(n-1))\frac{1}{2sp}}\|\tilde \vp\|^{1-\frac{1}{2s}}_{L^{\infty}(B_{2r}(y_0)\cap \partial \Omega)}[\tilde \vp]_{W^{s,p}(B_{2r}(y_0)\cap \partial \Omega)}^{\frac{1}{2s}}\\
&\aleq r^{(sp-(n-1))\frac{1}{2sp}} [\vp]_{W^{s,p}(B_{2r}(y_0)\cap \partial \Omega)}^{\frac{1}{2s}},
\end{split}
\]
where in the last inequality we used the Lipschtiz continuity of the universal cover $\pi\colon \widetilde{\mathcal N}\to \n$. This finishes the proof.
\end{proof}

\section{Strong convergence for minimizers and its consequences}
\label{sec:strong-convergence}

Historically, compactness of minimizers has been a huge challenge. Partial results in this direction were obtained by Schoen--Uhlenbeck \cite[Lemma 4.3]{SU1} (tangent maps), then by Hardt--Lin \cite[Theorem 6.4]{HL1987} (target manifolds with $\pi_1(\n)=0$), and finally the general case was solved with the help of the celebrated Luckhaus Lemma \cite[Lemma 1]{L88}. In our special case when $\pi_1(\n)$ is finite we may lift our initial map, and since the universal cover $\widetilde \n$ is compact we can apply the extension property (Theorem~\ref{th:extensionthm}), similarly as in \cite{HardtKinderlehrerLin1986}. This simplifies the situation. Here we present a proof inspired by \cite[Theorem 6.4]{HL1987}.

This is a counterpart of \cite[Theorem 1.2]{AlmgrenLieb1988}.
\begin{theorem}[Strong convergence of minimizers]\label{th:bsc}
Let $\pi_1(\n)$ be finite. Then the following assertions hold:

\begin{enumerate}
\item \label{it:strongconvergenceinterior} Let $\Omega\subset\R^n$ be a bounded domain with a $C^1$-boundary and $u_i \in W^{1,2}(\Omega,\n)$ be a sequence of minimizing harmonic maps. Then, up to taking a subsequence $i \to \infty$, we find $u \in W^{1,2}(\Omega, \n)$ which is a minimizer in any subdomain $\Omega' \Subset \Omega$ and $u_i\rightarrow u$ strongly in $W^{1,2}_{loc}(\Omega,\n)$. 

\hfill
\item \label{it:stongconvergenceboundary} Let $u_i \in W^{1,2}(B^+_1(0),\n)$ be a sequence of minimizing harmonic maps. Set $\vp_i \coloneqq u_i$ on $T_1$ and assume additionally that 
\[
 \sup_{i \in \N} [\varphi_i ]_{W^{s,p}(T_1)}  < \infty
\]
for some $s> \frac12$, $p>1$, and $sp>1$.

Then, up to taking a subsequence $i \to \infty$, we find $u \colon \B^+_1(0) \to \n$ such that $u \in W^{1,2}(\B_r^+(0),\n)$ for any $r \in (0,1)$ and $u_i\rightarrow u$ strongly in $W^{1,2}(B_{r}^+(0),\n)$. Moreover, for every $r \in (0,1)$, the map $u$ is a minimizing harmonic map in $B_r^+(0)$.

\hfill
\item \label{it:strongconvergencegeneral} Let the domain $\Omega$ and the maps $u_i$ be as in \eqref{it:strongconvergenceinterior}. Assume additionally that their traces $\vp_i$ are uniformly bounded in
$W^{s,p}(\partial \Omega, \cN)$ for some $s>\frac12$, $p>1$, $sp >1$. Then up to taking a subsequence, $u_i \to u$ strongly in $W^{1,2}(\Omega, \cN)$ and $u$ is minimizing in $\Omega$.

\end{enumerate}
\end{theorem} 
We will need the following lemma. 
\begin{lemma}[Poincar\'e-type Lemma]\label{la:1dpoinc}
Let $f \in W^{1,2}(B^+_1(0))$ be such that $f = 0$ on $T_{3/4}$ in the sense of trace. Then, for any $\delta \in (0,\frac{1}{2})$,
\[
 \int_{T_{3/4} \times (0,\delta)} |f|^2 \dx \le \delta^2 \int_{T_{3/4} \times (0,\delta)} |\nabla f|^2 \dx.
\]
\end{lemma}
\begin{proof}
If a function $\varphi \colon [0,\delta] \to \R$ is absolutely continuous, the fundamental theorem of calculus implies 
\[
 \int_0^\delta |\varphi(t) - \varphi(0)|^2 \dd t \leq \delta^2 \int_0^\delta |\varphi'(t)|^2 \dd t. 
\]
Since $f$ is absolutely continuous on almost all lines, for almost all $x' \in T_{3/4}$ we have
\[
 \int_0^\delta |f(x',t)|^2 \dd t 
 = \int_0^\delta |f(x',t) - f(x',0)|^2 \dd t 
 \leq \delta^2 \int_0^\delta |\nabla f(x',t)|^2 \dd t.
\]
Integrating this over $T_{3/4}$, we obtain our claim.
\end{proof}

\begin{proof}[Proof of Theorem~\ref{th:bsc}]
The proof of \eqref{it:strongconvergenceinterior} follows in a manner similar to \eqref{it:stongconvergenceboundary}, see for example \cite[Theorem 6.4]{HL1987}. For this reason, we skip the details and concentrate on \eqref{it:stongconvergenceboundary}. 

\underline{\textsc{Proof of \eqref{it:stongconvergenceboundary}:}}
From Theorem~\ref{th:uniformboundedness}~\eqref{it:uniformboundfractionalboundry} we have for any $\theta\in(1,2)$ and any $r<1$
\begin{equation}\label{eq:uniformboundednessinstrongconvergence}
\int_{B_r^+(0)} |\nabla u_i|^2 \dx \aleq r^{n-2} + r^{(sp-(n-1))\frac{\theta}{sp} + n-2}[\vp_i]^{\frac{\theta}{s}}_{W^{s,p}(T_{2r})}. 
\end{equation}
If $sp>\theta$ then $(sp-(n-1))\frac{\theta}{sp} + n-2>1$, thus since 
\eqref{eq:uniformboundednessinstrongconvergence} holds for any $1<\theta<2$ and since $sp>1$ we can take $\theta$ small enough so that $sp>\theta$. Thus, 
\begin{equation*}
\int_{B_r^+(0)} |\nabla u_i|^2 \dx \aleq 1 + [\vp_i]^{\frac{\theta}{s}}_{W^{s,p}(T_{2r})}.
\end{equation*}
By assumption the boundary maps $\vp_i$ are uniformly bounded in $W^{s,p}(T_1)$, consequently
\[
 \sup_{i \in \N} [u_i]_{W^{1,2}(B^+_r(0))} < \infty \quad \mbox{for any $r \in (0,1)$}.
\]
In particular, up to taking a subsequence and diagonalizing we find $u \colon \B^+_1(0) \to \n$ which is a weak $W^{1,2}$-limit, and strong $L^2$-limit of $u_i$ in each ball $B^+_r(0)$, and $\varphi$ is the weak $W^{s,p}$-limit of $\varphi_i$ on each $T_r$, such that $\varphi$ is the trace of $u$.

We need to show that $u$ is a minimizer in $\B^+_r(0)$ and that $u_i \to u$ strongly with respect to the $W^{1,2}$-norm in $B^+_r(0)$ for every $r \in (0,1)$. For simplicity of the notation we shall assume that $r = \frac{1}{2}$.

Since the boundary maps $\vp_i$ are uniformly bounded in $W^{s,p}$, \Cref{co:higherintint}~\eqref{it:higherintegrabilityboundary} implies uniform higher integrability of $u_i$. Namely, for some fixed $q > 2$ we have
\begin{equation}\label{eq:bsc:higherintegr}
\sup_{i} \int_{B^+_{3/4}(0)} |\nabla u_i|^q \dx < \infty.
\end{equation}

Fix a competitor $v \in W^{1,2}(\B_{3/4}^+(0),\cN)$, i.e., a map coinciding with $u$ outside $\B_{1/2}^+(0)$, in particular $u = v$ on $T_{3/4}(0)$. 

Using $v$, we construct a competitor for $u_i$. We do so by an interpolation on a set $I_\delta$ which separates $\Omega_\delta$ and $\mho_\delta$, which are defined as follows:
\begin{equation}\label{eq:Omegadeltadefinition}
\begin{split}
 \Omega_\delta &\coloneqq B_{1/2}^+(0) \setminus \brac{\R^{n-1} \times (0,2\delta)};\\
 I_\delta &\coloneqq B^+_{1/2+\delta}(0) \setminus \brac{\Omega_\delta \cup \brac{\R^{n-1} \times (0,\delta)} };\\
 \mho_\delta &\coloneqq B_{3/4}^+(0) \setminus \brac{I_\delta \cup \Omega_\delta}.
\end{split}
 \end{equation}
\begin{center}
\begin{tikzpicture}[line cap=round,line join=round,>=triangle 45,x=1cm,y=1cm]
 \fill[black!40, opacity=0.2] (-5,0) -- (5,0) arc(0:180:5) --cycle;
 \fill[blue!40, opacity=0.4] ([shift=(5:2.75cm)]0,0) arc(5:175:2.75cm)-- cycle;
 \fill[red!40, opacity=0.4] ([shift=(10:2.5cm)]0,0) arc(10:170:2.5cm)-- cycle;
\draw[black] (-5,0) -- (5,0) arc(0:180:5) --cycle;
\draw[blue] ([shift=(5:2.75cm)]0,0) arc(5:175:2.75cm)-- cycle;
 \draw[red] ([shift=(10:2.5cm)]0,0) arc(10:170:2.5cm)-- cycle;
 \node at (3.5,2) {$\mho_\delta$};
 \node at (1,0.7) {$\Omega_\delta$};
 \node at (2,1.7) {\small $I_\delta$};
\end{tikzpicture}

\end{center}

Observe that $\partial I_\delta$ is separated into two parts, the inner part being $\partial \Omega_\delta$ and the outer being $\partial (I_\delta \cup \Omega_\delta)$. Choose a cut-off function $\eta \in C^\infty$ such that $\eta \in [0,1]$, $\eta_\delta \equiv 1$ in $\Omega_\delta$, $\eta_\delta \equiv 0$ in $\mho_\delta$, and $|\nabla \eta_\delta| \aleq \frac{1}{\delta}$.

We glue $v$ with $u_i$ by defining the map  
\begin{equation*}
 v_{\delta,i} \coloneqq \eta_\delta v + (1-\eta_\delta) u_i = u_i + \eta_\delta (v-u_i), 
\end{equation*}
However, $w_{\delta,i}$ does not map into $\cN$ in the intermediate region $I_\delta$. This can be fixed by using the extension theorem (Theorem~\ref{th:extensionthm}). First we need to lift the maps $u = \pi \circ \tilde u$, $v=\pi\circ \tilde v$, $u_i = \pi \circ \tilde u_i$ where $\tilde u,\, \tilde v, \tilde u_i \in W^{1,2}(B_{3/4}^+(0),\widetilde \n)$ and define
\begin{equation}\label{eq:vdeltait}
 \tilde v_{\delta,i} = \eta_\delta \tilde v + (1-\eta_\delta)\tilde u_i.
\end{equation}
Now, recalling that $\widetilde \n$ is a simply connected compact manifold we may apply Theorem~\ref{th:extensionthm} on the region $I_\delta$ and obtain existence of a manifold valued map $\tilde w_{\delta,i}\in W^{1,2}(I_\delta,\widetilde \n)$ which agrees with $\tilde v_{\delta,i}$ on the boundary $\tilde w_{\delta,i} \big\rvert_{\partial I_\delta}= \tilde v_{\delta,i}\big\rvert_{\partial I_\delta}$ and 
\begin{equation}\label{eq:nv1}
 \int_{I_\delta} |\nabla \tilde w_{\delta,i}|^2 \dx \le C\int_{I_\delta} |\nabla \tilde v_{\delta,i}|^2 \dx
\end{equation}
with a constant independent of $i$ and $\delta$. We extend $\tilde w_{\delta,i}$ into $B_{3/4}^+(0)$ by setting $\tilde w_{\delta,i} \equiv \tilde v$ in $\Omega_\delta$, $\tilde w_{\delta,i} \equiv \tilde u_i$ in $\mho_\delta$. 
Now we define $w_{\delta,i} = \pi \circ \tilde w_{\delta,i} \in W^{1,2}(B_{3/4}^+(0),\n)$ and note that $w_{\delta,i}\big\rvert_{\partial B_{3/4}^+(0)} = u_i \big\rvert_{\partial B_{3/4}^+(0)}$. In particular, $w_{\delta,i}$ is a competitor for $u_i$ in $B^+_{3/4}(0)$, and the minimizing property of $u_i$ implies 
\begin{equation}\label{eq:sc:1}
\begin{split}
 \int_{B_{3/4}^+(0)} |\nabla u_i|^2 \dx &\leq \int_{B_{3/4}^+(0)} |\nabla w_{\delta,i}|^2 \dx\\
 & = \int_{\Omega_\delta} |\nabla v|^2 \dx 
 + \int_{\mho_\delta} |\nabla u_i|^2 \dx 
 + \int_{I_\delta} |\nabla w_{\delta,i}|^2 \dx\\
 & = \int_{\Omega_\delta} |\nabla v|^2 \dx 
 + \int_{\mho_\delta} |\nabla u_i|^2 \dx 
 + \int_{I_\delta} |\nabla \tilde w_{\delta,i}|^2 \dx.
\end{split}
 \end{equation}
We observe that
\[
 \int_{\mho_\delta} |\nabla u_i|^2 \dx \leq \int_{B_{3/4}^+(0) \setminus B_{1/2}^+(0)} |\nabla u_i|^2 \dx + \int_{T_{1/2} \times (0,\delta)} |\nabla u_i|^2 \dx,
\]
thus, after enlarging $\Omega_\delta$ to $\B_{1/2}^+$, \eqref{eq:sc:1} becomes
\begin{equation}\label{eq:sc:2}
\begin{split}
 \int_{B_{1/2}^+(0)} |\nabla u_i|^2 \dx 
 \leq \int_{B_{1/2}^+(0)} |\nabla v|^2 \dx 
 + \int_{T_{1/2}\times(0,\delta)} |\nabla u_i|^2 \dx 
 + \int_{I_\delta} |\nabla \tilde w_{\delta,i}|^2 \dx.
\end{split}
\end{equation} 
Moreover, by \eqref{eq:nv1} and \eqref{eq:vdeltait}
\[
\begin{split}
 \int_{I_\delta} |\nabla \tilde w_{\delta,i}|^2 \dx  
 &\aleq \int_{I_\delta} |\nabla \tilde{v}_{\delta,i}|^2 \dx\\
 &\aleq  \int_{I_\delta} |\nabla \tilde u_i|^2 \dx + \int_{I_\delta} |\nabla \tilde v|^2 \dx+ \frac{1}{\delta^2} \int_{I_\delta} |\tilde u_i-\tilde v|^2 \dx\\
 &\aleq  \int_{I_\delta} |\nabla u_i|^2 \dx + \int_{I_\delta} |\nabla v|^2 \dx + \frac{1}{\delta^2} \int_{I_\delta} |\tilde u-\tilde v|^2 \dx + \frac{1}{\delta^2} \int_{I_\delta} |\tilde u_i- \tilde u|^2 \dx,
\end{split}
 \]
 where $\tilde u = \pi \circ \tilde u$ and $\tilde u \in W^{1,2}(B_{3/4}^+(0),\widetilde \n)$.
 
Observe that $\tilde u \equiv \tilde v$ outside $\B_{1/2}^+$ and on $T_{3/4}$, in consequence 
\[
\begin{split}
 \frac{1}{\delta^2} \int_{I_\delta} |\tilde u- \tilde v|^2 \dx 
 &\leq \frac{1}{\delta^2} \int_{T_{3/4} \times (0,2\delta)} |\tilde u- \tilde v|^2 \dx
 \aleq \int_{T_{3/4} \times (0,2\delta)} |\nabla (\tilde u-\tilde v)|^2 \dx\\
 &\aleq \int_{T_{3/4} \times (0,2\delta)} |\nabla (u-v)|^2 \dx
\end{split}
 \]
by an application of Lemma~\ref{la:1dpoinc}. Moreover, higher integrability of $u_i$ \eqref{eq:bsc:higherintegr} allows us to use H\"older's inequality and estimate 
\[
 \int_{T_{1/2}\times(0,\delta)} |\nabla u_i|^2 \dx + \int_{I_\delta} |\nabla u_i|^2 \dx
 \aleq \left |T_{1/2} \times (0,\delta) \right |^{1-\frac{2}{q}} + |I_\delta|^{1-\frac{2}{q}} 
 \aleq \delta^{1 - \frac{2}{q}}
\]
with some positive exponent $1 - \frac{2}{q} > 0$. 

Of course, we also have
\[
 \frac{1}{\delta^2} \int_{I_\delta} |\tilde u_i- \tilde u|^2 \dx \aleq \frac{1}{\delta^2} \int_{I_\delta} |u_i- u|^2 \dx. 
\]
Thus, we arrive at
\begin{equation}\label{eq:sc:3}
\begin{split}
 \int_{B_{1/2}^+(0)} |\nabla u_i|^2 \dx \leq 
 & \int_{B_{1/2}^+(0)} |\nabla v|^2 \dx + C \delta^{1 - \frac{2}{q}} + C\int_{I_\delta} |\nabla v|^2 \dx\\
 &+ C \int_{T_{3/4} \times (0,2\delta)} |\nabla (u-v)|^2 \dx +\frac{C}{\delta^2} \int_{I_\delta} |u_i-u|^2 \dx .
\end{split}
\end{equation}
This estimate holds for all $i$ and $\delta$, so may take the limit superior $i \to \infty$. Recalling strong convergence $u_i \to u$ in $L^2(\B_{3/4}^+(0))$, we obtain 
\begin{equation}\label{eq:sc:4}
\begin{split}
 \limsup_{i \to \infty} \int_{B_{1/2}^+(0)} |\nabla u_i|^2 \dx \leq
 & \int_{B_{1/2}^+(0)} |\nabla v|^2 \dx + C \delta^{1-\frac{2}{q}} \\
 & + C \int_{I_\delta} |\nabla v|^2 \dx + C \int_{T_{3/4} \times (0,2\delta)} |\nabla (u-v)|^2 \dx. 
\end{split}
\end{equation}
In the limit $\delta \to 0$, the last two integrals vanish by absolute continuity of the integral, and hence we arrive at the estimate 
\begin{equation}\label{eq:sc:HUNGRY}
 \limsup_{i \to \infty} \int_{B_{1/2}^+(0)} |\nabla u_i|^2 
 \dx \leq \int_{B_{1/2}^+(0)} |\nabla v|^2 \dx.
\end{equation}

Employing weak lower semicontinuity of the Dirichlet energy, we can now easily conclude both strong convergence of $u_i$ and minimality of $u$. Indeed, after integrating the identity 
\[
|\nabla u_i - \nabla u|^2 = |\nabla u_i|^2 + |\nabla u|^2 - 2 \nabla u_i \cdot \nabla u
\]
over $\B_{1/2}^+(0)$ and taking the limit superior on both sides, we obtain 
\[
\limsup_{i \to \infty} \int_{\B_{1/2}^+(0)} |\nabla u_i - \nabla u|^2 \dd x 
= \limsup_{i \to \infty} \int_{\B_{1/2}^+(0)} |\nabla u_i|^2 - |\nabla u|^2 \dd x 
\le 0
\]
due to the weak convergence $\nabla u_i \rightharpoonup \nabla u$ in $L^2(\B_{1/2}^+(0))$ and the estimate \eqref{eq:sc:HUNGRY} applied with $v \equiv u$. This shows that $\nabla u_i \to \nabla u$ strongly in $L^2(\B_{1/2}^+(0))$. But now the left-hand side of \eqref{eq:sc:HUNGRY} is just the energy of $u$, and the minimality of $u$ follows. 

\underline{\textsc{Proof of \eqref{it:strongconvergencegeneral}:}}
Let $a\in \partial \Omega$ be any point on the boundary. 

Exactly as in proof of \eqref{it:stongconvergenceboundary}, we find that for an $R>0$ (up to a subsequence) $u_i$ converges weakly to $u$ in $W^{1,2}(B_{R}(a)\cap \Omega)$, strongly in $L^2(B_{R}(a)\cap \Omega)$, and $a.e.$. We will show that the convergence is in fact strong and that $u\in W^{1,2}(B_{R/2}(a),\n)$ is a minimizing harmonic map. To do so we proceed exactly as in proof of \eqref{it:stongconvergenceboundary}, the only difference is that we have to redefine the sets in \eqref{eq:Omegadeltadefinition} in terms of the distance to the boundary:
\begin{equation}\label{eq:Omegadeltanonflat}
\begin{split}
 \Omega_\delta &\coloneqq (B_{R/2}^+(a)\cap\Omega) \setminus \{x\in\Omega \colon \dist(x,\partial \Omega)<2\delta\};\\
 I_\delta &\coloneqq B^+_{R/2+\delta}(a) \setminus \brac{\Omega_\delta \cup \{x\in\Omega \colon \dist(x,\partial \Omega)<\delta\} };\\
 \mho_\delta &\coloneqq B_{3R/4}^+(a) \setminus \brac{I_\delta \cup \Omega_\delta}.
\end{split} 
\end{equation}

Since $\partial \Omega$ is compact, a covering argument leads to the conclusion. 
\end{proof}

\subsection{Strong convergence with variable boundary}\label{ss:convergencewithvariableboundry}

A technical modification of this reasoning allows us to consider in Theorem~\ref{th:bsc} a sequence of maps $u_i$ defined on converging Lipschitz domains with non-flat boundaries. This will be used in \Cref{th:r4s} and \Cref{th:hot-spots}. 

\begin{proposition}\label{pr:strongconvergenceconvergingboundary}
Let $\n$ be a manifold with finite fundamental group and 
\[
 \Omega_i \coloneqq \{x\in B_R\colon x_n>\alpha_i(x')\},
\]
where the sequence of functions $\alpha_i \in C^1(\R^{n-1},\R)$ converges to zero in $C^1$. Assume that $u_i\in W^{1,2}(\Omega_i, \n)$ is a sequence of minimizing harmonic maps with boundary maps $\vp_i \coloneqq u_i \big|_{\partial \Omega_i \cap B_R}$ satisfying the uniform bound
\[
 \sup_{i \in \N} [\varphi_i ]_{W^{s,p}(\partial \Omega_i \cap B_R)}  < \infty 
 \quad \text{for some } s < \frac12, \ p>1,\ sp>1.
\]
Let us choose $C^1$-diffeomorphisms $\varsigma_i \colon B_R \to B_R$ that map $\Omega_i$ into the half-ball $B_R^+$ and converge to identity in $C^1$.

Then, up to taking a subsequence $i \to \infty$, $u_i\circ \dif_i^{-1}$ converges strongly in $W^{1,2}(B_{r}^+,\n)$ (for each $r<R$) to a map $u \colon \B^+_R \to \n$ which is minimizing in each ball $B_r^+$. In consequence, the traces also converge: $\vp_i \circ \dif_i^{-1} \to u \big|_{T_r}$ in $W^{\frac12,2}(T_r,\cN)$. 
\end{proposition}

The proof will be based on the following lemma. 
\begin{lemma}
\label{lem:continuity-of-shifts}
Let $B \subset \R^n$ and $\dif_i \colon B \to B$ be a sequence of diffeomorphisms convergent to identity in $C^1$. If $u \in L^2(B)$, then $u \circ \dif_i \to u$ in $L^2(B)$. 
\end{lemma}

\begin{proof}
Let us fix $\eps > 0$. One can then choose $v \in C_c^\infty(B)$ such that
\begin{equation}\label{eq:easy0}
\| u-v \|_{L^2(B)} \le \eps. 
\end{equation}
Since $\dif_i$ is $C^1$-close to identity, we also have for large enough $i$ 
\begin{equation}\label{eq:easy1}
\| u \circ \dif_i - v \circ \dif_i \|_{L^2(B)} 
= \| (u - v) \circ \dif_i \|_{L^2(B)} 
\le 2 \eps.
\end{equation}
Moreover, since $v \in C^1$, the difference $v \circ \dif_i - v$ converges uniformly to $0$, in consequence for large enough $i$
\begin{equation}\label{eq:easy2}
\| v \circ \dif_i - v \|_{L^2(B)} \le \eps. 
\end{equation}
Thus, by triangle inequality and \eqref{eq:easy0}, \eqref{eq:easy1}, and \eqref{eq:easy2} we obtain 
\begin{equation}\label{eq:easy3}
\begin{split}
\| u - u \circ \dif_i \|_{L^2(B)} &\le \|u-v\|_{L^2(B)} + \| v- v\circ \varsigma_i\|_{L^2(B)}  + \|v\circ \sigma_i - u\circ \varsigma_i\|_{L^2(B)}\\
&\le 4\eps.
\end{split}
\end{equation}
\end{proof}

\begin{proof}[Proof of \Cref{pr:strongconvergenceconvergingboundary}]
\hfill

\textsc{Step 1.} Repeating the reasoning from the proof of \Cref{th:bsc}~\eqref{it:stongconvergenceboundary}, we are able to establish the following convergence for each $r < R$: 
\begin{equation}\label{eq:firstconvergence}
\begin{array}{rll}
u_i \chi_{\Omega_i \cap B_r} & \xrightarrow{i\to\infty} u \chi_{B_r^+} &\text{ strongly in } L^2(B_r); \\
\nabla u_i \chi_{\Omega_i \cap B_r} & \xrightarrow{i\to\infty} \nabla u \chi_{B_r^+} &\text{ strongly in } L^2(B_r).
\end{array}
\end{equation}
Unfortunately, this does not suffice to conclude convergence of boundary maps $\vp_i$. 

Using \Cref{lem:continuity-of-shifts} we will upgrade the convergence to $u_i\circ \dif_i^{-1} \to u$ strongly in $W^{1,2}(B_r^+)$.

\textsc{Step 2.} Strong convergence of $u_i\circ \dif_i^{-1}$ in $L^2(B_r^+)$.

Recall that $\varsigma_i, \varsigma_i^{-1} \to \id$ on $B_r$. Since $\dif_i^{-1}\colon B_r^+ \to \Omega_i\subset B_R^+$, we have by triangle inequality  
\begin{equation}\label{eq:decompositionuiciecsigmai}
\| u_i \circ \dif_i^{-1} - u \|_{L^2(B_r^+)} 
\le \| u \circ \dif_i^{-1} - u \|_{L^2(B_r^+)}
+ \| u_i \circ \dif_i^{-1} - u \circ \dif_i^{-1} \|_{L^2(B_r^+)}. 
\end{equation}
By \Cref{lem:continuity-of-shifts} (with $\Omega = B_R^+$ and $\Omega'=B_r^+$) we obtain  $\lim_{i\to \infty}\| u \circ \dif_i^{-1} - u \|_{L^2(B_r^+)}=0$. As for the second term of \eqref{eq:decompositionuiciecsigmai} we have from Step 1.
\[
\lim_{i\to\infty}\| u_i \circ \dif_i^{-1} - u \circ \dif_i^{-1} \|_{L^2(B_r^+)} \aleq  \lim_{i\to\infty} \| u_i \chi_{\Omega_i\cap B_r} - u \chi_{B_r^+}\|_{L^2(B_r^+)} = 0. 
\]

\textsc{Step 3.} Strong convergence of the gradients. 
We first note that 
\[
\nabla (u_i\circ \dif_i^{-1}) = \left( (\nabla u_i) \circ \dif_i^{-1} \right) \cdot \nabla (\dif_i^{-1}),
\]
here by $\cdot$ we mean the multiplication of the Jacobian matrices.
Since $\dif_i^{-1} \to \id$ in $C^1$, we have the uniform convergence $I - \nabla \dif_i^{-1} \to 0$, where $I\in M^{n\times n}$ is the identity matrix. This gives
\begin{equation}\label{eq:blablablablablablabla}
 \begin{split}
\| \nabla (u_i\circ \dif_i^{-1}) - (\nabla u_i) \circ \dif_i^{-1} \|_{L^2(B_r^+)} 
& \le \| \left( (\nabla u_i) \circ \dif_i^{-1} \right) \cdot \left( I - \nabla (\dif_i^{-1}) \right) \|_{L^2(B_r^+)} \\
& \lesssim \| \nabla u_i \|_{L^2(B_r^+)} \cdot \| I - \nabla (\dif_i^{-1}) \|_{L^\infty(B_r^+)}\xrightarrow{i\to\infty}0 .  
 \end{split}
\end{equation}
Thus,
\begin{equation}\label{eq:currentlylisteningtothat'swhatsup}
 \begin{split}
\|&\nabla (u_i\circ \dif_i^{-1}) - \nabla u\|_{L^2(B_r^+)} \\
&\le \| \nabla (u_i\circ \dif_i^{-1}) - (\nabla u_i) \circ \dif_i^{-1} \|_{L^2(B_r^+)} + \|(\nabla u_i) \circ \dif_i^{-1} -\nabla u\|_{L^2(B_r^+)}.   
 \end{split}
\end{equation}
The convergence of the first term of the left-hand side of \eqref{eq:currentlylisteningtothat'swhatsup} follows from \eqref{eq:blablablablablablabla}. In order to obtain the convergence of the second term, $\|(\nabla u_i) \circ \dif_i^{-1} -\nabla u\|_{L^2(B_r^+)}$ we proceed exactly as in Step 2. This finishes the proof.
\end{proof}

\section{Boundary regularity for smooth and singular boundary data in \texorpdfstring{$W^{1,n-1}$}{W(1,n-1)}}\label{sec:boundaryregularity}
It is a classical result by Schoen and Uhlenbeck that minimizing  harmonic maps with $C^{2,\alpha}$ boundary data are $C^{2,\alpha}$ in a neighborhood of the boundary \cite{SU2}.

One of the quite surprising results of Almgren and Lieb in \cite{AlmgrenLieb1988} is that even possibly singular boundary data (they consider $W^{1,2}(\partial \B^3_1(0),\S^2)$) prevents singularities from reaching the boundary. 

In this section we extend this result to larger trace space and general dimension.

First of all we notice the interior regularity, which is a corollary of the compactness result, Theorem \ref{th:bsc}~\eqref{it:strongconvergencegeneral}, and Theorem \ref{th:ALs2s}.
\begin{theorem}[interior regularity for almost constant boundary data]
\label{th:int-regularity-in-terms-of-bdry}
Assume $\pi_1(\n)$ is finite, $s>\frac12$, and $sp>1$. For each bounded domain with $C^1$-boundary $\Omega \subset \R^n$ and each $\sigma > 0$ there exist an $\eps= \eps(\Omega,\sigma) > 0$ such that the following holds: If $u \in W^{1,2}(\Omega,\n)$ is a minimizing harmonic map with trace $\varphi \coloneqq u \Big |_{\partial \Omega}$ and assume that for $s>\frac12$, $p>1$, and $sp> 1$ we have
\[
[\varphi]^p_{W^{s,p}(\partial \Omega)}\le \eps,
\]
then $u$ is smooth in the interior region $\{ x \in \Omega \colon \dist(x, \partial \Omega) > \sigma \}$. 
\end{theorem}

\begin{proof}
Assume on the contrary that there exists a $\sigma>0$ and sequence of minimizing maps $u_i\in W^{1,2}(\Omega,\n)$ with $u_i\big\rvert_{\partial\Omega} = \varphi_i$ and 
\[
 [\varphi_i]^p_{W^{s,p}(\partial \Omega)} \le \frac{1}{i}
\]
such that each $u_i$ has a singular point $y_i\in \{x\in\Omega \colon \dist(x,\partial\Omega)>\sigma \}$.

Then by the strong convergence of minimizers, Theorem \ref{th:bsc} we would obtain the existence of a minimizing harmonic map $u\in W^{1,2}(\Omega,\n)$ such that, up to a subsequence, $u_i\rightarrow u$ strongly in $W^{1,2}(\Omega,\n)$ with $u\big\rvert_{\partial\Omega} = const$. Thus $u$ itself would be a constant map and therefore have no singularities. On the other hand, from the sequence $y_i$ of singular points of $u_i$ we could choose a subsequence converging to a point $y\in \{x\in \Omega\colon \dist(x,\partial \Omega)\ge\sigma\}\subset\Omega$ and from Theorem \ref{th:ALs2s} (Singular points converge to singular points) we would know that $y$ must be a singular point of the limiting map $u$, which gives a contradiction. 
\end{proof}

\subsection{Uniform boundary regularity for constant boundary data}
The first step is uniform boundary regularity for constant boundary data, see \cite[Theorem 1.10]{AlmgrenLieb1988}. 
\begin{theorem}[Boundary regularity]\label{th:boundaryregularity-constant}
Assume $\pi_1(\n)$ is finite. There exists a uniform constant $\lambda > 0$ such that the following holds:
Let $u\in W^{1,2}(\B^+_1(0),\n)$ be a minimizer. Moreover, assume that $\varphi = u\Big |_{T_1}$ is constant.
Then $u$ is smooth in
\[
 (0,\lambda) \times T_{1/2}.
 \]
\end{theorem}

The main ingredient in Theorem~\ref{th:boundaryregularity-constant} is the following.
\begin{lemma}\label{la:uniformsmallnessboundaryconstant}
For any $\eps > 0$ there is a uniform constant $R_0(\eps) \in (0,\frac{1}{2})$ so that the following holds.
Let $u\in W^{1,2}(\B^+_1(0),\n)$ be a minimizer, where $\cN$ has finite fundamental group, and assume that $\varphi = u\Big |_{T_1}$ is constant.
Then for any $x_0 \in T_{1/2}$
\[
\sup_{r < R_0(\eps)} r^{2-n} \int_{B_r^+(x_0)} |\nabla u|^2 \dx < \eps.
 \]
\end{lemma}
\begin{proof}
Assume that the claim is false for some $\eps > 0$, then we find a sequence $R_i \to 0$, $x_i \in T_{\frac{1}{2}}$, and a sequence of minimizers $u_i$ with boundary data $u_i\big\rvert_{T_1} = \varphi_i \in \n$ constant on the flat part of the boundary, such that 
\begin{equation*}
 R_i^{2-n} \int_{B_{R_i}^+(x_i)} |\nabla u_i|^2 \dx \geq \eps.
\end{equation*}
By the boundary monotonicity formula, Theorem~\ref{th:bdmonotonicityformula}, 
\[
 \inf_{r \geq R_i} r^{2-n} \int_{B_{r}^+(x_i)} |\nabla u_i|^2 \dx \geq \eps.
\]
By Theorem \ref{th:uniformboundedness} \eqref{it:uniformboundboundary} we know that maps $u_i$ are uniformly bounded and thus by strong convergence of minimizing harmonic maps, Theorem~\ref{th:bsc}~\eqref{it:stongconvergenceboundary}, up to taking a subsequence, we find in the limit a minimizing harmonic map $u \in W^{1,2}(B_1^+(0),\n)$ and a limit point $x_0 \in T_{\frac{1}{2}}$ of $\{x_i\}_{i=1}^\infty$ such that 
\begin{equation}\label{eq:uniformconst}
 \inf_{r \geq 0} r^{2-n} \int_{B_{r}^+(x_0)} |\nabla u|^2 \dx \geq \frac{\eps}{2}.
\end{equation}
Now take any sequence $r_i \to 0$ such that, by \Cref{la:boundarytangentmaps}, we have
\[
 u_{r_i,x_0}(\cdot)\coloneqq u(x_0+r_i\,\cdot) \xrightarrow{i \to \infty} \Phi(\cdot) \quad \text{in } W^{1,2}(B_1^+(0)).
\]
From \Cref{la:bdtangentmaps} we find that $\Phi$ is constant. Thus, for small enough $r > 0$ we have 
\[
r^{2-n}\int_{B_{r}^+(x_0)} |\nabla u|^2 \dx = \int_{B_1^+(0)} |\nabla (u(x_0 + rx))|^2 \dx  < \frac{\eps}{2},
\]
contradicting \eqref{eq:uniformconst}.
\end{proof}

\begin{proof}[Proof of Theorem~\ref{th:boundaryregularity-constant}]
For any given $\eps > 0$ let $R_0(\eps)$ be the radius from Lemma~\ref{la:uniformsmallnessboundaryconstant} and set $\lambda \coloneqq R_0(\eps) /2$. For $x_0 \in (0,\lambda) \times \B_{\frac{1}{2}}(0)$ denote by $x_1 \in T_{1/2}$ the projection of $x_0$ onto $T_{1/2}$. Then for $\rho \coloneqq |x_0-x_1| < \lambda$ we have 
\[
 \int_{B_\rho(x_0)} |\nabla u|^2 \dx 
 \le \int_{B_{2\rho}(x_1)} |\nabla u|^2 \dx 
 \leq \eps (2 \rho)^{n-2} 
\]
due to Lemma~\ref{la:uniformsmallnessboundaryconstant}. Choosing $\eps > 0$ small, we infer smoothness of $u$ in $\B_{\rho/2}(x_0)$ from Theorem \ref{th:epsreg2} ($\eps$-regularity). Now $u$ is regular in $(0,\lambda) \times T_{1/2}$, since this region is covered by balls of this type. 
\end{proof} 
\subsection{Uniform boundary regularity for singular boundary data}

\begin{theorem}[Uniform boundary regularity for singular boundary data]\label{th:r4s}
Assume $s>\frac12$, $p>1$, and $sp> 1$. Let $\Omega \subset \R^n$ be a bounded domain with a $C^1$-boundary and let $\n$ be a~manifold with finite fundamental group. There are constants $R = R(\Omega)$ and $\eps = \eps(\Omega)$ such that the following holds: 

Take any minimizing harmonic map $u\in W^{1,2}(\Omega,\n)$ and denote its trace on $\partial \Omega$ by $\varphi$.

If for some $x_0 \in \partial \Omega$ and some $\rho_0 < R$ we have the estimate
\begin{equation}\label{eq:singularbryregsmallnesscondition}
\Lambda \coloneqq \sup_{B_\rho(y)\subset B_{\rho_0}(x_0)} \rho^{sp-(n-1)}[\varphi]^p_{W^{s,p}(B_\rho(y)\cap \partial \Omega)} \leq \eps
\end{equation}
then $u$ is smooth in $B_{\lambda\rho_0}(x_0) \cap \Omega$, where $\lambda$ is a uniform constant. 
\end{theorem}

We first prove Theorem \ref{th:r4s} for flat boundary.

\begin{proposition}\label{th:rs:flat}
There exist uniform constants $R$ and $\eps$ such that the following holds.
Take any minimizing harmonic map $u\colon B^+_1(0) \to \n$ with $\pi_1(\n)$ finite and denote the trace of $u$ on $T_1$ by $\varphi$. Let also  $s>\frac12$, $p>1$, $sp>1$. 

If for some $\rho_0 < R$ we have the estimate
\begin{equation}\label{eq:smallnessoffractionalnorm}
\rho_0^{sp-(n-1)}[\varphi ]_{W^{s,p}(T_{\rho_0})}^p \leq \eps
\end{equation}
then $u$ is smooth in 
\[
B_{\lambda \rho_0}(0) \cap \{x_n \geq \lambda\, \rho_0 /2\},
\]
where $\lambda$ is taken from Theorem~\ref{th:boundaryregularity-constant}.
\end{proposition}

\begin{remark}
 In particular one can take $\int_{T_{\rho_0}} |\nabla \varphi|^{n-1} \dd \mathcal{H}^{n-1} \leq \eps$ as the smallness condition in \eqref{eq:smallnessoffractionalnorm}.
\end{remark}

\begin{proof}[Proof of \Cref{th:rs:flat}]
Assume the claim is false. Then we find a sequence $\rho_k \to 0$, a sequence of minimizing harmonic maps $u_k \in W^{1,2}(B_1^+(0),\n)$ with trace $\varphi_k = u_k \Big |_{T_1}$ satisfying
\[
\rho_k^{sp-(n-1)}[\varphi_k ]_{W^{s,p}(T_{\rho_k})}^p \leq \frac{1}{k}
\]
however there is a singularity 
\[
y_k \in B_{\rho_k \lambda }(0) \cap \{(x_1, \dots, x_n) \in \R^n\colon x_n \ge \rho_k \lambda /2\}.
\]
We rescale, setting $v_k(x) \coloneqq u_k(\rho_k\, x)$, $\psi_k(x) \coloneqq \varphi_k(\rho_k\, x)$, and find $v_k \in W^{1,2}(B_1^+(0),\n)$, which is a minimizing harmonic map with trace $\psi_k$ on $T_1$ satisfying
\begin{equation}\label{eq:psikt0}
[\psi_k]_{W^{s,p}(T_1)}^p \leq \frac{1}{k}.
\end{equation}
(We note that here we used the scale invariance.) 
Moreover, $v_k$ has a singularity 
\[
z_k = \frac{1}{\rho_k}y_k \in B_{\lambda}(0) \cap \left\{x \in \R^n\colon\, x_n \geq \frac{\lambda}{2}\right\}.
\]

Thus, by strong convergence of minimizers, \Cref{th:bsc} \eqref{it:stongconvergenceboundary} and by convergence of singularities, \Cref{th:ALs2s}, up to taking a subsequence, we find in the limit a minimizing harmonic map $v \in W^{1,2}(B_1^+(0),\n)$ which in view of \eqref{eq:psikt0} is constant on $T_1$, but has a~singularity
\[
z = \lim_{k\rightarrow\infty} z_k \in B_{\lambda}(0) \cap \left\{x \in \R^n\colon\, x_n \geq \frac{\lambda}{2}\right\}.
\]
This contradicts \Cref{th:boundaryregularity-constant}.
\end{proof}

The proof of Theorem~\ref{th:r4s} follows now from Proposition~\ref{th:rs:flat} by a blowup argument.

\begin{proof}[Proof of Theorem~\ref{th:r4s}]
Take $\eps$, $\lambda$, $R_0$ from Proposition \ref{th:rs:flat}.
Assume on the contrary that we have a sequence of minimizing harmonic maps $u_i\colon \Omega \to \n$ with traces $u_i\big\rvert_{\partial \Omega} = \varphi_i$ such that for any $\rho_i<\frac 1i$ and $x_i\in \partial\Omega$
\[
 \rho_i^{sp-(n-1)}[\varphi_i]^p_{W^{s,p}(B_{\rho_i}(x_i)\cap \partial \Omega)} \le \eps
\]
and each $u_i$ has a singular point $y_i\in B_{\frac{1}{8}\lambda R_0 \rho_i}\cap\Omega$.

First of all we can assume that $|y_i - x_i| = \frac{1}{4}\lambda R_0 \rho_i$ and $\dist(y_i,\partial \Omega)> \frac18\lambda R_0 \rho_i $ (this can be done by choosing $\rho_i$ possibly smaller and moving $x_i$ to the projection of $y_i$ onto the boundary $\partial\Omega$).

Now define $v_i(x)= u_i \brac{ x_i + \rho_i x}$. We observe that, as in \Cref{sub:str8}, up to a rigid motion we may assume that for large enough $i$ we have 
\[
 \Omega_i \coloneqq \{x\in B_1(0)\colon x_n > \rho_i^{-1} \alpha(\rho_i x')\}
\]
where $\alpha\in C^1(\R^{n-1},\R)$ is the $C^1$ function which ''straightens out`` the boundary of $\Omega$ around $x_i\in \partial \Omega$ and $\alpha(0)=0$, $\nabla \alpha(0)=0$. (For large enough $i$ we have $\Omega_i\subset B_1^+(0)$). We note that $v_i\big\rvert_{\partial \Omega_i} = \psi_i$ are  still satisfying
\[
 [\psi_i]_{W^{s,p}(B_1(0)\cap \partial \Omega_i)}^p \le \eps.
\]
(Observe for this to hold true we only use the scale invariance of the expression). 

\begin{center}
	\begin{tikzpicture}
    \begin{scope}
    \clip (-2,-1.1) rectangle (2,1.1);
    \draw plot [smooth, tension=0.5] coordinates { (-3,1) (-2,0.25) (-1,0.3) (0,0) (1,0.3)  (3,2)};
	\draw (0,0) circle (1cm);
	\end{scope}
	\begin{scope}
	\clip (0,0) circle (1cm);
	\clip plot [smooth, tension=0.5] coordinates { (-3,1) (-2,0.25) (-1,0.3) (0,0) (1,0.3)  (3,2)};
	\fill[gray, opacity=0.1] (0,0) circle (1cm);
	\draw[blue!70, thick] plot [smooth, tension=0.5] coordinates { (-3,1) (-2,0.25) (-1,0.3) (0,0) (1,0.3)  (3,2)}; 
	\end{scope}
	\draw (-0.5,0.12) edge[->, bend right, blue] (-0.7,0.02); 
	\draw (0.72,0.12) edge[->, bend left, blue] (0.91,0.02); 
	\node[scale=0.75, below] at (0.5,0.8) {$\Omega_i$};
    \draw[dashed] (-2,0)--(2,0);
    \draw[fill=black] (0,0) circle (1pt);
 	\draw[fill=black] (-1,0) circle (1pt);
 	\draw[fill=black] (1,0) circle (1pt);
 	\node[scale=0.75, below] at (0,0) {$0$};
	\end{tikzpicture}
\end{center}

Thus, by \Cref{pr:strongconvergenceconvergingboundary}, we obtain a minimizing harmonic map $v\colon B_1^+(0)\to \n$ which by Theorem~\ref{th:ALs2s} has a singularity $y$ in $B_{\frac{1}{4} \lambda R_0}^+(0)$ with $y_n \geq \frac{1}{8} \lambda R_0$. But $v\big\rvert_{T_1} = \psi$ satisfies
\[
 [\psi_i]_{W^{s,p}(T_1(0))}^p \le \eps.
\]
This contradicts \Cref{th:rs:flat}.
\end{proof}

As a corollary from Theorem~\ref{th:r4s} we also obtain that the $(n-3)$-dimensional Hausdorff measure of the singular set of a minimizing harmonic map into manifolds with finite fundamental group is finite even if the boundary data is non-smooth.
\begin{corollary}\label{co:fewsg}
Let $\Omega \subset \R^n$ be a bounded domain with a $C^1$-boundary and let $\pi_1(\n)$ be finite.
If $u\colon \Omega \to \n$ is a minimizing harmonic map with trace $\varphi \in W^{s,p}(\partial \Omega,\n)$ with $s\in(\frac12,1]$ and $sp=n-1$,
then $\H^{n-3}(\sing u)<\infty$.
\end{corollary}
\begin{proof}
Since $\varphi\in W^{s,p}(\partial \Omega, \n)$ and for $sp=n-1$ this space is scale invariant, we can take $\eps>0$ from Theorem~\ref{th:r4s} and find a $\rho>0$ such that
\[
 \sup_{x_0\in\partial \Omega} [\vp]^p_{W^{s,p}(B_\rho(x_0)\cap\partial \Omega)} \le \eps
\]
and obtain a $\lambda\rho$-neighborhood of $\partial \Omega$ on which the minimizer $u$ is smooth. By Corollary~\ref{co:NabVal-meas-bound} we know that $\H^{n-3}(\sing u \cap \Omega_{\lambda\rho/2})<\infty$, where $\Omega_{\lambda\rho/2}\coloneqq\{x\in\Omega\colon \dist(x,\partial\Omega)> \lambda\rho/2\}$. This finishes the proof.  
\end{proof}
Later, in Theorem~\ref{th:almgren-lieb-high} we obtain a more precise bound on size of the singular set in terms of the boundary data.

\begin{example}
 As noted in the introduction, $\Psi\colon B^n(0)\rightarrow \S^2$ defined by $\Psi(x',x'') = \frac{x'}{|x'|}$ for $x'\in \R^3$ and $x''\in \R^{n-3}$ is a minimizing harmonic map with the singular set of dimension $(n-3)$. For $n\ge 4$ the singular set of $\Psi$ touches the boundary of $\partial B^n(0)$ but the trace $\psi = \Psi\big\rvert_{\partial B^n(0)}$ does not satisfy \eqref{eq:singularbryregsmallnesscondition}  in \Cref{th:r4s}. Indeed, $\psi \in W^{1,p}(\pl B^n(0),\S^2)$ for any $p < 3$ but not for $p = 3$, so if $n \geq 4$ then $\psi \not \in W^{s,p}(\partial B^n(0))$ for $sp=n-1$.
\end{example}

\subsection{Hot spots}\label{ss:hotspots}
The following is a generalization of the \cite[Theorem 2.3 (v)]{AlmgrenLieb1988}, it shows how to control the energy over an annulus centered in a point at the boundary by a term that depends on the boundary term $\vp$ but not on the minimizer. 

\begin{theorem}[Bridge theorem]\label{th:AL23}
Let $\Omega \subset \R^n$ be a bounded domain with a $C^1$-boundary let also $s>\frac12$, $p>1$, and $sp>1$. There exists a number $r_0 = r_0(\Omega)>0$ with the following property. 

For $x_0 \in \partial\Omega$ let $A{(\rho,r)}(x_0)\coloneqq \{ x \in \R^n \colon \rho < \dist(x,x_0) < r \}$. Suppose also that $\pi_1(\n)$ is finite and that $u\in W^{1,2}(\Omega,\n)$ is a minimizer in $\Omega$ having boundary map $\vp$. Then, whenever $0<r<r_0$, we have for any $1<\theta<2$
\begin{equation}\label{eq:rainbowinequality}
 r^{2-n} \int_{A{(r,2r)}(x_0)\cap \Omega} |\nabla u|^2 \dd x \le C + C
 r^{(sp-(n-1))\frac{\theta}{sp}}[\varphi]^\frac{\theta}{s}_{W^{s,p}(A{(\frac r2, \frac{5r}{2})}(x_0)\cap \partial \Omega
 )},
 \end{equation}
where $C=C(s,p, \theta) > 0$ is a constant independent of $\Omega$, $u$, and $\varphi$.
\end{theorem}

\begin{remark}\label{rem:hotspotsW12}
Replacing \Cref{th:uniformboundedness}~\eqref{it:uniformboundfractionalboundry} by \Cref{th:uniformboundedness}~\eqref{it:uniformboundboundary} we get in place of \eqref{eq:rainbowinequality}
\[
  r^{2-n} \int_{A{(r,2r)}(x_0)\cap \Omega} |\nabla u|^2 \dd x \le C + C r^{3-n}\int_{A{(\frac{r}{2},\frac{5r}{2})}(x_0)\cap \partial \Omega} |\nabla \vp|^2 \dhn.
\]
\end{remark} 

\begin{proof}[Proof of \Cref{th:AL23}]
Since $\Omega$ is bounded with a $C^1$-boundary, we may choose $r_0$ so small, that the boundary $\partial \Omega \cap B_{r_0}(x)$ is almost uniformly flat for all $x \in \partial \Omega$. 
Then we can find points $(x_i)_{i=1}^{M_{{1}}}$ (with $M_1$ a uniform combinatorial number) satisfying the following properties:
\[
\bigcup_{i=1}^{M_{{1}}} B_{r/4}(x_i) \cap \partial \Omega \supset A(r,2r)(x_0) \cap \partial \Omega
\] 
and
\[
\dist\brac{\partial \Omega, A(r,2r)(x_0) \setminus \bigcup_{i=1}^{M_{{1}}} B_{r/4}(x_i)} \geq \frac{r}{8}.
\]

\begin{center}
\begin{tikzpicture}
\clip (-3,-2.5) rectangle (3,3);
\begin{scope}[even odd rule]

\clip (0,-10) circle (10) (-3,-2.5) rectangle (3,3);
 \fill[color=gray!40] (0,0) circle (2);
 \fill[color=white] (0,0) circle (1);
\end{scope}

\draw[thick, shift={(0,-10)}] (106:10) arc (106:74:10);
\draw[shift={(0,-10)}] (96:10) circle (0.25);
\draw[shift={(0,-10)}] (97.66:10) circle (0.25);
\draw[shift={(0,-10)}] (99.33:10) circle (0.25);
\draw[shift={(0,-10)}] (101:10) circle (0.25);

\begin{scope}[yscale=1,xscale=-1]
\draw[shift={(0,-10)}] (96:10) circle (0.25);
\draw[shift={(0,-10)}] (97.66:10) circle (0.25);
\draw[shift={(0,-10)}] (99.33:10) circle (0.25);
\draw[shift={(0,-10)}] (101:10) circle (0.25); 
\end{scope}
\node[scale=0.7] at (0,1.5) {$A(r,2r)$};
\node[scale=0.7] at (2.5,0.7) {$\Omega$};
\node[scale=0.7] at (2.8,-0.2) {$\partial \Omega$};
\end{tikzpicture}

\end{center}
Also, there exists a uniform number $M_2$ such that we always find $(y_j)_{j=1}^{M_2}$ so that $B_{r/6}(y_j)  \subset \Omega$ and
\[
\Omega \supset \bigcup_{j=1}^{M_{2}} B_{r/8}(y_j) \supset \Omega \cap A(r,2r)(x_0) \setminus B_{r/8}(\partial \Omega).
\]
From Theorem \ref{th:uniformboundedness}~\eqref{it:uniformboundfractionalboundry} we obtain for any $1<\theta<2$
\[
  r^{2-n}\int_{B_{r/4}(x_i) \cap \Omega} |\nabla u|^2 \dx \aleq 1 + r^{(sp-(n-1))\frac{\theta}{sp}}[\varphi]^\frac{\theta}{s}_{W^{s,p}(B_{r/2}(x_i) \cap \partial \Omega)},
\]
%
%
%
Summing over all the $x_i$ we obtain 
\begin{equation}\label{eq:intnablaubd}
r^{2-n}\int_{A(r,2r)(x_0) \cap (\Omega \cap B_{r/8}(\partial \Omega))} |\nabla u|^2 \dx\aleq M_1 + M_1 r^{(sp-(n-1))\frac{\theta}{sp}}[\varphi]^\frac{\theta}{s}_{W^{s,p}(B_{r/2}(x_i) \cap \partial \Omega)}.
\end{equation}
From Theorem \ref{th:uniformboundedness}~\eqref{it:uniformboundinterior} we obtain
\[
r^{2-n}\int_{B_{r/8}(y_j)} |\nabla u|^2 \dx \aleq  \frac{\frac{r}{8}}{r-\frac{r}{8}} \aeq 1.
\]
and thus summing over $y_j$,
\begin{equation}\label{eq:intnabaluint}
r^{2-n}\int_{A(r,2r)(x_0) \cap (\Omega \setminus B_{r/8}(\partial \Omega))} |\nabla u|^2 \dx \aleq M_2 .
\end{equation}
Together \eqref{eq:intnablaubd} and \eqref{eq:intnabaluint} give the claim.
\end{proof}

With this uniform energy bound, we can actually show that \emph{boundary energy in small balls cannot induce distant singularities} \cite[Thm.~2.4]{AlmgrenLieb1988}. In the contradiction argument, the hot spot tends to zero in size and disappear completely in the limit. 

\begin{theorem}[regularity away from ``hot spots'']\label{th:AL24}
Let $s>\frac12$, $p>1$, and $sp> 1$. There is an $\eps(n) > 0$ such that the following holds. Suppose $\n$ is a manifold with finite fundamental group, $u \in W^{1,2}(\B^+_1(0),\n)$ is a minimizer with trace $\vp$ on $T_1$, and 
\begin{equation}\label{eq:regawayhotspotscondition}
 [\varphi]^{p}_{W^{s,p}(T_{1} \setminus \B_\eps(x_0))} \le \eps
\end{equation}
for some ball $\B_\eps(x_0)$. Then $u$ is smooth in
\[
 T_{1/2} \times (\mu/2,\mu),
\]
where $\mu > 0$ is a small constant depending on $n$ and $\cN$.
\end{theorem}

\begin{proof}
We argue by contradiction. Assume that $u_i \colon B^+_1(0) \to \n$ is a sequence of minimizers with boundary maps $\vp_i$ such that
\[
[\varphi_i]^{p}_{W^{s,p}(T_{1} \setminus \B_{\eps_i}(x_i))} \le \eps_i
\] 
for a sequence of balls $B_{\eps_i}(x_i)$ and $\eps_i \xrightarrow{i \to \infty} 0$. For $r_i\le \eps_i$ we have
\begin{equation}\label{eq:powerisfun}
 r_i^{(sp-(n-1))}[\varphi_i]^{p}_{W^{s,p}(T_{1} \setminus \B_{r_i}(x_i))} \le \eps_i r_i^{(sp-(n-1))}
\end{equation}
Setting $r_i \coloneqq (\eps_i)^{\frac{1}{\kappa}}$, where
\[
 \kappa = \frac{sp}{\theta} + (n-1)
\]
and $\theta\in(1,2)$ we have $\kappa>0$, (thus in particular $r_i < \eps_i$), and from \eqref{eq:powerisfun} we get 
 \begin{equation}\label{eq:varphitrace}
 \brac{r_i}^{(sp-(n-1))\frac{\theta}{sp}}
 [\varphi_i]^{\frac{\theta}{s}}_{W^{s,p}(T_1\setminus B_{r_i}(x_i))} \le r_i
 \end{equation}
where $r_i \xrightarrow{i \to \infty} 0$, and up to taking a subsequence, $r_i < 2^{-i}$.

Now, we assume (by contradiction) that each $u_i$ has at least one singularity $y_i \in T_{1/2} \times (\mu/2,\mu)$.

By \Cref{th:AL23}, for large enough $i$ and for any $r \ge 2^{-i}$
\[
r^{2-n}\int_{B^+_1 \cap A{(r,2r)}(x_i)}|\nabla u_i|^2 \dx \le C.
\]
Thus, for every $1 \le k \le i$,
\[
\int_{\B^+_1 \cap A{(2^{-k},2^{-k+1})}(x_i)}|\nabla u_i|^2 \dd x \le C\, 2^{-k(n-2)}.
\]
Up to taking another subsequence we can assume that $x_i \to x_0$, and for convenience also $|x_i-x_0| \le 2^{-i}$. Then, from the above estimate we have
\begin{equation}
\label{eq:bound-away-from-p0}
\int_{\B^+_{4/5} \setminus \B_{2^{-2i}}(x_0)}|\nabla u_i|^2 \dd x \le C\, \sum_{k=1}^{i} 2^{-k(n-2)} \le C.
\end{equation}
In particular by a diagonal argument and the strong convergence of minimizers, Theorem~\ref{th:bsc}, we obtain a minimizer $u$ in $W^{1,2}(\B_{3/4}^+(0) \setminus \B_{r}(x_0),\n)$ for any $r > 0$. Moreover, its trace, which we shall call $\varphi \in W^{1,2}_\loc(T_{1} \setminus \{ x_0 \},\n)$ is the limit of $\varphi_i$. Observe that $\varphi$ is constant on $T_{1}$ by \eqref{eq:varphitrace}. 

Moreover, by Theorem~\ref{th:ALs2s} the sequence of singular points $y_i$ can be assumed to converge to a singular point of $u$, which we call $y \in T_{1/2} \times [\mu/2,\mu]$. 

To reach a contradiction with Theorem~\ref{th:boundaryregularity-constant}, one needs to solve the subtle issue of minimality around $x_0$. To this end, we note that by \eqref{eq:bound-away-from-p0} the energy $\int_{\B^+_{3/4} \setminus \B_{r}(x_0)} |\nabla u_i|^2 \dx$ is uniformly bounded for all $r>0$, and hence by monotone convergence $u \in W^{1,2}(\B_{3/4}^+(0),\n)$. In view of Lemma~\ref{lem:removesg} below, the singularity $x_0$ is removable, and so $u$ is a minimizing harmonic map in $\B_{3/4}^+(0)$ with a constant boundary map $\vp$. This contradicts the singularity at $y$. 
\end{proof}

To complete the proof of Theorem~\ref{th:AL24}, we need the following removability lemma. 

\begin{lemma}[Removability of points for minimizing harmonic maps]
\label{lem:removesg}
Assume that $\pi_1(\n)$ is finite and $u \in W^{1,2}(\B_1^+(0),\n)$ is a minimizer away from the origin, i.e., for any $\delta > 0$ and any $v \in W^{1,2}(\B_1^+(0),\n)$ satisfying $v = u$ on $\partial \B_1^+(0)$ and $v \equiv u$ on $\B_\delta^+(0)$ we have 
\begin{equation}\label{eq:dmin}
\int_{\B_1^+(0) \setminus \B_{\delta}^+(0)} |\nabla u|^2 \dd x \le \int_{\B_1^+(0) \setminus \B_{\delta}^+(0)} |\nabla v|^2 \dd x.
\end{equation}
Then $u$ is a minimizing harmonic map in all of $\B_1^+(0)$.
\end{lemma}

\begin{proof}
Let $w \in W^{1,2}(B_1^+(0),\n)$ with $u \equiv w$ on $\partial B_1^+(0)$. We need to show that 
\begin{equation}\label{eq:min}
\int_{B_1^+(0)} |\nabla u|^2 \dx \leq \int_{B_1^+(0)} |\nabla w|^2 \dx.
\end{equation}
Let us lift $w=\pi \circ \tilde w$, $u=\pi\circ \tilde u$ as in \Cref{th:BethuelChiron}, where $\tilde w, \tilde u \in W^{1,2}(B_1^+(0),\widetilde \n)$ and $\widetilde \n$ is the universal cover of $\n$ and we have a.e. $|\nabla w| = |\nabla \tilde w|$ and $|\nabla u| = |\nabla \tilde u|$. For $\delta > 0$, let $\eta_\delta \in C_c^\infty(\B_{2\delta}(0))$ be a standard cut-off function satisfying $\eta_\delta \equiv 1$ in $\B_{\delta}(0)$ and $|\nabla \eta_\delta| \le \frac{2}{\delta}$. We define $\tilde{v}_\delta\in W^{1,2}(\B_1^+(0),\R^d)$ as 
\[
\tilde{v}_\delta \coloneqq (1-\eta_\delta) \tilde w + \eta_\delta \tilde u;
\]
this function satisfies $\tilde{v}_\delta = \tilde u$ on $\partial B_1^+(0)$, $\tilde{v}_\delta \equiv \tilde u$ in $B_\delta^+(0)$ and $\tilde{v}_\delta \equiv \tilde w$ in $B_1^+ \setminus B_{2\delta}(0)$. Since $\widetilde \n$ is compact and simply connected we may apply Theorem~\ref{th:extensionthm}, applied in $B_{2\delta}^+(0) \setminus B_\delta(0)$ and find $\tilde w_\delta \in W^{1,2}(B_{2\delta}^+(0)\setminus B_\delta(0),\widetilde \n)$ such that $\tilde w_\delta = \tilde v_\delta$ on $\partial (B_{2\delta}^+(0)\setminus B_\delta(0))$ and
\begin{equation}\label{eq:vwcomparisontilda}
 \int_{B_{2\delta}^+(0)\setminus B_{\delta}(0)} |\nabla \tilde w_\delta| \dx \le \int_{B_{2\delta}^+(0)\setminus B_{\delta}(0)} |\nabla \tilde v_\delta|^2 \dx.
\end{equation}
We extend $\tilde w_\delta$ to the whole half-ball $B_1^+(0)$ by setting $\tilde w_\delta \equiv \tilde u$ in $B_\delta^+(0)$ and $\tilde w_\delta = \tilde w$ in $B_1^+(0)\setminus B_{2\delta}(0))$. Now let us define $w_\delta \coloneqq \pi \circ \tilde w_\delta \in W^{1,2}(B_1^+(0),\n)$ and note that we have
\[
w_\delta = \begin{cases}
u \quad \mbox{in $B_\delta^+(0)$}\\
w \quad \mbox{in $B_1^+(0) \setminus B_{2\delta}(0)$}\\
u \quad \mbox{on $\partial B_1(0)$}
\end{cases}
\]
and from \eqref{eq:vwcomparisontilda}
\[
 \int_{B_{2\delta}^+(0)\setminus B_\delta (0)} |\nabla w_\delta|^2 \dx = \int_{B_{2\delta}^+(0)\setminus B_\delta (0)} |\nabla \tilde w_\delta|^2 \dx \le \int_{B_{2\delta}^+(0)\setminus B_\delta (0)} |\nabla \tilde v_\delta|^2 \dx. 
\]
In particular, $\tilde{w}_\delta$ is a competitor in the sense of \eqref{eq:dmin}, and we have 
\[
\begin{split}
\int_{B_1^+(0) \setminus B_{\delta}(0)} |\nabla u|^2 \dx &\leq \int_{B_1^+(0) \setminus B_{\delta}(0)} |\nabla w_\delta|^2 \dx\\
 &= \int_{B_1^+(0) \setminus B_{2\delta}(0)} |\nabla w_\delta|^2 \dx + \int_{B_{2\delta}^+(0) \setminus B_{\delta}(0)} |\nabla w_\delta|^2\dx \\
&\leq  \int_{B_1^+(0) \setminus B_{2\delta}(0)} |\nabla w|^2 \dx + C\int_{B_{2\delta}^+(0)} |\nabla \tilde{v}_\delta|^2 \dx.\\
\end{split}
\]
Since $u$, and $ w \in W^{1,2}(\B_1^+(0))$ using the absolute continuity of the integral we find that
\begin{equation}\label{eq:almminlim}
\int_{B_1^+(0)} |\nabla u|^2 \dx \leq \int_{B_1^+(0)} |\nabla w|^2 \dx + C\, \liminf_{\delta \to 0}\int_{B_{2\delta}^+(0)} |\nabla \tilde{v}_\delta|^2 \dx.
\end{equation}
Now
\[
\int_{B_{2\delta}^+(0)} |\nabla \tilde{v}_\delta|^2 \dx 
\aleq \frac{1}{\delta^2} \int_{B_{2\delta}^+(0)} |\tilde u- \tilde w|^2 \dx + \int_{B_{2\delta}^+(0)} |\nabla \tilde u|^2 \dx + \int_{B_{2\delta}^+(0)} |\nabla \tilde w|^2 \dx.
\]
Observe that we are in dimension $n \geq 3$ and since $\tilde u, \tilde w$ have values in the compact manifold $\widetilde \n$, we get
\[
\frac{1}{\delta^2} \int_{B_{2\delta}^+(0)} |\tilde u-\tilde w|^2 \dx\aleq \delta.
\]
Thus, using again the absolute continuity of the integral and that $\tilde u, \tilde w \in W^{1,2}$ we find
\[
\lim_{\delta \to 0} \int_{B_{2\delta}^+(0)} |\nabla \tilde{w}_\delta|^2 \dx = 0.
\]
Plugging this into \eqref{eq:almminlim} we conclude.
\end{proof}

In the applications, we will use the following global version of Theorem~\ref{th:AL24} (see \cite[Cor.~2.7]{AlmgrenLieb1988}). 

\begin{theorem}[boundary regularity with hot spots]
\label{th:hot-spots}
Let $s\in(\frac12,1]$, $p>1$, and $sp>1$. Let us also assume that $\pi_1(\n)$ is finite. For each bounded domain $\Omega \subset \R^n$ with $C^1$-boundary, there are small constants $\sigma, \eps, \lambda > 0$, $\Lambda > 1$, ($\sigma$ depending on $\cN$ and the geometry of $\Omega$, the others only on $n$ and $\cN$) so that the following statement holds true for any minimizer $u \in W^{1,2}(\Omega,\n)$ with trace $\varphi \coloneqq u \Big |_{\partial \Omega}$.

For any singular point $y \in \sing u$ with $r \coloneqq \dist(y,\partial \Omega) < \sigma$ and for any ball $B \subset \R^n$ with radius $\lambda r$, we have
\begin{equation}\label{eq:bdryreghotspotsfractionalnorm}
 r^{sp-(n-1)}[\vp]^p_{W^{s,p}(\partial \Omega \cap \brac{B_{\Lambda r}(y) \setminus B})} \geq \eps. 
\end{equation}
\end{theorem}


\begin{proof}
In principle, this is a rescaled version of Theorem~\ref{th:AL24}, only with a non-flat boundary. 

Assume to the contrary that we have a sequence of minimizing harmonic maps, $u_i \in W^{1,2}(\Omega,\n)$ with trace $\varphi_i \coloneqq u_i \Big |_{\partial\Omega}$ and singularities $y_i \in \sing u_i$ with $r_i \coloneqq \dist(y_i,\partial \Omega) < \frac{1}{i}$, but there exists a ball $B_i \subset \R^n$ of radius $\lambda r_i$ with
\[
r_i^{sp-(n-1)}[\vp_i]^p_{W^{s,p}(\partial \Omega \cap \brac{B_{\Lambda r_i}(y_i) \setminus B_i})}  < \eps. 
\]
After rescaling $v_i(x) \coloneqq u_i(y_i + r_i x)$, in view of \Cref{th:bsc} and \Cref{pr:strongconvergenceconvergingboundary} we find in the limit a map which contradicts \Cref{th:AL24} .
\end{proof}

\section{Hardt and Lin's stability of singularities for \texorpdfstring{$n \geq 3$}{n >= 3}}
\label{sec:stability}
This section is concerned with stability of singularities. By this we mean that if two boundary maps $\vp, \vp' \colon \partial \Omega \to \S^2$ are \emph{close} in the right Sobolev norm, then the singularities of their corresponding minimizers $u, u' \colon \Omega \to \S^2$ are \emph{close} as well. Since minimizers are in general non-unique, the precise statement is a little more subtle --- e.g., by assuming uniqueness a priori. 

In any case, let us discuss the right notions of \emph{closeness}. In dimension $n = 3$, when the singular set consists of finitely many points, Hardt and Lin \cite{HL1989} considered the Lipschitz norm for boundary data, and showed that small perturbations do not change the number of singularities. Moreover, they constructed a bi-Lipschitz diffeomorphism $\eta \colon \Omega \to \Omega$ (close to identity in Lipschitz norm) such that $u$ is close to $u' \circ \eta$ in some $C^\beta$ norm. These results were recently extended to the case of $W^{1,2}$-perturbations of boundary data by Li \cite{Li18}. 

In higher dimension $n \ge 3$, we consider perturbations in the $W^{1,n-1}$ norm. Since the singular set is a rectifiable set of codimension $3$, we prove its stability with respect to a~version of Wasserstein metric (see \cite{Villani}) also referred to as the flat metric: 
\begin{equation}
\label{eq:wasserstein}
d_W(\mu,\nu) = \sup \left \{ \int_{\R^n} h \dd \mu - \int_{\R^n} h \dd \nu \colon \ h \colon \R^n \to \R, \ |h| \le 1, \ |\nabla h| \le 1 \right \},
\end{equation}
i.e., we show that the distance between measures $\cH^{n-3} \llcorner \sing u$ and $\cH^{n-3} \llcorner \sing u'$ is small. Since taking $h \equiv 1$ in the definition yields 
\[
|\mu(\R^n) - \nu(\R^n)| \le d_W(\mu,\nu),
\]
we obtain in particular that the size of the singular set $\cH^{n-3}(\sing u)$ is also stable under $W^{1,n-1}$-perturbations of boundary data. 

\begin{theorem}[stability of singularities]
\label{th:global-stability}
Let $\Omega \subset \R^n$ be a bounded domain with a $C^1$-boundary and let $u \in W^{1,2}(\Omega,\S^2)$ be a~minimizer with boundary data $\varphi \in W^{1,n-1}(\partial \Omega,\S^2)$. If $u_k$ is a~sequence of minimizers with boundary data $\varphi_k$ and 
\begin{equation}\label{eq:thgs:convergence}
u_k \to u \text{ in } W^{1,2}(\Omega), 
\quad 
\varphi_k \to \varphi \text{ in } W^{1,n-1}(\partial \Omega), 
\end{equation}
then 
\[
\H^{n-3} \llcorner \sing u_k 
\xrightarrow{d_W}
\H^{n-3} \llcorner \sing u,
\]
in particular $\H^{n-3}(\sing u_k) \to \H^{n-3}(\sing u)$. 
\end{theorem}
Under the assumption of uniqueness, we obtain immediately Theorem~\ref{th:stabilitynD}

\begin{proof}[Proof of Theorem~\ref{th:stabilitynD}]
For the sake of contradiction, let $u_k$ be a sequence of minimizers with boundary data $\vp_k$ satisfying $\vp_k \to \vp$ in $W^{1,n-1}(\partial \Omega, \S^2)$ but not satisfying the claim. Taking a subsequence, by Theorem~\ref{th:bsc} \eqref{it:strongconvergencegeneral} we may assume that $u_k$ converges in $W^{1,2}(\Omega,\S^2)$ to a minimizer $\overline{u}$ with boundary data $\vp$. By uniqueness, $\overline{u} = u$ and Theorem~\ref{th:global-stability} implies that $\H^{n-3} \llcorner \sing u_k$ tends to $\H^{n-3} \llcorner \sing u$. Thus, we obtain a contradiction for large enough~$k$. 
\end{proof}

{We would like to emphasize that here we work only with the target manifold $\n = \S^2$ and we could not generalize the argument to the case of simply connected manifolds. The reason behind this is the Brezis--Coron--Lieb \cite{BCL1986} classification of tangent maps for $\S^2$ and its generalization Corollary~\ref{co:uniquetangent}, from which we deduce that all $(n-3)$-dimensional tangent maps have the same energy density.

The singularities were also classified for $\n=\S^3$ by Nakajima \cite{Nakajima}, see also \cite{LinWang2006}. It was shown by Schoen and Uhlenbeck \cite{SU84} that in the case when the target manifold is a $3$-dimensional sphere, then $\dim_{\mathcal{H}} \sing u \leq n-4$. This is why it is possible to extend Theorem~\ref{th:stabilitynD}, following the same arguments, in the terms of the stability of the highest stratum, which in this case is a $(n-4)$-dimensional set. In other words, it is possible to} consider $\mathcal{H}^{n-4} \llcorner \sing u$ in place of $\mathcal{H}^{n-3} \llcorner \sing u$.

\subsection{Outline}

In analogy to the original argument of Hardt and Lin \cite{HL1989}, the heart of the argument lies in the special case when $u$ is the tangent map $\Psi$ as in \eqref{eq:def-psi} given by 
\[
\R^3 \times \R^{n-3} \ni (x',x'')
\xmapsto{\quad \Psi \quad}
\frac{x'}{|x'|} \in \S^2.
\]
Establishing a stability result for the singular set (which for $\Psi$ is an $(n-3)$-dimensional plane) requires some care. Here we adopt the notion of $\delta$-flatness introduced in \cite{Mis18}, which combines topological and analytic conditions for a minimizer to be \emph{close} to $\Psi$. In Section~\ref{sec:stab-flatness} we cite the necessary results and also show that the condition for $\delta$-flatness is stable under $W^{1,2}$-perturbations of the minimizer (Proposition~\ref{prop:stability-of-flatness}). 

With this in hand, we are able to modify the original arguments of Naber and Valtorta \cite{NabVal17} and improve on them in the special case of maps into $\S^2$. In result, we obtain the stability result for $\Psi$ mentioned earlier (Lemma \ref{lem:local-stability}). 

Since around $\mathcal{H}^{n-3}$-almost every singular point, any energy minimizer is close to the map $\Psi$ (composed with an isometry), this stability result can be seen as a local case for Theorem~\ref{th:global-stability}. Indeed, in Section~\ref{sec:stab-global} we cover most of the singular set of $u$ by balls on which Lemma~\ref{lem:local-stability} can be applied. An argument based on Proposition~\ref{prop:stability-of-flatness} then shows that the same covering works for both $\sing u$ and $\sing u_k$, and the global estimate follows. 


\subsection{Behavior of top-dimensional singularities}
\label{sec:stab-flatness}
This subsection gathers the results of \cite{Mis18}, which allow us to study further the top-dimensional part of the singular set. 

Recall the tangent map $\Psi$ from \eqref{eq:def-psi} with its energy density $\Theta$ from \eqref{eq:def-Theta}, and the rescaled energy $\theta_u$ from \eqref{eq:thetau}. We introduce the following property, which basically says that $u$ is close to $\Psi$ (up to an isometry) on the ball $B_r(x)$. 
\begin{definition}[$\delta$-flatness]
\label{df:mis-flatness}
We say that an energy minimizer $u \colon \Omega \to \S^2$ is $\delta$-flat in the ball $\B_r(x) \subset \Omega$ if 
\begin{enumerate}
\item
$x$ is a singular point of $u$ and $\Theta \le \theta_u(x,0) \le \theta_u(x,r) \le \Theta + \delta$, 
\item
for some $(n-3)$-dimensional affine plane $L$ through $x$, $\sing u \cap B_r(x) \subset B_{r/10}(L)$, 
\item
$u$ restricted to $(x+L^\perp) \cap \partial \B_{r/2}(x)$ has degree $\pm 1$ as a map to $\S^2$. 
\end{enumerate}
\end{definition}
Note that this definition is scale-invariant in the following sense: $u$ is $\delta$-flat in $B_r(x)$ if and only if the rescaled map $\overline{u}(y) = u(x+ry)$ is $\delta$-flat in $B_1$. Also note that $u$ is smooth outside the tube around $L$ by and thus the degree is well-defined. 

Definition~\ref{df:mis-flatness} is strongly reminiscent of \cite[Def.~4.3]{Mis18}. There, Reifenberg flatness is additionally assumed, but it follows from 

\begin{lemma}[{{\cite[Lemma~5.1]{Mis18}}}]
\label{lem:mis-forced-sing}
Assume that $\sing u \cap B_r(x) \subset B_{\eps r}(L)$ for some $0 < \eps < \frac 12$ and some $(n-3)$-dimensional plane $L$ through $x$. Moreover, assume that $u$ restricted to $(x+L^\perp) \cap \partial \B_{r/2}(x)$
has degree $\pm 1$ as a map from $\S^2$ to itself. Then 
\[
L \cap B_{(1-\eps)r}(x) \subset \pi_L(\sing u \cap B_r(x)). 
\]
Here and henceforth, $\pi_L$ denotes the nearest-point projection from $\R^n$ onto $L$.
\end{lemma}

In particular, it follows from our definition of $\delta$-flatness that $L \cap B_r(x) \subset B_{r/5}(\sing u)$. This allows us to apply the results of \cite{Mis18}.

The first important point is that around each point in top-dimensional part of the singular set, $\sing_* u$, the map $u$ satisfies the $\delta$-flatness property on sufficiently small balls.

\begin{lemma}[{{\cite[Cor~5.4,~Lem~5.8]{Mis18}}}]
\label{lem:mis-some-flatness}
Let $x \in \sing_* u$. Then for each $\delta > 0$ there is $r_0 > 0$ such that $u$ is $\delta$-flat in $B_r(x)$ for all $r \in (0,r_0]$. 
\end{lemma}

Below we also note various consequences of $\delta$-flatness proved in \cite{Mis18}. For simplicity, we only deal with the unit ball, but one can easily obtain the corresponding statement for any ball using the scale-invariance. 

\begin{theorem}
\label{th:mis-all}
For each $\eps > 0$ there is $\delta > 0$ such that the following holds. If $u$ is $\delta$-flat in $B_2$, then 
\begin{enumerate}

\item for some tangent map of the form $\overline{\Psi} = \Psi \circ \tau$ (with $\Psi$ as in \eqref{eq:def-psi} and some linear isometry $\tau$) we have 
\[
\| u - \overline{\Psi} \|^2_{W^{1,2}(B_1)} \le \eps, 
\]

\item for the $(n-3)$-dimensional linear plane $L' \coloneqq \sing \overline{\Psi}$, 
\[
\sing u \cap B_1 \subset B_\eps (L')
\quad \text{and} \quad
L' \cap B_{1-\eps} \subset \pi_{L'}(\sing u \cap B_1), 
\]

\item all singular points in $B_{1}$ lie in the top-dimensional part $\sing_* u$, and $u$ is $\eps$-flat in each of the balls $B_r(z)$ with $z \in \sing u \cap B_{1}$ and $0 < r \le 1/2$. 

\end{enumerate}
\end{theorem}

\begin{proof}
Due to Lemma~\ref{lem:mis-forced-sing}, we may apply the results of \cite{Mis18} directly. 

Points (1) and (2) are essentially the content of \cite[Lem~5.3]{Mis18}, except for the condition $L \cap B_{1-\eps} \subset \pi_L(\sing u \cap B_1)$, which again follows from Lemma~\ref{lem:mis-forced-sing}. Point (3) comes from combining \cite[Prop~5.6]{Mis18} and its corollary \cite[Cor~5.7]{Mis18}. 
\end{proof}

The last ingredient is another consequence of the arguments in \cite{Mis18}. It is to some extent the higher-dimensional analogue of \cite[Theorem 1.8, (2)]{AlmgrenLieb1988}.

\begin{proposition}[Stability of $\delta$-flatness]
\label{prop:stability-of-flatness}
For each $\eps > 0$ there is $\delta > 0$ such that the following holds. If $u$ is $\delta$-flat in the ball $B_1$ and $u_k \xrightarrow{k \to \infty} u$ in $W^{1,2}(B_1)$, then for $k$ large enough there is $x_k \in \sing u_k \cap B_\eps$ such that $u_k$ is $\eps$-flat in the ball $B_{1-\eps}(x_k)$. 
\end{proposition}

\begin{proof}
Choose $\eps'(n,\eps) > 0$ small enough, more precisely such that 
\[
\eps' < \eps/2, 
\quad 
(1-2\eps')^{2-n} (\Theta + \eps/2) \le \Theta + \eps.
\]
By taking $\delta$ small enough, we may assume by Theorem~\ref{th:mis-all} that
\[
\sing u \cap B_{1-\eps'/2} \subset B_{\eps'/2} (L) 
\]
for some $(n-3)$-dimensional linear plane $L$. 
Since singular points converge again to singular points, Theorem~\ref{th:ALs2s}, we have for all large $k$,
\begin{equation}\label{eq:ukclosetoL}
\sing u_k \cap B_{1-\eps'} \subset B_{\eps'/2} (L) 
\end{equation}
By \cite[Proposition 4.6]{SU1}, we have locally uniform convergence outside the singular set, and thus
\[
u_k \rightrightarrows u \quad \text{ in } B_{1-\eps'} \setminus B_{\eps'/2}(L).
\] 
In particular, $u_k$ and $u$ restricted to $L^\perp \cap \partial B_{1/2}$ have the same homotopy type for large $k$.

By Lemma~\ref{lem:mis-forced-sing}
\[
 L \cap B_{1-2\eps'} \subset  \pi_L(\sing u_k \cap B_{1-\eps'}).
\]
Combined with \eqref{eq:ukclosetoL} this means that $u_k$ has many singular points near $L$. Since $\H^{n-3}$-a.e. singular point lies in $\sing_* u$ (see \eqref{eq:Sjdim}), we find $x_k \in \sing_* u_k$ with $|x_k| \le \frac{1}{2}\eps'$. In particular, we already have $\theta_{u_k}(x_k,0) = \Theta$, by Corollary~\ref{co:uniquetangent}. 

The last condition to show is
$\theta_{u_k}(x_k,1-\eps) \le \Theta + \eps$. By strong convergence, for large enough $k$,
\[
\int_{B_{1-\eps'}} |\nabla u_k|^2 \le \eps/4 + \int_{B_1} |\nabla u|^2.
\]
Thus 
\[
(1-2\eps')^{2-n} \int_{B_{1-2\eps'}(x_k)} |\nabla u_k|^2 
\le (1-2\eps')^{2-n} \left( \eps/4 + \int_{B_1} |\nabla u|^2 \right) 
\le (1-2\eps')^{2-n} (\Theta + \delta + \eps/4), 
\]
which does not exceed $\Theta + \eps$ if only $\delta \le \eps/4$. By the monotonicity formula, we conclude that $\theta_{u_k}(x_k,1-\eps) \le \theta_{u_k}(x_k,1-2\eps') \le \Theta + \eps$ and hence that $u_k$ is $\eps$-flat in the ball $B_{1-\eps}(x_k)$.
\end{proof}

\subsection{Local case}
\label{sec:stab-local}

The lemma below can be thought of as a local version of the stability theorem. It says that perturbing the tangent map $\Psi$ a little does not change the size of the singular set much. 

\begin{lemma}
\label{lem:local-stability}
For each $\eps > 0$ there is $\delta > 0$ such that the following is true. If $u \colon B_{80} \to \S^2$ is energy minimizing and $\delta$-flat in $B_{80}$ (see Definition \ref{df:mis-flatness}), then 
\[
(1-\eps) \omega_{n-3} \le \H^{n-3}(\sing u \cap B_1) \le (1+\eps) \omega_{n-3}. 
\]
Here $\omega_{n-3} = \H^{n-3}(\sing \Psi \cap B_1) $ is the volume of the $(n-3)$-dimensional ball.
\end{lemma}

It is natural that in order to conclude the right estimate on $\B_1$, one needs to make assumptions on a larger ball. The ball $\B_2$ would be enough here, but working with $\B_{80}$ saves us from an additional covering argument. 

\begin{proof}

The lower bound follows from a simple topological argument. Fix $\eps' = \frac{\eps}{n-2}$, then apply Theorem~\ref{th:mis-all} to find that there is an $(n-3)$-dimensional linear plane $L$ such that 
\[
L \cap B_{1-\eps'} \subset \pi_L( \sing u \cap B_1 ),
\]
provided $\delta$ is small enough. Since the orthogonal projection $\pi_L$ is $1$-Lipschitz, this shows 
\[
\H^{n-3}(\sing u \cap B_1) \ge \H^{n-3}(L \cap B_{1-\eps'}) = (1-\eps')^{n-3} \omega_{n-3} \ge (1-\eps) \omega_{n-3}. 
\]

A rough upper bound follows from Naber and Valtorta's work \cite{NabVal17}, namely Corollary~\ref{co:NabVal-meas-bound},
\begin{equation}
\label{eq:weak-meas-bound}
\H^{n-3}(\sing u \cap B_r(z)) \le C(n) r^{n-3}
\end{equation}
for each ball $B_{2r}(z) \subset B_{2}$.

To obtain the sharp upper bound, we will follow the general outline of Naber and Valtorta's work \cite[Sec.~1.4]{NabVal17}. When the target manifold is $\S^2$, the original reasoning can be made significantly easier due to topological control of singularities (analyzed in \cite{Mis18}). In particular, we we will be able to apply Rectifiable Reifenberg Theorem~\ref{th:NabVal-reifenberg} to the whole singular set in $B_1$, without decomposing it into many pieces. 

With $\delta_1 > 0$ to be fixed later, by Theorem~\ref{th:mis-all} we can choose $\delta$ small enough so that all singular points in $B_{40}$ lie in the top-dimensional part $\sing_* u$, moreover $u$ is also $\delta_1$-flat in each ball $B_{r}(z)$ with $z \in \sing u \cap B_{40}$ and $0 < r \le 20$. 

We can now apply the $L^2$-best approximation Theorem~\ref{th:NabVal-best-approx} on these balls; for simplicity, we consider the ball $B_{10}$ first. By Theorem~\ref{th:mis-all}, $u$ is $W^{1,2}$-close to a map of the form $\overline{\Psi} = \Psi \circ q$ (with $\Psi$ as in \eqref{eq:def-psi} and some linear isometry $\tau$). Note that $\overline{\Psi}$ lies in $\mathrm{sym}_{n,0}$ and the value 
\[
\eps_0 \coloneqq \dist_{L^2(B_{10})} (\overline{\Psi}, \ \mathrm{sym}_{n,k+1}) > 0
\]
depends only on the dimension $n$ (not on the choice of $\tau$). Hence, by taking $\delta_1$ small enough we can ensure that 
\begin{align*}
\dist_{L^2(B_{10})} (u, \ \mathrm{sym}_{n,0}  ) & \le \delta_2, \\
\dist_{L^2(B_{10})} (u, \ \mathrm{sym}_{n,k+1}) & \ge \eps_0 / 2
\end{align*}
with $\delta_2=\delta_2(\eps_0)$ chosen according to Theorem~\ref{th:NabVal-best-approx}. Then we obtain 
\[
\beta (0,1)^2
\le C(n) \int_{B_1} \left( \theta_u(y,8) - \theta_u(y,1) \right) \dd \mu(y),
\]
where $\mu \coloneqq \H^{n-3} \llcorner \sing u$ and $\beta = \beta_{\mu,n-3,2}$. Similarly, 
\begin{equation}
\label{eq:beta-estimate}
\beta (z,s)^2
\le C(n) s^{-(n-3)} \int_{B_s(z)} \left( \theta_u(y,8s) - \theta_u(y,s) \right) \dd \mu(y) 
\end{equation}
for each ball $B_s(z) \subset B_2$ with $z \in \sing u$. To see this, one simply needs to consider the rescaled map $\overline{u}(x) = u(z+rx)$ and apply scaling-invariance of $\delta$-flatness and $\beta$-numbers. 


Now we verify the hypotheses of Rectifiable Reifenberg Theorem~\ref{th:NabVal-reifenberg}. Fix a ball $B_r(x) \subset B_2$; we only need to check that 
\begin{equation}
\label{eq:jones-estimate}
\int_{B_r(x)} \int_0^r \beta(z,s)^2 \frac{\dd s}{s} \dd \mu(z) \le \delta_3 r^{n-3} 
\end{equation}
with $\delta_3(\eps) > 0$ chosen according to Theorem~\ref{th:NabVal-reifenberg}. 

First, we integrate the estimate \eqref{eq:beta-estimate} over $B_r(x)$ and exchange the order of summation: 
\begin{align*}
\int_{B_r(x)} \beta(z,s)^2 \dd \mu(z) 
& \lesssim s^{-(n-3)} \int_{B_r(x)} \int_{B_s(z)} (\theta_u(y,8s)-\theta_u(y,s)) \dd \mu(y) \dd \mu(z) \\
& \le s^{-(n-3)} \int_{B_{2r}(x)} \int_{B_s(y)} (\theta_u(y,8s)-\theta_u(y,s)) \dd \mu(z) \dd \mu(y) \\
& \lesssim \int_{B_{2r}(x)} (\theta_u(y,8s)-\theta_u(y,s)) \dd \mu(y)
\end{align*}
Note that in the last step we used the weak upper bound \eqref{eq:weak-meas-bound} on the ball $B_s(y)$. 

When the above is integrated with respect to $s$, we obtain a telescopic sum. In order to estimate it, first recall that $u$ is $\delta_1$-flat in each ball $B_{8r}(y)$ such that $y \in \sing u$ $B_r(y) \subset B_2$, in particular 
\[
\theta_u(y,8r) - \theta_u(y,0) \le \delta_1 
\]
on the support of $\mu$. Thus, the substitution $s \mapsto 8s$ together with monotone convergence $\theta_u(y,s) \searrow \theta_u(y,0)$ give us 
\begin{align*}
\int_0^r (\theta_u(y,8s)-\theta_u(y,s)) \frac{\dd s}{s} 
& = \int_{r}^{8r} (\theta_u(y,s) - \theta_u(y,0)) \frac{\dd s}{s} \\
& \le \ln(8) \delta_1.
\end{align*}
Now we are ready to combine the above estimates: 
\begin{align*}
\int_{B_r(x)} \int_0^r \beta_{\mu,2}^2(z,s) \frac{\dd s}{s} \dd \mu(z) 
& \lesssim \int_0^r \int_{B_{2r}(x)} (\theta_u(y,8s)-\theta_u(y,s)) \dd \mu(y) \frac{\dd s}{s} \\
& \le \int_{B_{2r}(x)} \ln(8) \delta_1 \dd \mu(y) \\
& \lesssim \delta_1 r^{n-3},
\end{align*}
where we used \eqref{eq:weak-meas-bound} again in the last line. Assuming $\delta_1 \le \delta_3(\eps) / C(n)$, we have verified the assumption \eqref{eq:jones-estimate} and we infer the upper estimate 
\[
\H^{n-3}(\sing u \cap B_1) = \mu(B_1) \le (1+\eps) \omega_{n-3}. 
\] 

\end{proof}

\subsection{Global case}
\label{sec:stab-global}
The idea of the proof is to cover most of $\sing u$ by good balls, on which $u$ is $\delta$-flat and thus the measure of $\sing u$ is controlled by Lemma \ref{lem:local-stability}. The rest of the singular set is to be covered by bad balls, whose total mass is small. To achieve this, we will need the following simple covering lemma. 

\begin{lemma}
\label{lem:easy-covering}
Let $S \subseteq \R^n$ be a compact set of finite $\cH^{k}$-measure and let $\cB$ be a~family of open balls such that for each point $p \in S$, all small enough balls around $p$ belong to $\cB$. Then, given any $\eps > 0$, $S$ can be covered by the union of two finite families of open balls $\good$, $\bad$, where $\good \subseteq \cB$ consists of pairwise disjoint balls and $\bad = \B_{r_j}(p_j)$ is a small family in the sense that 
%
%
\begin{equation}
\label{eq:cov:smallness}
\sum_j r_j^{k} \le \eps. 
\end{equation}
\end{lemma}

\begin{proof}
One way to construct this covering is by using Vitali's covering theorem for Radon measures (e.g., \cite[Theorem 2.8]{Mat95}). Applying it to the measure $\mu \coloneqq \cH^{k} \llcorner S$, we obtain a~countable family of pairwise disjoint balls $\cA = \left\{ \overline{B_{r_s}(p_s)} \right\}$, covering $\mu$-almost all $S$ and satisfying $B_{2r_s}(p_s) \in \cB$ for each $s$. Since the measure $\mu$ is finite, we can divide $\cA$ into two subfamilies $\good'$, $\bad'$, where $\good'$ is finite and $\bad'$ is small, i.e., $\mu \left( \bigcup \bad' \right) \le \eps$. To obtain the desired properties, we still need to alter these families a little. 

First, we define $\good$ to be the balls of $\good'$ slightly enlarged to open balls, but still pairwise disjoint and still belonging to $\cB$. 

Now, the remaining part $S \setminus \bigcup \good$ is a compact set and 
\[
\mu \left( S \setminus \bigcup \good \right) \le \mu \left( \bigcup \bad' \right) \le \eps.
\]
By definition of Hausdorff measure, this set can be covered by a finite family of open balls $\bad$ satisfying the smallness condition \eqref{eq:cov:smallness}. 
\end{proof}

\begin{proof}[Proof of Theorem~\ref{th:global-stability}]

Fix $\varepsilon > 0$. For the sake of clarity, we focus on showing that the difference $| \H^{n-3}(\sing u_k) - \H^{n-3}(\sing u) |$ is controlled by $\eps$ for $k$ large enough. The estimate for Wasserstein distance follows the same lines; it is briefly discussed at the end of the proof. 

\textsc{Step 1 (boundary regularity).}
Choose $\varepsilon_0 > 0$ according to the boundary regularity theorem, Theorem~\ref{th:r4s}. Fix $\rho > 0$ such that 
\[
\sup_{x \in \partial \Omega} \int_{B_\rho(x)} |\nabla \varphi|^{n-1} \le \varepsilon_0 / 2.
\]
Then $u$ is smooth in a $\lambda \rho$-neighborhood of $\partial \Omega$. By strong convergence of $\varphi_k$ to $\varphi$ in $W^{1,n-1}(\partial \Omega)$, we may assume w.l.o.g. for all $k \in \N$,
\[
\sup_{k} \sup_{x \in \partial \Omega} \int_{B_\rho(x)} |\nabla \varphi_k|^{n-1} \le \varepsilon_0.
\]
As a consequence, we may assume each $u_k$ is also smooth in the same fixed neighborhood of $\partial \Omega$. 

\textsc{Step 2 (covering the low-dimensional part).}
Recall the stratification, Section~\ref{s:tangentmaps},
\[
S_0 \subset \ldots \subset S_{n-4} \subset S_{n-3} = \sing u,
\]
in which the $k$-th stratum $S_k$ has Hausdorff dimension $k$ or smaller. We will consider separately the set $S_{n-4}$ and the top-dimensional part 
\[
\sing_* u \coloneqq S_{n-3} \setminus S_{n-4}.
\]
Since $\sing u$ is compact and $\sing_* u$ is an open subset of $\sing u$ (see Theorem~\ref{th:mis-all}), $S_{n-4}$ is also compact. At the same time, it has a uniform distance from $\partial \Omega$ and $\H^{n-3}(S_{n-4})=0$, so it can be covered by a finite family $\bad_1 = \{ B_{r_i}(p_i) \}$ of open balls satisfying the smallness condition \eqref{eq:cov:smallness} 
\[
\sum_i r_i^{n-3} \le \varepsilon 
\]
and such that $B_{2r_i}(p_i) \subset \Omega$ for each $i$. 

On each such ball Corollary~\ref{co:NabVal-meas-bound} yields $\H^{n-3}(\sing u \cap B_{r_i}(p_i)) \le C r_i^{n-3}$, with $C$ depending only on the dimension $n$. Summing over all balls, we obtain 
\[
\H^{n-3} \left( \sing u \cap \bigcup \bad_1 \right) 
\le C \eps. 
\]
The same estimate holds verbatim for each $u_k$, by the same application of Corollary \ref{co:NabVal-meas-bound}. 

\textsc{Step 3 (covering the top-dimensional part and estimating $\H^{n-3}(\sing u)$).}
Here, we use the covering lemma (Lemma \ref{lem:easy-covering}) for the set $S \coloneqq \sing u \setminus \bigcup \bad_1$. Thanks to Step 1, $\sing u$ has positive distance from the boundary, so it is a compact set of finite $\cH^{n-3}$-measure due to Corollary \ref{co:NabVal-meas-bound}. We choose $\cB$ to be 
\[
\cB = \left \{ B_r(p) \colon \quad p \in \sing_* u, \text{ $u$ is $\delta$-flat in $B_{81r}(p)$} \right \},
\]
where $\delta(\eps) > 0$ is chosen according to Lemma \ref{lem:local-stability}. Since $S_{n-4}$ is already covered by $\bad_1$, we know that $S \subset \sing_* u$ and hence small enough balls around each point in $S$ lie in $\cB$ by Lemma \ref{lem:mis-some-flatness}. 

Having checked the properties required by Lemma \ref{lem:easy-covering}, we can cover $S$ by the union of a~finite disjoint family $\good \subset \cB$ and another finite family $\bad_2$ satisfying \eqref{eq:cov:smallness}. We add the latter to $\bad_1$ to obtain the family of bad balls $\bad \coloneqq \bad_1 \cup \bad_2$, which still satisfies the smallness condition \eqref{eq:cov:smallness}. 

Repeating the reasoning from Step 2, we have again via Corollary~\ref{co:NabVal-meas-bound},
\begin{equation}
\label{eq:comparison-bad-balls}
\H^{n-3} \left( \sing u \cap \bigcup \bad \right) 
\le 2C \eps, 
\quad  
\H^{n-3} \left( \sing u_k \cap \bigcup \bad \right) 
\le 2C \eps 
\quad \text{for all } k.
\end{equation}

By assumption, the map $u$ is $\delta$-flat in $B_{80r_s}(p_s)$ for each ball $B_{r_s}(p_s) \in \good$. By Lemma~\ref{lem:local-stability}, we now obtain 
\[
(1-\eps) \omega_{n-3} r_s^{n-3} 
\le \H^{n-3}(\sing u \cap B_{r_s}(p_s)) 
\le (1+\eps) \omega_{n-3} r_s^{n-3} 
\]
for each $s$. To finish the proof, we need to show that a similar comparison holds for $u_k$ if $k$ is large.

\textsc{Step 4 (estimating $\H^{n-3}(\sing u_k)$).} Since $u_k \to u$ in $W^{1,2}(\Omega)$ and $\sing u$ is covered by the open families $\good,\bad$, Theorem~\ref{th:ALs2s} (singular points converge to singular points) implies that the same holds for $u_k$ if $k$ is large enough (from now on we assume it is). For bad balls, the rough estimate \eqref{eq:comparison-bad-balls} will be enough, so we focus on good balls. 

By Proposition~\ref{prop:stability-of-flatness}, we can assume (by taking $k$ large and $\delta$ small) that for each $B_{r_s}(p_s) \in \good$ there is $p^k_s \in \sing u_k$ such that $|p^k_s-p_s| \le \eps r_s$ and $u_k$ is $\delta'$-flat in the ball $B_{80(1+\eps)r_s}(p^k_s)$. Here, the value of $\delta'$ is chosen to be $\delta(\eps)$ from Lemma~\ref{lem:local-stability}. 

Applying Lemma~\ref{lem:local-stability} to $u_k$ on balls $B_{(1-\eps)r_s}(p^k_s)$ and $B_{(1-\eps)r}(p^k_s)$, we obtain  
\begin{align*}
(1-\eps)^{n-2} \omega_{n-3} r_s^{n-3} 
& \le \H^{n-3}(\sing u_k \cap B_{(1-\eps)r_s}(p^k_s)) \\
& \le \H^{n-3}(\sing u_k \cap B_{r_s}(p_s)) \\
& \le \H^{n-3}(\sing u_k \cap B_{(1+\eps)r_s}(p^k_s)) \\
& \le (1+\eps)^{n-2} \omega_{n-3} r_s^{n-3}, 
\end{align*}
which is only slightly worse that the estimate for $\H^{n-3}(\sing u)$.

\textsc{Step 5 (comparison).}
Recalling that $\good$ is a disjoint family, we can sum the above estimate over all $s$ to obtain 
\[
(1-\eps)^{n-2} A 
\le \H^{n-3}(\sing u_k \cap \bigcup \good) 
\le (1+\eps)^{n-2} A,  
\]
where $A \coloneqq \sum_s \omega_{n-3} r_s^{n-3}$. 
Combining it with the estimate for bad balls \eqref{eq:comparison-bad-balls}, we finally obtain 
\[
(1-\eps)^{n-2} A 
\le \H^{n-3}(\sing u_k)
\le (1+\eps)^{n-2} A + 2C \eps.
\]
Exactly the same estimate is true for $u$. Combining these two yields 
\begin{align*}
\big| \H^{n-3}(\sing u_k) - \H^{n-3}(\sing u) \big| 
& \le \left( (1+\eps)^{n-2} - (1-\eps)^{n-2} \right) A + 2C \eps \\
& \le \left( \frac{(1+\eps)^{n-2}}{(1-\eps)^{n-2}} - 1 \right) \H^{n-3}(\sing u) + 2C \eps. 
\end{align*}
Evidently the right-hand side tends to zero when $\eps \to 0$, which ends the proof of stability of $\cH^{n-3}(\sing u)$. 

\textsc{Step 6 (Wasserstein distance estimate).} With just a little bit more care, the Wasserstein distance estimate follows. Let us decompose the measure $\mu \coloneqq \cH^{n-3} \llcorner \sing u$ into $\mu = \mu_b + \sum_s \mu_s$, where 
\[
\mu_b = \mu \llcorner \left( \bigcup \bad \setminus \bigcup \good \right), \quad \mu_s = \mu \llcorner B_{r_s}(p_s) \quad \text{for each ball } B_{r_s}(p_s) \in \good.
\]
The estimate for $\mu_b$ is simply $d_W(\mu_b, 0) \le \mu \left ( \bigcup \bad \right) \le 2 C \eps$, whereas on each good ball $B_{r_s}(p_s)$ we have the inequalities 
\begin{align*}
\int_{\R^n} h \dd \mu_s - \omega_{n-3} r_s^{n-3} h(p_s) 
& = \int_{B_{r_s}(p_s)} (h - h(p_s)) \dd \mu + (\mu(B_{r_s}(p_s)) - \omega_{n-3} r_s^{n-3}) h(p_s) \\
& \le r_s \mu(B_{r_s}(p_s)) + |\mu(B_{r_s}(p_s)) - \omega_{n-3} r_s^{n-3}| \\
& \le (r_s + 2\eps) \omega_{n-3} r_s^{n-3}. 
\end{align*}
for any function $h \colon \R^n \to \R$ satisfying $|h| \le 1$ and $|\nabla h| \le 1$. Thus $d_W(\mu_k, \omega_{n-3} r_s^{n-3} \delta_{p_s}) \le 3\eps \omega_{n-3} r_s^{n-3}$, if only each radius is smaller than $\eps$. By triangle inequality, $d_W(\mu,\nu) \le 3 \eps A + 2 C \eps$, where $\nu = \sum_s \omega_{n-3} r_s^{n-3} \delta_{p_s}$ is the packing measure associated to $\good$ and once again $A = \nu(\R^n)$. Applying the same reasoning to $u_k$, we conclude as before. 
\end{proof}

\subsection{The case $n=3$}
We close this section with the proof in the special case when the domain is three dimensional. In this case, using the results of previous sections and adapting the arguments of Hardt and Lin from \cite{HL1989} the proof is quite quick.

The counter-example by Strzelecki and the first-named author in \cite{MazowieckaStrzelecki} implies that there is no stability result for $W^{1,p}$ with $p < 2$ perturbations of the boundary. In this sense the following Theorem~\ref{th:stability} is the sharp limit case. 

\begin{theorem}\label{th:stability}
Let $\Omega \subset \R^3$ be a bounded domain with a $C^1$-boundary.  

Assume that $v \in W^{1,2}(\Omega,\S^2)$ is the {unique} minimizing harmonic map with boundary data $v\rvert_{\partial \Omega} = \psi$ and let $\psi\in W^{s,p}(\partial\Omega,\S^2)$, where 
$sp = 2$, $s\in (\frac 12,1]$, $p\in [2,\infty)$.

Then for any $\eps > 0$ there is a $\delta = \delta(\eps,\Omega,\psi) > 0$ such that whenever $u$ is a minimizing harmonic map $u \in W^{1,2}(\Omega,\S^2)$ with trace $\varphi \coloneqq u \Big |_{\partial \Omega}$ close to $\psi$,
\begin{equation}\label{eq:h12smallcond2}
[\psi-\varphi]_{W^{s,p}(\partial \Omega)} \leq \delta 
\end{equation}
then $u$ has the same number of singularities as $v$. Moreover,
\begin{equation}\label{eq:umvsmall}
\|u - v\|_{W^{1,2}(\Omega)} \leq \eps.
\end{equation}
\end{theorem}

\begin{proof}
We first prove the statement~\eqref{eq:umvsmall}. Assume the claim is false for a given unique minimizer $v$ and for some $\eps > 0$. Then we find a sequence of minimizers $u_i$ with traces $\varphi_i$ which satisfy 
\[
[\varphi_i-\psi]_{W^{s,p}(\partial \Omega)}^p \leq \frac{1}{i} 
\]
but
\begin{equation}\label{eq:minimizersfar}
\|u_i - v\|_{W^{1,2}(\Omega)} > \eps.
\end{equation}
Now we obtain a contradiction since by strong convergence of minimizers, \Cref{th:bsc}, the sequence $u_i$ converges to the unique minimizer $v$ in $W^{1,2}$. In particular \eqref{eq:minimizersfar} cannot be true for all $i \in \N$.

Regarding the same number of singularities we recall that by \Cref{co:fewsg} for boundary data in $W^{s,p}(\partial \Omega)$ with $sp=2$ and $s>\frac12$ there can only be finitely many singularities. Let us assume that the theorem is false for a unique minimizer $v$ which has exactly $N < \infty$ singularities $x_1,\ldots,x_N$.

Then we find a sequence $u_i \in W^{1,2}(\Omega,\S^2)$ of minimizing harmonic maps with traces $\varphi_i \coloneqq u_i \Big |_{\partial \Omega}\in W^{s,p}(\partial \Omega, \S^2)$ with 
\begin{equation}\label{eq:stability3Dboundarysmallness}
 [\psi - \varphi_i]_{W^{s,p}(\partial \Omega)}^2 \le \frac{1}{i}
\end{equation}
and such that either all $u_i$ have $M < N$ singularities $(y_{i,k})_{k=1}^M$ or all $u_i$ have at least $N+1$ singularities $(y_{i,k})_{k=1}^{N+1}$.

From the strong convergence of minimizing harmonic maps, 
\Cref{th:bsc}, and the uniqueness of $v$,  we may assume, up to a subsequence, that
\[
u_i \to v \quad \mbox{in $W^{1,2}(\Omega,\S^2)$}.
\]

If each $u_i$ had $M < N$ singularities we find a contradiction to Theorem~\ref{th:ALs2sii}, since all the singularities of $v$ have to come as limits of singularities of $u_i$.

So we may assume that each $u_i$ has at least $M>N$ singularities. Since, by Theorem~\ref{th:ALs2s}, singularities of $u_i$ which do not approach the boundary $\partial \Omega$ converge to singularities of $v$, and two different singularities of $u_i$ cannot converge to the same singularity of $v$ by uniform distance (proportional to the distance of the boundary) of singularities, Theorem~\ref{th:udsg}, the only way this is possible is if singularities of $u_i$ approach the boundary $\partial \Omega$. 

However, we can rule this out simply by the assumption~\eqref{eq:h12smallcond2}. Since $\psi\in W^{s,p}(\partial \Omega,\S^2)$ we may find a $\varrho>0$ such that 
\[
 \sup_{x_0\in\partial\Omega} [\psi]_{W^{s,p}(\partial \Omega\cap B_\rho(x_0))}^p < \frac{\eps}{2^p},
\]
where $\eps$ is as in the uniform boundary regularity Theorem~\ref{th:r4s}. By \eqref{eq:stability3Dboundarysmallness} we have for $\varphi_i$ 
\[
\sup_{x_0\in\partial\Omega} [\vp_i]_{W^{s,p}(\partial \Omega\cap B_\rho(x_0))}^p < 2^{p-1}\brac{\frac\eps
2 + \frac1i}.
\]
Thus for sufficiently large $i$ the singularities of $u_i$ and $v$ cannot approach the $\lambda\rho$-neighborhood of the boundary, where $\lambda$ is a uniform constant. 
\end{proof}

\section{Almgren and Lieb's linear law: size of the singular set}\label{sec:AL}

Here we obtain a higher-dimensional counterpart and at the same time a sharpened version of Almgren--Lieb's linear estimate on the number of singularities. 
Let us stress that the fundamental result that makes such estimates possible is Naber and Valtorta's interior measure bound (Corollary~\ref{co:NabVal-meas-bound}). 

\begin{theorem}
\label{th:almgren-lieb-high}
Let us assume that the fundamental group of $\mathcal N$ is finite. Let $\Omega\subset \R^n$ be a bounded domain with a $C^1$ boundary and let $u\in W^{1,2}(\Omega,\n)$ be a minimizing map with $u\rvert_{\partial \Omega} = \varphi$. Assume that $\varphi \in W^{s,p}(\partial \Omega,\n)$ for $s\in(\frac12,1]$ and $p>1$ with $sp=n-1$. Then 
 \begin{equation}\label{eq:mainestimate}
  \H^{n-3}(\sing u) \le C [\vp]_{W^{s,p}(\partial \Omega)}^p.
 \end{equation}
\end{theorem}

\begin{remark}
 In particular, we recall that we denoted $[\vp]_{W^{1,n-1}(\partial \Omega)}^{n-1} = \int_{\partial \Omega} |\nabla \vp|^{n-1} \dhn$, thus if $\vp\in W^{1,n-1}(\partial \Omega, \n)$, then we have
 \[
  \H^{n-3}(\sing u) \le C \int_{\partial \Omega} |\nabla \vp|^{n-1} \dhn.
 \]
\end{remark}

\begin{remark}
 As shown in \Cref{ss:examples} the result is optimal in the case $n=3$ in the sense that it fails for $sp<2$.
\end{remark}

The study of singularities near the boundary involves the following covering lemma, which we here cite from \cite[Theorem 2.8, 2.9]{AlmgrenLieb1988}. 

\begin{theorem}[Covering lemma]
\label{th:covering}
Let $\mathcal B$ be a family of closed balls in $\R^n$, $\mu$ be a Borel measure over $\R^n$, and let $\tau,\omega \in (0,1)$. Moreover, assume that the following two hypotheses hold: 
\begin{enumerate}
 \item For any two different $B_r(p), B_s(q) \in \mathcal{B}$ we have 
 \[
  |p-q| \ge \omega \min(r,s).
 \]
 \item Suppose that $B_r(p)\in \cB$ and $q\in\R^n$ is an arbitrary point, then 
 \[
  \mu \left( B_r(p) \setminus B_{\tau r}(q) \right) \ge 1.
 \]
\end{enumerate}
Then
\[
\# \text{balls in }\mathcal{B} \le C \mu(\R^n), 
\]
for a constant $C(\omega,\tau,n) > 0$. 
\end{theorem}

\begin{proof}[Proof of \Cref{th:almgren-lieb-high}]
Choose $\sigma > 0$ (depending on the geometry of $\partial \Omega$) according to \Cref{th:int-regularity-in-terms-of-bdry} and \Cref{th:hot-spots}. We first estimate the measure of the set 
\[
 A_1 \coloneqq \{ z \in \sing u\colon  r(z) \le \sigma \}, 
 \quad \text{where } r(z) = \tfrac 12 \dist(z,\partial \Omega),  
\]
which is covered by balls $B_{r(z)}(z)$. Then choose a Vitali subcovering such that the balls $B_{r_j}(z_j)$ cover $A_1$ and the balls $B_{r_j/5}(z_j)$ are disjoint; let $\cB$ be the family of balls $B_{r_j/\lambda}(z_j)$ with $\lambda$ as in \Cref{th:hot-spots}. The first condition from \Cref{th:covering} with $\omega = \lambda/5$ follows: for any two distinct balls in our collection we have 
\[
 |z_i-z_j| \ge \tfrac 15 (r_i+r_j) \ge \tfrac{\lambda}{5} \max(r_i/\lambda,r_j/\lambda).
\]

Now for an open $U\subset\R^n$ we define the Borel measure $\mu$ as follows:

\underline{\textsc{Case 1:} $s\in(\frac12,1)$, $p>1$, $sp=n-1$}

\begin{equation}\label{eq:fractionalmuchoice}
 \mu(U) \coloneqq \frac1\eps \int_{U\cap\partial \Omega} \int_{\partial \Omega} \frac{|\vp(x) - \vp(y)|^p}{|x-y|^{2n-2}} \, dx \, dy.
\end{equation}
Obviously, $\mu$ is a measure and we have
\[
 \mu(U) \ge \frac1\eps \int_{U\cap \partial \Omega} \int_{U\cap \partial \Omega} \frac{|\vp(x) - \vp(y)|^p}{|x-y|^{2n-2}} \, dx \, dy.
\]

\underline{\textsc{Case 2:} $s=1$, $p=n-1$}

\[
 \mu = \frac 1 \eps |\nabla \varphi|^{n-1} \, \cH^{n-1} \llcorner \partial \Omega, 
 \quad \text{i.e., } 
 \mu(U) = \frac 1 \eps \int_{\partial \Omega \cap U} |\nabla \varphi|^{n-1} \dd \cH^{n-1}.
\]
In both cases $\eps > 0$ is the constant from Theorem~\ref{th:hot-spots}. If we set $\tau = \lambda^2$, then the second condition of Theorem~\ref{th:covering} with $k = n-3$ follows from Theorem~\ref{th:hot-spots} and we infer that 
\[
\# \cB \le C [\vp]_{W^{s,p}(\partial \Omega)}^p. 
\]
On each ball $B_{r_j}(z_j)$ Corollary~\ref{co:NabVal-meas-bound} implies $\H^{n-3}(\sing u \cap B_{r_j}(z_j)) \le C r_j^{n-3} \le C(\Omega)$. Summing over all balls, we obtain 
\[
\H^{n-3}(A_1) \le C [\vp]_{W^{s,p}(\partial \Omega)}^p. 
\]

Next we estimate the set 
\[
 A_2 \coloneqq \{ z \in \sing u\colon  r(z) \ge \sigma \}.
\]
For each ball $B_\sigma(y)$ with $\dist(y,\partial \Omega) \ge 2 \sigma$ we have a bound $\H^{n-3}(\sing u \cap B_\sigma(y)) \le C \sigma^{n-3}$ by Corollary~\ref{co:NabVal-meas-bound}. The set $A_2$ can be covered by finitely many such balls (the number of balls depending only on $\sigma$ and the geometry of $\Omega$), which gives us an estimate 
\[
 \H^{n-3}(A_2) \le C_0. 
\]
Taking $C_0$ as above and $\eps$ as in Theorem~\ref{th:int-regularity-in-terms-of-bdry}, we have two possibilities. Either the smallness condition $[\vp]_{W^{s,p}(\partial \Omega)}^p \le \eps$ is satisfied and $\cH^{n-3}(A_2) = 0$ follows, or 
\[
\cH^{n-3}(A_2) \le C_0 \le \frac{C_0}{\eps} [\vp]_{W^{s,p}(\partial \Omega)}^p.
\]
In both cases, combining the estimates for $A_1$ and $A_2$ ends the proof.
\end{proof}

\section{Final remarks}\label{s:finalremarks}
\subsection{Examples}\label{ss:examples}
We present two examples that show that the linear law we obtain in dimension $n=3$ is sharp. Our examples will be based on the following method of installing singular points from \cite{AlmgrenLieb1988}, see also \cite[Section 2]{MazowieckaStrzelecki}.

\begin{definition}\label{def:singinstal}
	Let $\psi\colon \S^2\rightarrow\S^2$ be smooth near $y\in\S^2$ and let $r>0$ be a fixed number. We denote by $\langle\psi\rangle_{y,r}\colon\S^2\to\S^2$ a smooth boundary map which arises from $\psi$ by a small deformation in a neighborhood of $y$ so that the following conditions are satisfied:
	\begin{enumerate}
[label=(\alph*)]
	  \item $\langle\psi\rangle_{y,r}(x)=\psi(x)$ whenever $|x-y|\ge r$;
	  \item $\langle\psi\rangle_{y,r}(x)\equiv\psi(y)$ if $|x-y|=r/2$;
	  \item \label{it:installannuli} The restriction of $\langle\psi\rangle_{y,r}$ to the annular region $\frac r 2<|x-y|<r$ satisfies the Lipschitz condition with a Lipschitz constant $L_\psi$ which depends only on $\psi$ and not on $r$;
	  \item \label{it:installenergy}$\langle\psi\rangle_{y,r}$ is a diffeomorphism of the spherical cap $\{|x-y|<r/2\} \, \cap\, \S^2$ onto the punctured sphere $\S^2\setminus\{\psi(y)\}$ such that the boundary Dirichlet integral energy of $\langle\psi\rangle_{y,r}$ on this cap equals $8\pi+\text{o}(1)$ as $r\rightarrow 0$.
	 \end{enumerate}
\end{definition}

The existence of maps described in \Cref{def:singinstal} is well known, the proof follows for example a modification of \cite[{Appendix A.2}]{Almgren-Browder-Lieb}.

\begin{theorem}[{{\cite[Theorem 4.3]{AlmgrenLieb1988}}}]\label{th:ALinstallsing}
 Suppose $u\colon B_1^3\rightarrow\S^2$ is a minimizer which is unique for its boundary mapping $\psi\colon \S^2\rightarrow\S^2$ and which has $k$ interior singularities at $x_1,\ldots,x_k\in B_1^3$. Moreover, assume that $\int_{\S^2}|\nabla \psi|^2\dh<\infty$ and that $\psi$ is smooth near $y_0\in\S^2$. Let $\psi_j\colon\S^2\rightarrow\S^2$ be any sequence of continuous boundary mappings such that $\psi_j=\langle\psi\rangle_{y_0,2/j}$ for all $j$ sufficiently large.

Finally, let $u_j$ be any minimizer in $B_1^3$ with boundary mapping $\psi_j$. Then, for all sufficiently large $j$, the mapping $u_j$ will have at least $k+1$ interior singular points $y_{0,j}$ and $x_{1,j},\ldots,x_{k,j}$ such that $y_{0,j}\rightarrow y_0$ and $x_{\ell,j}\rightarrow x_\ell$ (for each $\ell = 1, \ldots, k$) as $j\rightarrow\infty$.
\end{theorem}

In the first example, \Cref{la:sharpll1s}, we show that there cannot be any linear law (or a~similar result, e.g., a~power law) for boundary energies $[\varphi]_{W^{s,p}(\partial \Omega)}$ if $sp < 2$. See a similar construction in \cite[Remark 3.9]{MazowieckaStrzelecki}.

\begin{lemma}\label{la:sharpll1s}
Assume that the following holds for $0 < s \le 1$, $p \geq 1$: for every $\eps > 0$ there exists a $\delta$ such that if $u \in W^{1,2}(B_1^3,\S^2)$ is a minimizer with trace $\varphi$ and
\begin{equation}\label{eq:vpsd}
[\varphi]_{W^{s,p}(\partial B_1)} \leq \delta
\end{equation} 
then
\begin{equation}\label{eq:sigmaeps}
\# \left \{\mbox{singularities of u}\right \} \leq \eps.
\footnote{In fact, for small $\eps$ we simply conclude that $u$ has no singularities at all.}
\end{equation}
Then $sp \geq 2$.
\end{lemma}

\begin{proof}
Let $\vp_0\colon \S^2 \to \S^2$ be a constant map, which admits a constant map as the unique smooth minimizer $B_1^3 \to\S^2$. Let us fix a point $N \in \S^2$ and modify $\vp_0$ in order to insert a singularity -- consider the map $\vp_\delta = \langle\vp_0\rangle_{N,\delta}$. Then, for sufficiently small $\delta$, we have by \Cref{def:singinstal} 
\begin{equation*}
 \int_{\S^2} |\nabla \vp_\delta|^2 \dh = \int_{D_\delta(N)} |\nabla \vp_\delta|^2 \dh \le 10\pi.  
\end{equation*}
Assume to the contrary that $q := sp < 2$. Then, by H\"{o}lder's inequality we get
\begin{equation}\label{eq:W1p2normalconstant}
 \int_{\S^2} |\nabla \vp_\delta|^{q} \dh \le \brac{\int_{D_\delta(N)} |\nabla \vp_\delta|^2 \dh}^{\frac{q}{2}}|D_\delta(N)|^{1-\frac{q}{2}} \aleq \delta^{2-q}.
\end{equation}
Combining this with $\vp_\delta\in\S^2$, we obtain from Gagliardo--Nirenberg's inequality \cite{BrezisMironescu-gagliardo} (for $\theta = 1-s$) 
\begin{equation}
 [\vp_\delta]_{W^{s,p}(\S^2)} \le \|\vp_\delta\|^\theta_{L^{\infty}}\|\nabla\vp_\delta\|^{1-\theta}_{L^{q}} \aleq \delta^{\frac{2-q}{p}}.
\end{equation}
Since $q = sp < 2$, the seminorm $[\vp_\delta]_{W^{s,p}(\S^2)}$ can be made arbitrarily small and so \eqref{eq:vpsd} is satisfied. On the other hand, by \Cref{th:ALinstallsing} any minimizer with boundary data $\vp_\delta$ must have at least one singularity, thus \eqref{eq:sigmaeps} fails. 

\end{proof}


If we modify this example further, it is also possible to construct a minimizing harmonic maps with infinitely many singularities (with finite $W^{1,2-\eps}$-energy at the boundary).

\begin{theorem}\label{ex:sharpnessoflinearlaw}
Let $\eps>0$ be any positive number. There is a boundary map $\vp\in W^{1,2-\eps}(\partial B_1^3,\S^2)$, such that the following holds:

there is a minimizer $u\colon B_1^3\rightarrow\S^2$ with $u\big\rvert_{\partial B_1^3} = \vp$ and $u$ has countably infinitely many singularities. 
\end{theorem}

In order to prove \Cref{ex:sharpnessoflinearlaw} we will modify Almgren and Lieb's ``boiling water example'' \cite[Theorem 4.4]{AlmgrenLieb1988}.

We will need the following Lemma, which shows how we can use a small modification of the boundary data in order to guarantee that the corresponding minimizer is unique, see \cite[Theorem 3.2]{AlmgrenLieb1988} and also \cite[Lemma 3.8]{MazowieckaStrzelecki}. 
\begin{lemma}\label{le:AL32}
Given a map $\vp\in C^\infty(\S^2,\S^2)$ and a small spherical cap $D_\delta(y_0)$, one may find a new map $\tilde \vp$ which differs from $\vp$ only on $D_\delta(y_0)$, $\|\vp - \tilde \vp\|_{W^{1,2}(\partial B^3)}< 10 \delta$ and -- most importantly -- there exists exactly one minimizer $\tilde u\colon B_1^3 \to \S^2$ with $\tilde u \big\rvert_{\partial B_1^3} = \tilde \vp$.
\end{lemma}

%
%
%
%

\begin{proof}[Proof of Theorem \ref{ex:sharpnessoflinearlaw}]
Fix $0<\eps<1$. We divide the sphere $\S^2$ into infinitely many disjoint regions near which the singularities will appear (and infinitely many disjoint regions near which we make small corrections in order to ensure that the boundary data will admit only one minimizer).

 We choose two sequences of points on the sphere $\{y_j\}_{j=1}^\infty$ and $\{\tilde{y}_j\}_{j=1}^\infty$,  $y_j,\, \tilde{y}_j\in \S^2$. We also choose a sequence of radii $\{\varrho_j\}_{j=1}^\infty$ with $\varrho_j<2^{-j}$, such that all of the discs from $\bigcup_{j=1}^\infty D_{\varrho_j}(y_j)$ and $\bigcup_{j=1}^\infty D_{\varrho_j}(\tilde y_j)$ are pairwise disjoint, where $D_{\varrho_k}(y_k)\coloneqq B_{\varrho_j}(y_j)\cap \S^2$ stands for the disc on the sphere $\S^2$.
 
 The regions near $D_{\varrho_j}(y_j)$ will be the regions near which singularities will appear and the discs $D_{\varrho_j}(\tilde y_j)$ will be used to correct the boundary map, using \Cref{le:AL32}, in such a way that the boundary map will admit a unique minimizer.

We begin with a constant map $\vp_0\colon\S^2 \rightarrow \S^2$ and we will modify it until we obtain the desired boundary map. Since $\vp_0$ is smooth and admits only one minimizer, i.e., the constant map, we can install a singularity using \Cref{th:ALinstallsing}, by modifying the map $\vp_0$ in the following way 
 \begin{equation*}
  \overline{\vp}_1\coloneqq 
  \left\{ \begin{array}{ll}
    \langle\vp_0\rangle_{y_1,\varrho_1} & \textrm{ for } x\in D_{\varrho_1}(y_1),\\
    \vp_0(x) & \textrm{ for } x\in\S^2 \setminus D_{\varrho_1}(q_1).
  \end{array} \right.
 \end{equation*}
for $\varrho_1<2^{-1}$ small enough so that \Cref{th:ALinstallsing} would guarantee that any minimizer $\overline{u}_1$ with $\overline{u}_1\big\rvert_{\partial B_1^3} = \overline{\vp}_1$ has an interior singularity $\overline{x}_1\in B_{\varrho_1}(y_1)\cap B_1^3$. 

For any $\eps\in(0,1)$ by H\"{o}lder's inequality and conditions \ref{it:installannuli} and \ref{it:installenergy} in the Definition \ref{def:singinstal} we have
\begin{equation}\label{eq:bigH1normofbarvp1}
\begin{split}
\int_{D_{\varrho_1}(y_1)}|\nabla \overline{\vp}_1|^{2-\eps}\dh 
&= \int_{D_{\frac{\varrho_1}{2}}(y_1)}|\nabla \overline{\vp}_1|^{2-\eps}\dh  + \int_{D_{\varrho_1}(y_1) \setminus D_{\frac{\varrho_1}{2}}(y_1)}|\nabla \overline{\vp}_1|^{2-\eps}\dh\\
&\le (8\pi + o(1))^\frac{2-\eps}{2}\pi^{\frac{\eps}{2}}\brac{\frac{\varrho_1}{2}}^{\eps} + L_{\vp_0}^{2-\eps}\pi\frac34 \varrho_1^2,  
\end{split}
\end{equation}
where $o(1)\to 0\text{ as } \rho_1 \to 0$. 

In order to install the next singularity, we need to modify the first map $\overline{\varphi}_1$ in such a way that the new boundary map will admit only one minimizer. For this we use \Cref{le:AL32} and modify $\overline{\vp}_1$ in a small disc $D_{\varrho_1}(\tilde y_1)$, away from the disc $D_{\varrho_1}(y_1)$ in order to obtain $\vp_1\in C^\infty(\S^2,\S^2)$ with the properties:
\begin{enumerate}[label={(P1.\arabic*)}]
 \item $\vp_1 \equiv \overline{\vp}_1$ outside $D_{\varrho_1(\tilde y_1)}$;\\
 \item There exists exactly one minimizer $u_1$ with $u_1\big\rvert_{\partial B_1^3} = \vp_1$;
 \item $u_1$ has at least one singularity $x_1\in B_{\varrho_1(y_1)}\cap B_1^3$; 
 \item \label{it:normdifferencevp1} $\|\overline{\vp}_1 - \vp_1\|_{W^{1,2}(\S^2)}< 10 \varrho_1$.
\end{enumerate}
For $\vp_1$ we have
\begin{equation}\label{eq:111111111}
 \begin{split}
  \int_{\S^2} |\nabla \vp_1|^{2-\eps} \dh 
  &= \int_{D_{\varrho_1}(\tilde y_1)} |\nabla \vp_1|^{2-\eps} \dh + \int_{\S^2 \setminus D_{\varrho_1}(\tilde y_1)} |\nabla \overline{\vp}_1|^{2-\eps} \dh\\
  &\le \brac{\int_{D_{\varrho_1}(\tilde y_1)} |\nabla \vp_1|^{2} \dh}^{\frac{2-\eps}{2}}\pi^{\frac{\eps}{2}}\varrho_1^{\eps} + (8\pi + o(1))^\frac{2-\eps}{2}\pi^{\frac{\eps}{2}}\brac{\frac{\varrho_1}{2}}^{\eps} + L_{\vp_0}^{2-\eps}\pi\frac34 \varrho_1^2.
 \end{split}
\end{equation}
By \ref{it:normdifferencevp1} and since $\overline{\vp}_1\equiv const$ on $D_{\varrho_1}(\tilde y_1)$, we have
 \begin{equation}\label{eq:222222222}
  \int_{D_{\varrho_1}(\tilde y_1)} |\nabla \vp_1|^{2} \dh \le 2\int_{D_{\varrho_1}(\tilde y_1)} |\nabla (\vp_1-\overline{\vp}_1)|^{2} \dh \le 20\varrho_1.
 \end{equation}
Thus, combining \eqref{eq:111111111} with \eqref{eq:222222222} we get
\begin{equation}\label{eq:33333333}
 \int_{\S^2} |\nabla \vp_1|^{2-\eps} \dh 
  \le
  \brac{20}^{\frac{2-\eps}{2}}\pi^{\frac{\eps}{2}}\varrho_1^{\frac{2+\eps}{2}} + (8\pi + o(1))^\frac{2-\eps}{2}\pi^{\frac{\eps}{2}}\brac{\frac{\varrho_1}{2}}^{\eps} + L_{\vp_0}^{2-\eps}\pi\frac34 \varrho_1^2.
\end{equation}
Now we proceed by induction and repeat this procedure in order to install another singularity near the point $y_j\in \S^2$
. Let $j\in \{1,2,\ldots\}$, assume we have already defined the boundary map $\varphi_{j}$, which satisfies:
\begin{enumerate}[label={(Pj.\arabic*)}]
 \item $\varphi_j$ admits only one minimizer $u_j\colon B_1^3 \to \S^2$;
 \item $u_j$ has at least $j$ singular points: $x_{1,j},\ldots, x_{j,j}$, such that $x_{k,j} \in B_{\varrho_k}(y_k)\cap B_1^3$ for each $k\in \{1,\ldots,j\}$; 
 \item outside of the union of the discs $\bigcup_{k=1}^j D_{\varrho_k}(y_k)\cup D_{\varrho_k}(\tilde{y}_k)$ we have $\vp_j \equiv \vp_0$;
 \item \label{it:estimateofvpj}$\vp_j$ satisfies the estimate
 \begin{equation*}
  \int_{\S^2} |\nabla \vp_j|^{2-\eps} \dh \le \sum_{k=1}^j C(\eps)\brac{\varrho_j^{\frac{2+\eps}{2}} + \varrho_j^\eps + \varrho_j^{2}} \aleq \sum_{k=1}^j \varrho_j^{\eps}.
 \end{equation*}
\end{enumerate}
Then, we define the map $\overline{\vp}_{j+1}\colon \S^2 \to \S^2 $ by  
\begin{equation*}
  \overline{\vp}_{j+1}\coloneqq 
  \left\{ \begin{array}{ll}
    \langle\vp_{j}\rangle_{y_{j},\varrho_j} & \textrm{ for } x\in D_{\varrho_{j+1}}(y_{j+1}),\\
    \vp_{j} & \textrm{ for } x\in\S^2 \setminus D_{\varrho_{j+1}}(y_{j+1}).
  \end{array} \right.
 \end{equation*}
We define also the correction $\vp_{j+1}$, to make sure that there is exactly one minimizer corresponding to the boundary map, by applying \Cref{le:AL32} to $\overline{\vp}_{j+1}$ on a small disk $D_{\varrho_{j+1}}(\tilde{y}_{j+1})$. Obtaining a boundary map for which 
\begin{equation}\label{eq:smalldifferencevpjvpj+1}
 \|\vp_{j+1} - \overline{\vp}_{j+1}\|_{W^{1,2}(\S^2)}^2 = \|\vp_{j+1} - \vp_j\|_{W^{1,2}(D_{\varrho_{j+1}}(\tilde{y}_{j+1}))}^2\le 100\varrho_{j+1}^2,
\end{equation}
for a $\varrho_{j+1}<2^{-(j+1)}$ small enough so that \Cref{th:ALinstallsing} would guarantee that any minimizer $u_{j+1}$ corresponding to $\vp_{j+1}$ has an interior singularity $x_{j+1,j+1}\in B_{\varrho_{j+1}}(y_{j+1})\cap B_1^3$ and at least $j$ other singular points $x_{1,j+1},\ldots,x_{j,j+1}$, such that each $x_{k,j+1}\in B_{\varrho_k}(y_k)\cap B_1^3
$ for $k\in\{1,\ldots,j\}$. 

We have exactly as in \eqref{eq:bigH1normofbarvp1}
\begin{equation}\label{eq:againsmallnormbelowW12}
\begin{split}
\int_{D_{\varrho_{j+1}}(y_{j+1})}|\nabla \overline{\vp}_{j+1}|^{2-\eps}\dh 
&\le (8\pi + o(1))^\frac{2-\eps}{2}\pi^{\frac{\eps}{2}}\brac{\frac{\varrho_{j+1}}{2}}^{\eps} + L_{\vp_0}^{2-\eps}\pi\frac34 \varrho_{j+1}^2,  
\end{split}
\end{equation}
where $o(1)\to 0\text{ as } \rho_1 \to 0$. The Lipschitz constant appearing in \eqref{eq:againsmallnormbelowW12} is again $L_{\vp_0}$ as the map $\overline{\vp}_{j+1}=\vp_j = \vp_0$ in $D_{\varrho_j}(y_j)$.

We also have from \eqref{eq:smalldifferencevpjvpj+1} and \ref{it:estimateofvpj}
\begin{equation}\label{eq:vpj+1ontildeyj+1}
 \begin{split}
  \int_{D_{\varrho_{j+1}}(\tilde y_{j+1})}|\nabla \vp_{j+1}|^{2-\eps} \dh 
  &\le 2 \brac{\int_{D_{\varrho_{j+1}}(\tilde y_{j+1})}|\nabla (\vp_{j+1} -\vp_j) |^{2-\eps}\dh + \int_{D_{\varrho_{j+1}}(\tilde y_{j+1})}|\nabla \vp_j|^{2-\eps} \dh}\\ 
  &\aleq \brac{\int_{D_{\varrho_{j+1}}(\tilde y_{j+1})}|\nabla (\vp_{j+1}-\vp_j)|^{2} \dh}^{\frac{2-\eps}{2}}\varrho_{j+1}^\eps + \sum_{k=1}^j \varrho_k^\eps\\
  &\le \varrho_{j+1}^{\eps} + \sum_{k=1}^j \varrho_k^\eps = \sum_{k=1}^{j+1} \varrho_{k}^\eps.
 \end{split}
\end{equation}
Thus, by \eqref{eq:againsmallnormbelowW12}, \eqref{eq:vpj+1ontildeyj+1}, and \eqref{it:estimateofvpj}
\begin{equation}
 \begin{split}
  &\int_{\S^2} |\nabla \vp_{j+1}|^{2-\eps} \dh\\ 
  &\le \int_{D_{\varrho_{j+1}}(\tilde{y}_{j+1})} |\nabla \vp_{j+1}|^{2-\eps} \dh + \int_{D_{\varrho_{j+1}}(y_{j+1})} |\nabla \overline{\vp}_{j+1}|^{2-\eps} + \int_{\S^2} |\nabla \vp_j|^{2-\eps} \dh\\
  &\aleq \sum_{k=1}^{j+1} \varrho_{k}^\eps + (8\pi + o(1))^\frac{2-\eps}{2}\pi^{\frac{\eps}{2}}\brac{\frac{\varrho_{j+1}}{2}}^{\eps} + L_{\vp_0}^{2-\eps}\pi\frac34 \varrho_{j+1}^2 + \sum_{k=1}^j \varrho_{k}^\eps \aleq \sum_{k=1}^{j+1} \varrho_k^\eps.
 \end{split}
\end{equation}
Now we pass with $j\rightarrow\infty$ in order to obtain $\varphi\in W^{1,2-\eps}(\S^2,\S^2)$ with
\[
 \int_{\S^2} |\nabla \vp|^{2-\eps} \aleq \sum_{k=1}^\infty \varrho_k^\eps \le \sum_{k=1}^\infty 2^{-k\eps}<\infty. 
\]
Moreover, $\vp$ admits a minimizer $u$ (any limit point of $u_j$) that has infinitely many singular points $x_k\in B_{\varrho_k}(y_k)\cap B_1^3$, these singularities accumulate at $q_\infty\in \S^2$. 

We note also that $\vp\notin W^{1,2}(\S^2,\S^2)$.
\begin{center}
\begin{tikzpicture}
\fill[shading=ball, ball color=white] (0, 0) circle (4);

\begin{scope}[shift={(0.5,0.1)}]
\clip (-1,1) circle (2);
\fill[shading=ball, ball color=red!40, rotate=-45] (-3.2,0) circle (1);
\end{scope}

\begin{scope}[even odd rule]
\clip (-3,0.6) circle (0.5) (-3.2,0.6) ellipse (0.5cm and 0.7cm);
\fill[shading=ball, ball color=red!40] (-3,0.6) circle (0.5);
\end{scope}

\begin{scope}[even odd rule]
\clip (-3,0.6) circle (0.5);
\clip (-3.2,0.6) ellipse (0.5cm and 0.7cm);
\fill [color=red!40, opacity=0.4] (-4,-1) rectangle (0,4);
\end{scope}

\draw[dashed, opacity =0.4] (-0.55,0.8) circle (0.25);
\fill [color=red!40, opacity=0.4] (-0.55,0.8) circle (0.25);

\begin{scope}
\clip[overlay]  (0,0) circle (4);

\fill[shading=ball, ball color=red!40] (0, -4) circle (4); 

\fill[shading=ball, ball color=red!40] (45:4) circle (2);


\end{scope}

\draw[dashed, opacity =0.4] (-1.2,-0.15) circle (0.125);
\fill [color=red!40, opacity=0.4] (-1.2,-0.15) circle (0.125);

\draw[dashed, opacity =0.4] (0,0.1) circle (0.061);
\fill [color=red!40, opacity=0.4] (0,0.1) circle (0.061);

\draw[dashed, opacity =0.4] (-0.25,-0.4) circle (0.03);
\fill [color=red!40, opacity=0.4] (-0.25,-0.4) circle (0.03);

\fill[shading=ball, ball color=gray!20, opacity=0.2] (0, 0) circle (4);

\draw[thick] (0:4) arc (0:-180:4 and 0.7);
\draw[thick,dashed,opacity=0.9] (180:4) arc (180:0:4 and 0.7);

\draw[thick, rotate=90] (0:4) arc (0:-100:4 and 0.7);
\begin{scope}[yscale=1,xscale=-1]
\draw[thick,dashed, rotate=90, opacity=0.9] (0:4) arc (0:-80:4 and 0.7);
\end{scope}

\draw[thick] (90:4) arc (90:180:4);
\begin{scope}
\clip (-4,1) rectangle (0.64,2.5);
\draw[thick] (30:4) arc (0:-180:3.46 and 0.6);
\end{scope}

\begin{scope}[yscale=1,xscale=-1]
\clip (-2,-1) rectangle (2,1.45);
\draw[thick, rotate=90] (0:4) arc (0:-99:4 and 1.8);
\end{scope}

\begin{scope}
\clip (-1.8,0) rectangle (0.69,2.5);
\draw[thick] (15:4) arc (0:-180:3.87 and 0.65);
\end{scope}

\begin{scope}[yscale=1,xscale=-1, shift={(-1.22,0)}]
\clip (-2,-0.7) rectangle (2,0.4);
\draw[thick, rotate=90, ] (0:4) arc (0:-110:4 and 1.8);
\end{scope}

\begin{scope}
\clip (-0.6,-1) rectangle (0.69,2.5);
\draw[thick] (7.5:4) arc (0:-180:3.96 and 0.65);
\end{scope}

\begin{scope}[yscale=1,xscale=-1, shift={(-1.82,0)}]
\clip (-2,-0.7) rectangle (2,-0.15);
\draw[thick, rotate=90, ] (0:4) arc (0:-110:4 and 1.8);
\end{scope}

\filldraw[black] (0,-4) circle (1pt) node[below] {$q_0$};
\filldraw[black] (45:4) circle (1pt) node[right] {$q_1$};
\filldraw[black,opacity=0.4] (-1.7,2.2) circle (1pt) node[right] {$q_2$};
\filldraw[black] (-3.1,0.6) circle (1pt) node[right] {$q_3$};
\filldraw[black] (-0.55,0.8) circle (0.8pt) node[right] {$q_4$};
\filldraw[black] (-1.2,-0.15) circle (0.7pt) node[right] {$q_5$};
\filldraw[black] (0,0.1) circle (0.5pt) node[right] {$q_6$};
\filldraw[black] (-0.25,-0.4) circle (0.25pt) node[below] {$q_7$};
\filldraw[black] (0.69,-0.69) circle (1pt) node[below] {$q_\infty$};

\node[scale=0.75] at (0,-5) {The singularities appear somewhere in the red bubbles};
\end{tikzpicture}

\end{center}

\end{proof}

\subsection{Remarks about the optimal boundary norm in higher dimensions}\label{ss:optimalboundarynorm}
We do not know if our results are optimal for $n>3$. However, we can construct the following.

\begin{example}\label{example:lowerstrata}
We recall that by \cite{SU84}, for maps into $\S^3$ the singular set has dimension less or equal $(n-4)$.

 There exist $\varphi_k\colon \S^3 \to \S^3$ homotopically nontrivial ($\deg(\varphi_k)=1$) such that for any $p<3$ 
 \[
  \int_{\S^3}|\nabla \varphi_k|^p \dd\mathcal{H}^{3} \rightarrow 0,
 \]
implying $\mathcal{H}^{0}(\text{sing}\, u_k)\ge 1$, where $u_k\colon B^4_1\rightarrow \S^3$ are minimizers corresponding to the boundary data $\vp_k$, i.e.,  $u_k\big\rvert_{\S^3} = \vp_k$.

This in particular, implies that the linear law with $W^{1,2}$ cannot be true for the stratum $\mathcal S _{n-4}$ and singularities are not stable under $W^{1,p}$ for $p<3$ perturbations of the boundary.
\end{example}

Motivated by this example we conjecture the following.
\begin{conjecture} Let $\Omega \subset \R^n$, $\n$ be such that $\pi_1(\n)$ is finite and let $u\in W^{1,2}(\Omega,\n)$ be a minimizing harmonic map with $u\big\rvert_{\partial \Omega} = \vp$. Then for each $k=3,\ldots,n$ we have
\[
 \H^{n-k}(\mathcal S_{n-k}) \le C \int_{\partial \Omega} |\nabla \vp|^{k-1} \dhn.
\]
\end{conjecture}
A starting point to study this problem would be to develop a theory for each strata $\mathcal S_{n-k}$ similar to the one for $\mathcal S_{n-3}=\sing u $ by \cite{NabVal17}, in particular a counterpart of \Cref{th:NabVal-general-meas-bound}. Note that even for $k=3$ our conjecture suggests an improvement from $W^{1,n-1}$ to $W^{1,2}$-control of the boundary data.

\subsection{Remarks about other target manifolds}\label{ss:othertargetmanifolds}
In the case when the target manifold is an orientable surface ($\dim \n=2$), since our results cover the case when $\pi_1(\n)$ is finite, we are left with the case when genus of the target manifold is positive ($\rm{genus}(\n)\ge 1$). In this situation, we have as a~consequence of a~maximum principle --- the result of Wood \cite{Wood1977,Wood-erratum}.
\begin{theorem}[{\cite[Theorem 3.3 (ii)]{Wood1977} \& Erratum \cite{Wood-erratum}}]
Let $\mathcal M$, $\n$ be compact orientable surfaces and assume that $\rm{genus}(\mathcal M)=0$ and $\rm{genus}(\n) >0$. Then the only harmonic maps $w\colon \mathcal M \to \n$ are constant maps. 
\end{theorem}

One can see that this result implies $\sing u = \mathcal{S}_{n-4}$ and hence $\cH^{n-3}(\sing u)=0$ in the case where $u$ is a minimizing harmonic map into a compact orientable surface $\cN$ of genus $\ge 1$. Indeed, at each point in the top-dimensional part of the singular set $\sing_* u = \mathcal{S}_{n-3} \setminus \mathcal{S}_{n-4}$ there is a nonconstant $(n-3)$-symmetric tangent map $w \colon \R^n \to \cN$. Due to its symmetries, such a map reduces to a harmonic map $w \colon \S^2 \to \cN$. The genus of $\S^2$ is $0$ and hence $w$ cannot be constant by the above theorem. 

Thus, the linear law of Almgren and Lieb \Cref{th:almgren-lieb-high} holds for every compact connected orientable surface in the target.

We note here also two special cases of manifolds with infinite fundamental group: $\n = \S^1$ and $\n = \mathbb{T}^2$. In the first case if we consider harmonic maps (not necessarily \emph{minimizing}) $u\in W^{1,2}(B,\S^1)$, where $B\subset\R^n$ is a simply connected domain, with given boundary data $u\big\rvert_{\partial B}=\phi\in W^{\frac 12,2}(\partial B, \S^1)$, then we have by a lifting argument (see \cite{BethuelZheng, BourgainBrezisMironescu}) $u=e^{i\tilde u}$ with $\tilde u \in W^{1,2}(B,\R)$ and one can easily check that $\tilde u$ is a solution to 
\[
 \left\{\begin{array}{rll}
  -\Delta \tilde u &= 0 & \text{ in } B\\
  \tilde u &= e^{i \tilde \phi} & \text{ on } \partial B,
 \end{array}\right.
\]
where $\tilde \phi$ is a lifting of $\phi$, i.e., $\phi = e^{i\tilde \phi}$. Thus, $u$ has no singularities.

Moreover, Rivi\`{e}re showed that if we consider a harmonic map $u\colon B \subset \R^n \to \mathbb{T}^2$, where $\mathbb T^2 = \S^1 \times \S^1$ is a torus of revolution (i.e., with metric of the form $\lambda(\phi) \dd \theta^2 + \dd\phi^2$), then $u$ must be smooth ($u$ does not have to be \emph{minimizing}). See \cite{RivierePhD}.

Thus, in the case when $\mathcal N = \S^1$ or $\mathcal N = \mathbb T^2$ not only \emph{minimizing} harmonic maps have no singularities but \emph{all weakly} harmonic maps are regular.

\appendix
\section{Trace theorems}
\subsection{A trace theorem}\label{a:traces}
In this section we review the trace theorems used throughout the paper. Here we present the results for domains in $\R^n$ for $n\ge 3$. The main point is to set the trace separately for two parts of the boundary and obtain estimates without the interaction term.

We will employ the following notation. For a point $y_0\in\partial \Omega$ on the boundary we will consider intersection of balls $B_r(y_0)$ with the domain $\Omega$. We distinguish two parts of the boundary of $\partial (B_r(y_0)\cap \Omega)$ and write:
\[
\partial (B_r(y_0)\cap \Omega) = \partial \Omega_r^+(y_0) \cup T_{\Omega_r(y_0)},
\]
where $\partial \Omega_r^+(y_0)\coloneqq \partial B_r(y_0) \cap \Omega$ and $ T_{\Omega_r(y_0)}\coloneqq B_r(y_0) \cap \partial \Omega$.

If the center of the ball will play no role we will omit $y_0$ in the above notation and write simply $\partial \Omega_r^+$ and $T_{\Omega_r}$.
\begin{lemma}\label{la:t1}
	Let $\Omega \subset \R^n$ be a $C^1$ domain. Then, there exists an $R=R(\Omega)$ such that for every $y_0\in\partial \Omega$ and $r<R$ the following is true: 
	
	Let $u \in W^{1,2}(B_r(y_0)\cap \Omega)$ be a solution to
	\[
	\left\{
	\begin{array}{rcll}
	\lap u & = & 0&\mbox{in $B_r(y_0)\cap \Omega$}\\
	u\, & = & 0 &\mbox{on $\partial \Omega_r^+(y_0) = \partial B_r(y_0)\cap \Omega$}\\
	u\, & = & \psi &\mbox{on $T_{\Omega_r(y_0)}=B_r(y_0)\cap \partial \Omega$}.
	\end{array}
	\right.
	\]
Then for any $s$ satisfying $s>\frac{1}{2}$ we have
\begin{equation}\label{eq:harmonicbyflatpartwith2}
 \|\nabla u\|_{L^2(B_r(y_0)\cap \Omega)} \aleq r^{\frac{-1+2s}{2}}[\psi]_{W^{s,2}(T_{\Omega_r(y_0)})}.
\end{equation}
\end{lemma}

\begin{proof}
	We begin by noting that since the $\Omega$ is of class $C^1$ there exists an $R=R(\Omega)$ for which we have for every $y_0\in\partial \Omega$ and $r<R$
	\begin{equation}\label{eq:notflatboundarybutalmost}
	|x-y| \ge C \,\dist(x, \partial T_{\Omega_r(y_0)}), \quad \text{ for } x\in T_{\Omega_r(y_0)},\ y\in \partial \Omega_r^+(y_0),
	\end{equation}
	where $C$ is a constant independent of $\Omega$. In what follows we will omit $y_0$ in the notation.

\begin{center}
	\begin{tikzpicture}
	\draw plot [smooth, tension=1] coordinates { (-3,0.5) (-2,0.25) (-1,0) (0,0) (1,0)  (3,1)};
	\draw (0,0) circle (1cm);
	\begin{scope}
	\clip (0,0) circle (1cm);
	\clip plot [smooth, tension=1] coordinates { (-3,0.5) (-2,0.25) (-1,0) (0,0) (1,0)  (3,1) (0,2) (-3,0.5)};
	\fill[gray, opacity=0.1] (0,0) circle (1cm);
	\draw[blue!70, thick] plot [smooth, tension=1] coordinates { (-3,0.5) (-2,0.25) (-1,0) (0,0) (1,0)  (3,1)}; 
	\end{scope}
	\draw[black!40, thick] (1,0) arc (0:180:1);
	\node[scale=0.75, below, blue!70] at (0.95,-0.1) {$ x\in T_{\Omega_r}$};
	\node[scale=0.75, below, gray] at (1.5,1) {$y\in \partial \Omega^+_r$};
	\draw[fill=black] (0,0) circle (1pt);
	\draw[fill=black] (-1,0) circle (1pt);
	\draw[fill=black] (1,0) circle (1pt);
	\node[scale=0.75, below] at (0,0) {$y_0$};
	\end{tikzpicture}
\end{center}

By the trace theorem \cite{Gagliardo} (see also \cite[\textsection 20]{BN78} or \cite[Section 12.2]{mironescu:cel-00747696}) and Poincar\'{e} inequality we have
\[
 \int_{B_r\cap \Omega} |\nabla u|^2 \dx \aleq [u]_{W^{\frac{1}{2},2}(\partial (B_r\cap \Omega))}^2 \aleq [u]^2_{W^{s,2}(\partial (B_r\cap \Omega))}.
 \]
Moreover, since $u = 0$ on $\partial \Omega_r^+$ we have 
\[
 [u]_{W^{s,2}(\partial (B_r\cap \Omega))}^2 = [\psi]_{W^{s,2}(T_{\Omega_r})}^2 + 2\int_{T_{\Omega_r}} \int_{\partial \Omega_r^+} \frac{|\psi(x)|^2}{|x-y|^{n-1+2s}}\dd y\dd x.
\]
For the latter term we have using \eqref{eq:notflatboundarybutalmost}
\[
\begin{split}
 \int_{T_{\Omega_r}} \int_{\partial \Omega_r^+}& \frac{|\psi(x)|^2}{|x-y|^{n-1+2s}}\dd y\dd x\\  
 &\aleq \int_{T_{\Omega_r}} |\psi(x)|^2 \int_{\partial \Omega_r^+} \frac{1}{|x-y|^{n-1+2s}} \dd y \dd x\\
 &\aleq \int_{T_{\Omega_r}} |\psi(x)|^2 \int_{|z| \ge \dist (x,\partial T_{\Omega_r})} \frac{1}{|z|^{n-1+2s}} \dd y \dd x\\
 &\aleq \int_{T_{\Omega_r}} \frac{|\psi(x)|^2}{\dist (x,\partial T_{\Omega_r})^{2s}}\dd x.
 \end{split}
\]
Now, since $2s > 1$ we can apply Hardy's inequality \cite{Dyda-2004}. Observe that since $\chi_{T_{\Omega_r}} \psi$ is the trace of a $W^{1,2}$ function, it can be approximated by functions in $C_c^\infty$ (simply by scaling the support inside and convolution). Thus
\[
 \int_{T_{\Omega_r}} \frac{|\psi(x)|^2}{\dist(x,\partial T_{\Omega_r})^{s2}}\dd x 
 \aleq \int_{T_{\Omega_r}} \int_{T_{\Omega_r}} \frac{|\psi(x)-\psi(y)|^2}{|x-y|^{n-1+s2}}\dd x\dd y = [\psi]_{W^{s,2}(T_{\Omega_r})}^2.
\]
This proves \eqref{eq:harmonicbyflatpartwith2}.

\end{proof}

As a consequence we can obtain a trace inequality, which depends only on the behavior of the boundary map of the curved part of the boundary $\partial B_r(y_0)\cap \Omega$ and the "flat" part $B_r(y_0)\cap \partial \Omega$ but not on the interaction term.

\begin{lemma}\label{la:t2}
Let $\Omega \subset \R^n$ be a $C^1$ domain. Then, there exists an $R=R(\Omega)$ such that for every $y_0\in\partial \Omega$ and $r<R$ the following is true: 

Let $u \in W^{1,2}(B_r(y_0)\cap \Omega)$ be a solution to
\[
 \left\{
\begin{array}{rcll}
\lap u & = & 0&\mbox{in $B_r(y_0)\cap \Omega$}\\
u\, & = & \varphi &\mbox{on $\partial \Omega_r^+(y_0)$}\\
u\, & = & \psi &\mbox{on $T_{\Omega_r(y_0)}$}.
\end{array}
\right.
\]
Then for any $s > \frac{1}{2}$ we have
\[
 \|\nabla u\|_{L^2(B_r(y_0)\cap \Omega)} \aleq r^{\frac{-1+2s}{2}}\brac{[\psi]_{W^{s,2}(T_{\Omega_r(y_0)})} + [\varphi]_{W^{s,2}(\partial \Omega_r^+(y_0))}}.
\]
\end{lemma}
\begin{proof}
Again, it what follows we will omit $y_0$.

First we note that we can extend $\varphi$ to all of $\partial (B_r\cap \Omega)$ with\footnote{For $\S^{n-1}_+$ we would proceed in the following way --- we extend by an even reflection the map $\varphi$ into the whole sphere $\S^{n-1}$, this way the seminorm on the interaction term may be estimated by $[\varphi]_{W^{s,2}(\S^{n-1}_+)}$, then we may project $\S^{n-1}_-$ into $T_1$ and we obtain the desired estimate. For $\partial (B_r\cap \Omega)$ we proceed similarly.}
\begin{equation}\label{eq:appreflection}
 [\varphi]_{W^{s,2}(\partial (B_r\cap \Omega))} \aleq [\varphi]_{W^{s,2}( \partial \Omega_r^+)}.
\end{equation}
Now we solve the equation
\[
 \left\{
\begin{array}{rcll}
\lap v & = & 0&\mbox{in $B_r\cap \Omega$}\\
v\, & = & \varphi &\mbox{on $\partial (B_r\cap \Omega)$.}
\end{array}
\right.
\]
Then, again by the classical trace inequality, Poincar\'e inequality, and \eqref{eq:appreflection} we have
\begin{equation}\label{eq:tr:21}
 \|\nabla v\|_{L^2(B_r\cap \Omega)} \aleq [\varphi]_{W^{s,2}( \partial \Omega_r^+)}.
\end{equation}
On the other hand we have
\[
\left\{
\begin{array}{rcll}
\lap (u-v) & = & 0&\mbox{in $B_r\cap \Omega$}\\
u-v & = & 0 &\mbox{on $\partial \Omega_r^+$}\\
 u-v &=& \psi-\varphi &\mbox{on $T_{\Omega_r}$.}
\end{array}
\right.
\]
By Lemma~\ref{la:t1},
\begin{equation}\label{eq:tr:22}
\begin{split}
 \|\nabla (u-v)\|_{L^2(B_r\cap \Omega)} &\aleq [\psi-\varphi]_{W^{s,2}(T_{\Omega_r})}\\
 &\aleq [\psi]_{W^{s,2}(T_{\Omega_r})} + [\varphi]_{W^{s,2}(T_{\Omega_r})}\\
 &\aleq [\psi]_{W^{s,2}(T_{\Omega_r})} + [\varphi]_{W^{s,2}(\partial \Omega_r^+)}.
\end{split}
 \end{equation}
Together, \eqref{eq:tr:21} and \eqref{eq:tr:22} imply the claim.
\end{proof}

We also need the following Gagliardo--Nirenberg type inequality
\begin{lemma}\label{le:appGN}
Let $\Gamma\subset \R^{n-1}$ be a compact domain. For every $\varphi \in W^{1,2}\cap L^\infty(\Gamma)$ the following interpolation inequality holds for a constant independent of $\varphi$:
\begin{equation}\label{eq:Sntracesplit}
 [\varphi]_{W^{\frac12,2}(\Gamma)}^2\aleq \|\varphi\|_{L^\infty(\Gamma)}\, \|\nabla \varphi\|_{L^2(\Gamma)}.
\end{equation}
\end{lemma}

\begin{proof}
We have by the Gagliardo--Nirenberg inequality, \cite{BrezisMironescu-gagliardo}
\[
[\varphi]_{W^{\frac12,2}(\Gamma)} \aleq \|\varphi\|^\frac12_{L^2(\Gamma)} \|\nabla \varphi\|^\frac12_{L^2(\Gamma)}.
\]
Since $\Gamma$ is compact, we also have 
\[
\|\varphi\|_{L^2(\Gamma)} \aleq \|\varphi\|_{L^\infty(\Gamma)}.
\]
\end{proof}

From the above lemmata we obtain the following trace estimates.
\begin{theorem}[Trace Theorem]\label{th:trace}
Let $B_r \subset \R^n$, $n \geq 3$, be a ball of radius $r > 0$ and $\varphi^h\colon \B_r \to \R^d$ be the harmonic extension of $\varphi\colon \partial \B_r \to \R^d$, then
\begin{equation}\label{eq:tracespherer}
\int_{\B_r} |\nabla \varphi^h|^2 \aleq \int_{\partial B_r}\int_{\partial B_r} \frac{|\varphi(x)-\varphi(y)|^2}{|x-y|^{n}} \dd x\dd y
\end{equation}
and 
\begin{equation}\label{eq:tracespherercombined}
\int_{\B_r} |\nabla \varphi^h|^2 \aleq r^{\frac{n-1}{2}} \|\varphi\|_{L^\infty(\partial B_r)}\, \|\nabla \varphi\|_{L^2(\partial B_r)}.
\end{equation}

Let $\Omega\subset\R^n$ be a bounded domain with a $C^1$-boundary and let $y_0\in \partial \Omega$. If $\varphi^h\colon B_r(y_0)\cap \Omega \to \R^d$ is the harmonic extension of $\varphi\colon \partial (B_r(y_0)\cap \Omega) \to \R^d$, then
\begin{equation}\label{eq:tracehalfsphere}
\int_{B_r(y_0)\cap \Omega} |\nabla \varphi^h|^2 \aleq \int_{\partial (B_r(y_0)\cap \Omega)}\int_{\partial (B_r(y_0)\cap \Omega)} \frac{|\varphi(x)-\varphi(y)|^2}{|x-y|^{n}} \dd x \dd y.
\end{equation}
Moreover, for any $1<\theta< 2$, $s>\frac12$, $p>1$, $sp>1$, we have
\begin{equation}\label{eq:tracehalfspherecombined}
\begin{split}
\int_{B_r(y_0)\cap \Omega} |\nabla \varphi^h|^2 
&\aleq r^{\frac{(3-n)\theta}{2}+n-2}\, \|\varphi\|^{2-\theta}_{L^\infty(\partial \Omega_r^+(y_0))}\, \|\nabla \varphi\|^{\theta}_{L^2(\partial \Omega_r^+(y_0))}\\
&\quad + r^{(sp-(n-1))\frac{\theta}{sp}+n-2} \|\varphi\|_{L^\infty(T_{\Omega_r(y_0)})}^{2-\frac{\theta}{s}}[\varphi]_{W^{s,p}(T_{\Omega_r(y_0)})}^{\frac{\theta}{s}}.
\end{split}
\end{equation}
\end{theorem}

\begin{proof}
\eqref{eq:tracespherer} and \eqref{eq:tracehalfsphere} are classical trace inequalities (see \cite{Gagliardo}). The inequality \eqref{eq:tracespherercombined} is a~consequence of \Cref{le:appGN} and \eqref{eq:tracespherer}. 

For the clarity of the presentation we write the proof of \eqref{eq:tracehalfspherecombined} for the case $B_r(y_0)\cap \Omega= B_1^+$. Applying \Cref{la:t2} we get  for any $1 < \theta <2$
\begin{equation}\label{eq:GNsplit1}
\int_{B_1^+} |\nabla \varphi^h|^2 \aleq [\varphi]_{W^{\frac\theta2,2}(S^+_1)}^2 + [\varphi]_{W^{\frac\theta2,2}(T_1)}^2
\end{equation}
By Gagliardo--Nirenberg inequality we have
\[
[\varphi]_{W^{\frac\theta2,2}(S_1^+)}^2 \aleq \|\varphi\|^{2-\theta}_{L^2(S_1^+)}\|\nabla \varphi\|^{\theta}_{L^2(S_1^+)} \aleq \|\varphi\|_{L^\infty(S_1^+)}^{2-\theta}\|\nabla \varphi\|^{\theta}_{L^2(S_1^+)}.
\]
Applying Gagliardo--Nirenberg inequality for an $s_0 > \frac\theta2$ to the second term of \eqref{eq:GNsplit1} we obtain
\[
 [\varphi]_{W^{\frac\theta2,2}(T_1)}^2 \aleq \|\varphi\|_{L^\infty(T_1)}^{2-\frac{\theta}{s_0}}[\varphi]^{\frac{\theta}{s_0}}_{W^{s_0,\frac{\theta}{s_0}}(T_1)}.
\]
Applying once again Gagliardo--Nirenberg inequality for any $s> s_0$ and any $p>1$ we get
\[
[\varphi]^{\frac{\theta}{s_0}}_{W^{s_0,\frac{\theta}{s_0}}(T_1)} \aleq \|\vp\|_{L^{p_1}(T_1)}^{\frac{\theta}{s_0}\brac{1-\frac{s_0}{s}}}[\vp]_{W^{s,p}(T_1)}^{\frac{\theta}{s}} \aleq \|\vp\|_{L^{\infty}(T_1)}^{\frac{\theta}{s_0}\brac{1-\frac{s_0}{s}}}[\vp]_{W^{s,p}(T_1)}^{\frac{\theta}{s}},
\]
where $\frac{s_0}{\theta} = \frac{1-\frac{s_0}{s}}{p_1} + \frac{s_0}{sp}$.

Combining the last two inequalities gives
\[
 [\varphi]_{W^{\frac\theta2,2}(T_1)}^2 \aleq \|\vp\|^{2-\frac{\theta}{s}}_{L^\infty(T_1)}[\vp]^{\frac{\theta}{s}}_{W^{s,p}(T_1)}.
\]
Thus,
\[
\int_{B_1^+} |\nabla \varphi^h|^2 \aleq \|\varphi\|_{L^\infty(S_1^+)}^{2-\theta}\|\nabla \varphi\|^{\theta}_{L^2(S_1^+)} + \|\varphi\|_{L^\infty(T_1)}^{2-\frac{\theta}{s}}[\varphi]_{W^{s,p}(T_1)}^{\frac{\theta}{s}}. 
\]
The general statement follows from rescaling.
\end{proof}
%
\subsection{A Counterexample}
The above trace theorem, part \eqref{eq:tracehalfspherecombined}, does not hold with $W^{\frac{1}{2},2}$. 
Indeed, this follows essentially from a counterexample to Hardy--Sobolev inequality on bounded domains for $W^{\frac{1}{2},2}$ by Dyda  \cite{Dyda-2004} (attributed to an idea by Bogdan). For an overview on available Hardy--Sobolev inequalities see also \cite{BC18}.
\begin{lemma}
There does \emph{not} exist a constant $C>0$ such that the following holds.

Assume $u \in W^{1,2}(B_1^+)$ is a harmonic function in $B_1^+ \subset \R^n$, $n\ge 3$ with
\[
 \left\{
 \begin{array}{rcll}
  \Delta u &=& 0 & \text{ in } B_1^+\\
  u&=&\vp & \text{ on } T_2\\
  u&=& \psi & \text{ on } S_1^+,
 \end{array}
 \right.
\]
where $\varphi \in W^{\frac{1}{2},2}(T_2)$ and $\psi \in W^{1,2}(S_1^+)$. Then 
\begin{equation}\label{eq:splittracefalse}
\|\nabla u\|_{L^2(B_1^+)}^2 \leq C\ \brac{\int_{T_2} \int_{T_2} \frac{|\varphi(x)-\varphi(y)|^2}{|x-y|^{n}}\dd x\dd y + \int_{S_1^+} |\nabla_T \psi|^2}.
\end{equation}
\end{lemma}
\begin{proof}
\hfill

\textsc{Step 1.} There exist two sequences of functions $\{\vp_i^1\}_{i=1}^\infty$ and $\{\vp_i^2\}_{i=1}^\infty$ with the following properties: for each $i=1,2,\ldots$ we have $\varphi^1_{i}, \varphi_{i}^2\colon T_2 \to \R$, $\supp \varphi_{i}^1 \subset T_{1-\frac{1}{i}}$, $\varphi^1_i = \varphi^2_i$ in $T_1$, 
\[
\lim_{i\to \infty}\int_{T_2} \int_{T_2} \frac{|\varphi^1_i(x)-\varphi^1_i(y)|^2}{|x-y|^{n}}\dd x\dd y = \infty, 
\]
and
\[
 \limsup_{i \to \infty} \int_{T_2} \int_{T_2} \frac{|\varphi^2_i(x)-\varphi^2_i(y)|^2}{|x-y|^{n}}\dd x\dd y  < \infty.
\]
Indeed, the sequence $\{\varphi_i^1\}_{i=1}^\infty$ is constructed by Dyda in \cite{Dyda-2004}. More precisely, he obtains a~sequence of smooth functions $\varphi^1_i \in C_c^\infty(T_1)$ such that
\[
\lim_{i\to\infty}\int_{T_1} \int_{T_1} \frac{|\varphi^1_i(x)-\varphi^1_i(y)|^2}{|x-y|^{n}}\dd x\dd y = 0,
\]
but
\[
\lim_{i\to\infty}\int_{T_2} \int_{T_2} \frac{|\varphi^1_i(x)-\varphi^1_i(y)|^2}{|x-y|^{n}}\dd x\dd y =\infty.
\]
On the other hand $T_1$ is an extension domain, see \cite{Z15}, so there exists an extension $\varphi_i^2$ of $\varphi_i^1 \Big |_{T_1}$ such that 
\[
\limsup_{i \to \infty} \int_{T_2} \int_{T_2} \frac{|\varphi^2_i(x)-\varphi^2_i(y)|^2}{|x-y|^{n}}\dd x\dd y <\infty.
\]

\textsc{Step 2.} Now consider the solution $u^i \in C^\infty(\overline{B_1^+})$ to
\[
\left\{
\begin{array}{rcll}
\lap u^i &=& 0 \quad &\mbox{in $B_1^+$}\\
u^i &=& 0 \quad &\mbox{on $S_1^+$}\\
u^i &=& \varphi_i^1 \quad &\mbox{on $T_1$}.\\
\end{array}
\right.
\]
By Gagliardo's trace theorem \cite{Gagliardo}(which was originally proved for Lipschitz domains), we have
\begin{equation}\label{eq:tracetheoremvpi1}
 \|\nabla u_i\|_{L^2(B_1^+)}^2 \aeq \int_{\partial B_1^+}\int_{\partial B_1^+} \frac{\abs{\chi_{T_1}(x) \varphi_i^1(x)-\chi_{T_1}(y) \varphi_i^1(y)}^2}{|x-y|^n} \dx \dy.
\end{equation}
Let now $N \in \partial B_1$ be the north pole. There is a bi-lipschitz map $\tau\colon T_2 \to \partial B_1^+ \setminus B_{\frac{1}{100}}(N)$ which is the identity on $T_1$. Then we have 
\begin{equation}\label{eq:bilipschitzchangeofvariables}
\begin{split}
 \int_{\partial B_1^+}\int_{\partial B_1^+}& \frac{\abs{\chi_{T_1}(x) \varphi_i^1(x)-\chi_{T_1}(y) \varphi_i^1(y)}^2}{|x-y|^n} \dx \dy\\
 &\geq\int_{\tau(T_2)}\int_{\tau(T_2)} \frac{\abs{\chi_{T_1}(x) \varphi_i^1(x)-\chi_{T_1}(y) \varphi_i^1(y)}^2}{|x-y|^n} \dx \dy\\
 &\ageq\int_{T_2}\int_{T_2} \frac{\abs{\chi_{T_1}(\tau(x)) \varphi_i^1\circ \tau(x)-\chi_{T_1}(\tau(y)) \varphi_i^1\circ \tau(y)}^2}{|x-y|^n} \dx \dy\\
 &= \int_{T_2}\int_{T_2} \frac{\abs{\varphi_i^1(x)-\varphi_i^1(y)}^2}{|x-y|^n} \dx \dy.
 \end{split}
\end{equation}
Here we used the change of variables formula which holds for bi-Lipschitz maps, see e.g., \cite[\textsection 3.3.3, Theorem 2]{EG15}.
Thus, combining \eqref{eq:tracetheoremvpi1} with \eqref{eq:bilipschitzchangeofvariables} we get
\[
 \lim_{i \to \infty} \|\nabla u_i\|_{L^2(B_1^+)}^2 =\infty.
\]
On the other hand, since $\varphi_i^1 = \varphi_i^2$ on $T_1$, we have
\[
\left\{ \begin{array}{rcll}
\lap u^i &=& 0 \quad &\mbox{in $B_1^+$}\\
u^i &=& 0 \quad &\mbox{on $S_1^+$}\\
u^i &=& \varphi_i^2 \quad &\mbox{on $T_1$}.\\
\end{array}
\right.
\]
Therefore, if \eqref{eq:splittracefalse} was true, we would obtain 
\[
 \limsup_{i \to \infty}\|\nabla u_i\|_{L^2(B_1^+)}^2 \aleq \limsup_{i \to \infty}\int_{T_2}\int_{T_2} \frac{\abs{\varphi_i^2(x)-\varphi_i^2(y)}^2}{|x-y|^n} \dx \dy  < \infty,
\]
a contradiction.
\end{proof}

\bibliographystyle{abbrv}%
\bibliography{bib}%

\end{document}